%% file: main.tex
\documentclass[leqno]{siamart1116}

\usepackage[T1]{fontenc}
\usepackage[utf8]{inputenc}
\usepackage[american]{babel}

\usepackage{color}
\usepackage{amssymb}
\usepackage{graphicx}
\usepackage{here}
\usepackage{xspace}
\usepackage{multirow}
\usepackage{bm,bbm}
\usepackage[algo2e,vlined]{algorithm2e}
\usepackage{algpseudocode}
\usepackage[numbers,sort]{natbib}
\let\cite=\citet
\usepackage{subfig}
\usepackage{enumerate}
\usepackage{dsfont}
\usepackage{adjustbox}
\usepackage[normalem]{ulem}
\usepackage[autolanguage]{numprint}
\usepackage{booktabs}
\usepackage[figuresleft]{rotating}
\usepackage{geometry}
\usepackage{marginnote}
\usepackage{afterpage}
\usepackage{tikz}
\usetikzlibrary{arrows.meta}
\usepackage{gnuplot-lua-tikz}

\usepackage{pgfplots}
\usepackage{pgfplotstable}
\pgfplotsset{compat=1.9}

\usepackage{enumitem}
\setlist{leftmargin=20pt}

\usepackage{mydef}

\begin{document}
\newcommand\footnotemarkfromtitle[1]{%
\renewcommand{\thefootnote}{\fnsymbol{footnote}}%
\footnotemark[#1]%
\renewcommand{\thefootnote}{\arabic{footnote}}}

\title{On the implementation of a robust and efficient \\
finite element-based parallel solver for  the \\
  compressible Navier--Stokes equations}

\author{Jean-Luc~Guermond\footnotemark[2]
  \and Martin Kronbichler\footnotemark[3] \footnotemark[4]
  \and Matthias~Maier\footnotemark[2]
  \and Bojan~Popov\footnotemark[2]
  \and Ignacio~Tomas\footnotemark[5]}

\date{Draft version \today}

\maketitle

\renewcommand{\thefootnote}{\fnsymbol{footnote}}

\footnotetext[2]{%
  Department of Mathematics, Texas A\&M University 3368 TAMU, College
  Station, TX 77843, USA.}%

\footnotetext[3]{%
  Institute for Computational Mechanics, Department for Mechanical
  Engineering, Technical University of Munich, Germany.}

\footnotetext[4]{%
  Division of Scientific Computing, Department of Information Technology,
  Uppsala University, Sweden.}

\footnotetext[5]{%
  Sandia National Laboratories$^{\S}$, P.O. Box 5800, MS 1320, Albuquerque,
  NM 87185-1320.}
\footnotetext[6]{%
  Sandia National Laboratories is a multimission laboratory managed and
  operated by National Technology \& Engineering Solutions of Sandia, LLC,
  a wholly owned subsidiary of Honeywell International Inc., for the U.S.
  Department of Energy's National Nuclear Security Administration under
  contract DE-NA0003525. This document describes objective technical
  results and analysis. Any subjective views or opinions that might be
  expressed in the paper do not necessarily represent the views of the U.S.
  Department of Energy or the United States Government.}

\renewcommand{\thefootnote}{\arabic{footnote}}

\begin{abstract}
  This paper describes in detail the implementation of a finite
  element technique for solving the compressible Navier-Stokes
  equations that is provably robust and demonstrates excellent
  performance on modern computer hardware.  The method is second-order
  accurate in time and space.  Robustness here means that the
    method is proved to be invariant domain preserving under the
    hyperbolic CFL time step restriction, and the method delivers
    results that are reproducible. The proposed technique is shown to
  be accurate on challenging 2D and 3D realistic benchmarks.
\end{abstract}



\pagestyle{myheadings} \thispagestyle{plain}
\markboth{}{Invariant domain approximation of the compressible Navier--Stokes equations}

%%%%%%%%%%%%%%%%%%%%%%%%%%%%%%%%%%%%%%%%%%%%%%%%%%%%%%%%%%%%%%%%%%%%%%%%%%%%%%%%
%%%%%%%%%%%%%%%%%%%%%%%%%%%%%%%%%%%%%%%%%%%%%%%%%%%%%%%%%%%%%%%%%%%%%%%%%%%%%%%%
%%%%%%%%%%%%%%%%%%%%%%%%%%%%%%%%%%%%%%%%%%%%%%%%%%%%%%%%%%%%%%%%%%%%%%%%%%%%%%%%

\section{Introduction}
The objective of the paper is to describe in detail a robust and
efficient massively parallel finite element technique for solving the
compressible Navier-Stokes equations. This paper is the second part of
a research project described in
\cite{Guermond_Maier_Popov_Tomas_CMAME_2020}. The principles of the
method have been introduced in
\citep{Guermond_Maier_Popov_Tomas_CMAME_2020}, but in order to
guarantee reproducibility (following the guidelines described in
\cite{Leveque_Mitchell_Stodden_2012}), we here describe the
implementation details regarding the algorithm \emph{per se} and
identify the ingredients that enable efficient execution on large-scale
parallel machines.  We also explain the implementation of
non-reflecting boundary conditions and show that these conditions are
robust, invariant-domain preserving, and accurate.  Based on these
ingredients, the accuracy of the method is demonstrated on
  well-documented (\ie reproducible)  non-trivial benchmarks.  The robustness of the method and its
  capability to scale well on large parallel architectures are also
  demonstrated.

As there is currently a regain of interest for supersonic and hypersonic flight
of aircrafts and other devices, there is also a renewed interest for provably
robust numerical methods that can solve the compressible Navier-Stokes
equations. Here we say that a numerical method is provably robust if it can be
unambiguously proved to be invariant domain preserving, \ie among other things,
it ensures positivity of the density, positivity of the internal energy, and
preserves a meaningful entropy-dissipation property. There are many papers in
the literature addressing this question, but invariant domain properties are
available only for very few methods. One notable result in this direction can be
found in \cite{Grapsas_Herbin_Kheriji_Latche_2016} where a first-order staggered
approximation using velocity-based upwinding is developed (see Eq.~(3.1)
therein), and positivity of the density and the internal energy is established
(Lem.~4.4 therein). Unconditional stability is obtained by using an implicit
time-stepping coupling the mass conservation equation and the internal energy
equation. This method is robust, including in the low Mach regime. A
similar technique solving the compressible barotropic Navier-Stokes
equation is proposed in
\cite[\S 3.6]{Gallouet_Gastaldo_Herbin_Latche_2008}.  In the
discontinuous Galerkin literature, robustness is established in
\cite{Zhang_JCP_2017} for the approximation of the compressible Navier-Stokes equations.
The time stepping is explicit though, and this entails a parabolic restriction
on the time step that unfortunately makes the method ill-suited for realistic
large-scale applications (\ie $\dt \lesssim \calO(h^2)/\mu$, where $\mu$ is
some reference viscosity scale, $\dt$ is the time step size and $h$ is the mesh
size).

The approximation technique described in the present paper draws
robustness from an operator-splitting strategy that uncouples the
hyperbolic and the parabolic phenomena. We do not claim originality
for this ``divide and conquer'' strategy since the operator-splitting
idea has been successfully used in the CFD literature numerous times
in the past.  Among the references that inspired the present work in
one way or another, we refer the reader to
\cite{Beam_warming_AIAA_1977}, \cite{Bristeau_Glowinski_Periaux_1987},
\cite{Demkowicz_etal_1990}. The key novelties of the paper are as
  follows: (i) The paper describes an exhaustive and unambiguous
  (thereby reproducible) robust algorithm for solving the compressible
  Navier-Stokes using finite elements. Algorithm~\ref{alg:euler} gives
  a flow chart that minimizes the complexity of the hyperbolic step;
  (ii) The implementation of various boundary condition is fully and
  unambiguously described. In particular, non-reflecting boundary
  conditions are discussed. Unambiguous, fully discrete,
  finite-element based algorithms are proposed. These boundary
  conditions are explicit and are proved to be invariant-domain
  preserving and to maintain conservation; (iii) The parabolic substep
  of the algorithm is also fully described and important details
  regarding its matrix-free implementation are given; (iv) The
  algorithm is verified against analytical solutions and validated
  against two challenging benchmarks (one is two-dimensional, the
  other is three-dimensional). In particular, we provide a reference
  solution for the benchmark proposed in
  \cite{Daru_Tenaud_2000,Daru_Tenaud_2009} with an accuracy that has
  never been matched before (see Table~\ref{tab:shocktube_extrema} and
  Figure~\ref{fig:shocktube-detailed}).

The paper is organized as follows. The problem along with the finite
element setting and the principles of the time stepping that are used
for the approximation is described in \S\ref{sec:preliminaries}.  As
the time stepping is based on Strang's splitting using a hyperbolic
substep and a parabolic substep, we describe in
\S\ref{Sec:hyperbolic_limit} the full approximation of the hyperbolic
step. All the details that are necessary to guarantee
  reproducibility are given.  Key results regarding admissibility and
conservation after limiting are collected in
Lemma~\ref{Lem:boundary_flux}.  Important details regarding the
treatment of boundary conditions for the hyperbolic step are reported
in \S\ref{Sec:boundary_conditions}.  Key original results
regarding admissibility and conservation after boundary postprocessing
are collected in Lemma~\ref{Lem:slip_bc},
Lemma~\ref{Lem:global_conservation}, and
Corollary~\ref{Cor:admissibility_after_postproc_io_bcs}.  The full
approximation of the parabolic substep is described in
\S\ref{Sec:discretization_hyperbolic_limit}. Here again, all the details that are necessary to guarantee
  reproducibility are given.  The key results of this
section regarding admissibility and conservation are stated in
Lemma~\ref{Lemma:parabolic_conservation}. The method has been
implemented using the finite element library \texttt{deal.II}
\citep{dealII92,dealIIcanonical} and mapped continuous $\polQ_1$
finite elements. Our implementation is freely available
online\footnote{\url{https://github.com/conservation-laws/ryujin}}
\citep{maier2021ryujin} under a permissible open source
license.\footnote{\url{https://spdx.org/licenses/MIT.html}} The method
and its implementation are verified and validated in
\S\ref{Sec:validation_and_benchmarks}. In addition to standard code
verifications using analytical solutions (see \S\ref{Sec:validation})
and tests on non-reflecting boundary conditions (see
\S\ref{Subsec:non_reflecting_boundary_conditions}), we revisit two
benchmarks problems. First, we solve in \S\ref{subse:shocktube} a
two-dimensional shocktube problem proposed by
\cite{Daru_Tenaud_2000,Daru_Tenaud_2009} and demonstrate grid
convergence. Following the initiative of \citep{Daru_Tenaud_2020} and
to facilitate rigorous quantitative comparisons with other research
codes, we provide very accurate computations of the skin friction
coefficient for this problem; these results are freely available at
\citep{testvectors_2021}. To the best of our knowledge,
the level of accuracy we achieved for this benchmark has never been matched before. We also demonstrate in
\S\ref{subse:airfoil} that the proposed method can reliably predict
pressure coefficients on the well-studied supercritical airfoil Onera
OAT15a in the supercritical regime at
Mach 0.73 in three dimensions and at Reynolds number $3\CROSS 10^6$ (see \citep{Deck_2005}, \cite{Deck_Renard_2020},
\cite{Nguyen_Terrana_Peraire_AIAA_2020}). Finally, a series of
synthetic benchmarks are presented in \S\ref{subse:scaling} to assess
the performance of the compute kernels by investigating the strong and
weak scalability of our implementation. Technical details are reported
in Appendix~\ref{sec:BCmath}.

%%%%%%%%%%%%%%%%%%%%%%%%%%%%%%%%%%%%%%%%%%%%%%%%%%%%%%%%%%%%%%%%%%%%%%%%%%%%%%%%
%%%%%%%%%%%%%%%%%%%%%%%%%%%%%%%%%%%%%%%%%%%%%%%%%%%%%%%%%%%%%%%%%%%%%%%%%%%%%%%%
%%%%%%%%%%%%%%%%%%%%%%%%%%%%%%%%%%%%%%%%%%%%%%%%%%%%%%%%%%%%%%%%%%%%%%%%%%%%%%%%

\section{Problem description, finite element setting, time splitting}
\label{sec:preliminaries}

We briefly introduce relevant notation, recall the compressible
Navier-Stokes equations, discuss the finite element setting for the
proposed algorithm, and introduce the operator-splitting technique that is
used to make the method invariant domain preserving under a standard
hyperbolic CFL time step restriction. We follow in large parts the notation
introduced in \citep{Guermond_Maier_Popov_Tomas_CMAME_2020}.

%%%%%%%%%%%%%%%%%%%%%%%%%%%%%%%%%%%%%%%%%%%%%%%%%%%%%%%%%%%%%%%%%%%%%%%%%%%%%%%%

\subsection{The model} \label{Sec:model}
Given a bounded, polyhedral domain $\Dom$ in $\Real^d$, an initial time
$t_0$, and initial data $\bu_0\eqq (\rho_0,\bbm_0,E_0)$, we look for $\bu:
\Dom\CROSS[t_0,+\infty) \to \Real_+\CROSS \Real^d\CROSS \Real_+$ solving
the compressible Navier-Stokes system in some weak sense:
\begin{subequations} \label{NS}%
  \begin{align}%
    \partial_t \rho + \DIV(\bv \rho) &=  0,
    \label{mass}
    \\
    \partial_t \bbm + \DIV\big(\bv\otimes \bbm + p(\bu) \polI -
    \pols(\bv)\big) &= \bef, \label{momentum}
    \\
    \partial_t E  + \DIV\big(\bv(E + p(\bu)) - \pols(\bv)\bv +
    \Hflux(\bu)\big) &= \bef\SCAL\bv.
    \label{total_energy}
  \end{align}%
\end{subequations}
Here $\rho$ is the density, $\bbm$ is the momentum, $E$ is the total
energy, $p(\bu)$ is the pressure, $\polI \in \mathbb{R}^{d\times d}$
is the identity matrix, $\bef$ is an external force,
$\pols(\bv)$ is the viscous stress tensor and $\Hflux(\bu)$ is the
heat-flux.  The quantity $\bv\eqq \rho^{-1} \bbm$ is called velocity and
$ e(\bu) \eqq \rho^{-1}E - \frac12\|\rho^{-1} \bbm\|_{\ell^2}^2$ is
called specific internal energy.  Given a state $\bu\in \Real^{d+2}$,
$\rho(\bu)$ denotes the first coordinate (\ie density), $\bbm(\bu)$ denotes the
$\Real^d$-valued vector whose components are the $2$-nd up to the $(d+1)$-th
coordinates of $\bu$ (\ie the momentum), and $E(\bu)$ is the last
coordinate of $\bu$ (\ie the total energy). Boundary conditions for \eqref{NS} and the implementation
of these condition are discussed in
detail in Section~\ref{Sec:boundary_conditions}.

To simplify the notation later on, we introduce the flux
$\polf(\bu) := (\bbm, \bv\otimes \bbm + p(\bu)\polI_d, \bv(E+p))\tr
\in \Real^{(d+2)\times d}$,
where $\polI_d$ is the $d\CROSS d$ identity matrix. Although it is
often convenient to assume that the pressure $p(\bu)$ is derived from
a complete equation of state,
most of what is said here holds true by only assuming that the
pressure is given by an oracle (see, \eg
\cite{Clayton_Guermond_Popov_2021}). In the applications reported at
the end of the paper, though, we use the ideal gas law
$p(\bu) = (\gamma-1) \rho e(\bu)$.

The fluid is assumed to be Newtonian and the heat-flux is assumed to
follow Fourier's law:
\begin{align}
  \pols(\bv) &\eqq 2 \mu \pole(\bv) + (\lambda-\tfrac23 \mu) \DIV\bv \polI,
  \qquad
  \pole(\bv) \eqq \GRADs \bv \eqq \tfrac12\big(\GRAD \bv + (\GRAD
  \bv)\tr\big),
  \\
  \Hflux(\bu) &\eqq-c_v^{-1}\kappa \GRAD e,
\end{align}%
where $\mu>0$ and $\lambda\ge 0$ are the shear and the bulk
viscosities, respectively, $\kappa$ is the thermal conductivity, and
$c_v$ is the heat capacity at constant volume. For the sake of
simplicity, we assume that $\mu, \lambda, \kappa,$ and $c_v$ are
constant.

Important properties we want to maintain at the discrete level are the
positivity of the density and the positivity of the specific internal
energy. We formalize these constraints by introducing the set of admissible states:
\begin{equation}
\calA \eqq \{ \bu\eqq (\rho,\bbm, E ) \tr \in \Real^{d+2} \st \rho > 0 ,
\ e ( \bu ) > 0 \}.
\end{equation}
We also want that in the inviscid regime limit (\ie $\lambda\to 0$,
$\mu\to 0$, $\kappa\to 0$), the algorithm satisfies the local minimum
principle of the specific entropy at each time step.

%%%%%%%%%%%%%%%%%%%%%%%%%%%%%%%%%%%%%%%%%%%%%%%%%%%%%%%%%%%%%%%%%%%%%%%%%%%%%%%%

\subsection{Finite element setting}\label{Sec:fe}
Although, as claimed in \citep{Guermond_Popov_Tomas_CMAME_2019}, the
proposed approximation technique is discretization agnostic
and can be implemented with finite volumes and with discontinuous or continuous finite elements,
we restrict ourselves here to continuous finite elements since it greatly
simplifies the approximation of the second-order differential
operators.

Let $\famTh$ be a sequence of shape-regular meshes covering $\Dom$
exactly. Here $\calH$ is the index set of the mesh sequence, and
$h$ is the typical mesh-size.  Given some mesh $\calT_h$, we denote by
$P(\calT_h)$ a scalar-valued finite element space with global shape functions
$\{\varphi_i\}_{i\in\calV}$. Here, the index $i$ is abusively called a
degree of freedom, and since we restrict the presentation to continuous Lagrange
elements, degrees of freedom are also called nodes.  The approximation of the
state $\bu\eqq(\rho,\bbm,E)$ will be done in the vector-valued space
$\bP(\calT_h)\eqq (P(\calT_h))^{d+2}$. We define the stencil at $i$ by
\begin{align*}
 \calI(i)\eqq
  \Big\{j\in\calV \;\Big|\; |\text{supp}(\varphi_j)\cap\text{supp}(\varphi_i)|
  \not= 0 \Big\}, \quad\text{and we set}\quad
 \calI^*(i) \eqq\calI(i){\setminus}\{i\}.
\end{align*}
 We
assume that the shape functions are non-negative, \ie $\varphi_i\ge 0$ for all
$i\in \calV$ on all of $\Dom$, and satisfy the
partition of unity property $\sum_{i\in\calV}\varphi_i=1$.

The concrete implementation used in the verification and benchmark
section~\S\ref{Sec:validation_and_benchmarks} is based on the finite
element library \texttt{deal.II} \citep{dealII92,dealIIcanonical} and
uses continuous mapped $\polQ_1$ elements. The code called
\texttt{ryujin} is available online
(\url{https://github.com/conservation-laws/ryujin} and documented in
\cite{maier2021ryujin}).  We denote by $\calV\upbnd$ the set of the
degrees of freedom whose shape functions are supported on the boundary $\front$. The
set $\calV\upint\eqq \calV{\setminus}\calV\upbnd$ is composed of the
degrees of freedom  whose shape functions are supported in the interior of $\Dom$. They
are henceforth called interior degrees of freedom.

The hyperbolic part of the algorithm depends on four mesh-dependent
quantities, $m_i$, $m_{ij}$, $\bc_{ij}$, and $\bn_{ij}$ defined as
follows for all $i\in\calV$ and all $j\in\cal(I)$:
\begin{equation}
m_{i} :=\int_\Dom \varphi_i\diff x,\qquad
m_{ij} := \int_{\Dom} \varphi_i \varphi_j \diff x ,\qquad
\bc_{ij} :=\int_\Dom \varphi_i \GRAD \varphi_j \diff x,\qquad
\bn_{ij} := \frac{\bc_{ij}}{\|\bc_{ij}\|_{\ell^2}}.
\end{equation}
Here $m_i$ and $m_{ij}$ are the entries of the lumped mass matrix and
consistent mass matrix, respectively. The partition of
unity property implies the identities $m_i = \sum_{j\in\calI(i)}m_{ij}$ and
$\sum_{j\in\calI(i)} \bc_{ij} =0$.  The second identity is essential
to establish conservation. Using the lumped mass matrix introduces
undesirable dispersive errors, whereas using the consistent mass matrix may
require global matrix inversions. The present contribution uses
the following approximate inverse of the mass matrix with entries defined by
\begin{align}\label{bijdef}
  \tfrac{1}{m_i} (\delta_{ij} + b_{ij}), \ \ \text{where the coefficients}  \ \
  b_{ij} \eqq \delta_{ij} -\tfrac{m_{ij}}{m_j} \ \ \text{satisfy} \ \
  \textstyle{\sum_{i \in \calI(j)}} b_{ij} = 0.
\end{align}
Using this approximate inverse bypasses the need to invert the mass
matrix. These ideas were originally documented in \cite{GuerPasq2013}
and \cite[\S3.3]{GuerNaza2014}. It is also shown therein that
  this approximate inverse preserves the conservation properties of
  the scheme. After extensive benchmarking, it is observed in
  \cite[\S3.2]{maier2020massively} that the best parallel performance
is achieved by pre-computing and storing on each MPI rank the
coefficients $\{m_i\}_{i\in\calV}$ and the matrices
$\{m_{ij}\}_{i\in\calV,j\in\calI(i)}$,
$\{\bc_{ij}\}_{i\in\calV,j\in\calI(i)}$.  The coefficients of the
matrices $\{b_{ij}\}_{i\in\calV,j\in\calI(i)}$,
$\{b_{ji}\}_{j\in\calV,i\in\calI(j)}$ and
$\{\bn_{ij}\}_{i\in\calV,j\in\calI(i)}$ can be recomputed on the fly
from $\{m_{ij}\}_{i\in\calV,j\in\calI(i)}$,
$\{\bc_{ij}\}_{i\in\calV,j\in\calI(i)}$ in each time step. For later
reference, $\polM_L$, $\polM$, and $\polB$ denote: the lumped mass
matrix, the consistent mass matrix, and the matrix with entries
$\{b_{ij}\}_{i\in\calV,j\in\calI(i)}$ respectively.

\begin{remark}[Space discretization] The focus of the paper is on
continuous Lagrange elements since the discretization of both the hyperbolic and the
diffusion operators is relatively natural with these elements. Discontinuous elements can
also be used at the expense of additional overhead in the assembly of the diffusion
terms \citep{Kronbichler2018}. Another space discretization enjoying a
straightforward implementation of the diffusion terms are rational barycentric
coordinates on arbitrary polygons/polyhedrons. Rational barycentric coordinates
satisfy the partition of unity property and can be  made globally
continuous (\ie $H^1$-conforming), see \cite{Floater2015} and references
therein. All the developments presented in this manuscript are
directly applicable in that context too.
\end{remark}

%%%%%%%%%%%%%%%%%%%%%%%%%%%%%%%%%%%%%%%%%%%%%%%%%%%%%%%%%%%%%%%%%%%%%%%%%%%%%%%%

\subsection{Strang splitting}
\label{subsec:strang}
The key idea for the time approximation of \eqref{NS} is to use
Strang's splitting.  As routinely done in the literature, we separate
the hyperbolic part and the parabolic parts of the problem (see \eg
\cite{Demkowicz_etal_1990}, we also refer the reader to
\cite{Beam_warming_AIAA_1977},
\cite[\S11.2]{Bristeau_Glowinski_Periaux_1987} where other operator-splittings
are considered).
The hyperbolic part of the problem consists of solving  the Euler equations:
\begin{subequations}
  \label{hyperbolic}
  \begin{align}
    \partial_t \rho + \DIV(\bv \rho) &=  0,
    \label{hyperbolic_mass}
    \\
    \partial_t \bbm + \DIV(\bv\otimes \bbm + p(\bu) \polI) &= \bzero,
    \label{hyperbolic_momentum}
    \\
    \partial_t E  + \DIV(\bv(E + p(\bu)) &= 0 ,
    \label{hyperbolic_total_energy}
  \end{align}%
\end{subequations}
which is formally equivalent to considering the limit for $\mu,\lambda,\kappa
\to0 $ in \eqref{NS}. The missing dissipative terms in \eqref{hyperbolic} compose the parabolic part of the problem:
\begin{subequations}
  \label{parabolic_simplified}
  \begin{align}
    \partial_t\rho &= 0,
    \\
    \partial_t(\rho\bv) - \DIV(\pols(\bv)) &= \bef,
    \label{parabolic_momentum_simplified}
    \\
    \partial_t (\rho e) - c_v^{-1}\kappa \LAP e  &=
    \pols(\bv){:}\pole(\bv).
    \label{parabolic_internal_energy_simplified}
  \end{align}
\end{subequations}
This decomposition is the cornerstone of the operator-splitting
scheme considered in this paper. For any admissible state $\bu_0\in\calA$ at
time $t_0$ and any time $t\ge t_0$, we denote by
$S\loH(t,t_0)(\bu(t_0))=\bu(t)$ the solution map of the hyperbolic
system \eqref{hyperbolic}; that is, $\bu(t)$ solves \eqref{hyperbolic} with
appropriate boundary conditions and $\bu(t_0)=\bu_0$.
The subscript $\loH$ is meant to remind us that $S\loH$ solves the hyperbolic
problem. Similarly, letting $\bu_0\in\calA$ be some admissible state at some
time $t_0$, and letting $\bef$ be some source term, we denote by
$S\loP(t,t_0)(\bu_0,\bef)=\bu(t)$ the solution map of the parabolic system
\eqref{parabolic_simplified}. Then, given an admissible state $\bu_0\in\calA$
at time $t_0$ and given some time step $\dt$, we approximate the solution to
the full Navier-Stokes system \eqref{NS} at $t_0+2\dt$ by using Strang's
splitting technique:
\begin{equation}
  \label{continuous_Strang}
  \bu(t_0+2\tau)
  \;\approx\;
  \big( S_{\text{H}}(t_0+2\dt,t_0+\dt)
  \circ S_{\text{P}}(t_0+2\dt,t_0    )(\,.\,,\bef)
  \circ S_{\text{H}}(t_0+ \dt,t_0    )\big)(\bu_0).
\end{equation}
In other words, we first perform an explicit hyperbolic update of $\bu_0$ with
step size $\tau$. Then, using this update as initial state at $t_0$ and
the source term $\bef$, we solve the parabolic problem from $t_0$ to $t_0+2\dt$.
Using in turn this solution as initial state at $t_0+\dt$, we compute
the final update by solving \eqref{hyperbolic} from $t_0+\dt$
to $t_0+2\dt$.

%%%%%%%%%%%%%%%%%%%%%%%%%%%%%%%%%%%%%%%%%%%%%%%%%%%%%%%%%%%%%%%%%%%%%%%%%%%%%%%%
%%%%%%%%%%%%%%%%%%%%%%%%%%%%%%%%%%%%%%%%%%%%%%%%%%%%%%%%%%%%%%%%%%%%%%%%%%%%%%%%
%%%%%%%%%%%%%%%%%%%%%%%%%%%%%%%%%%%%%%%%%%%%%%%%%%%%%%%%%%%%%%%%%%%%%%%%%%%%%%%%

\section{Discretization of $S_{\text{H}}$}
\label{Sec:hyperbolic_limit}

In this section we describe the space and time approximation of the
hyperbolic operator $S\loH$ and summarize important implementation
details that make the algorithm efficient and highly scalable. The
time approximation is done by using the explicit strong stability
preserving Runge-Kutta method SSPRK(3,3), see
\cite[Eq.~(2.18)]{Shu_Osher1988} and
\cite[Thm.~9.4]{Kraaijevanger_1991}. This method requires three calls
to the forward-Euler update discussed in this section. The
forward-Euler scheme itself requires the computation of a low-order
solution, a provisional high-order solution (possibly constraint
violating), and the final flux-limited solution to be returned.  The
various steps described in \S\ref{Sec:low_order}--\S\ref{Sec:limiting} are
summarized in Algorithm~\ref{alg:euler}
in Appendix~\ref{Sec:appendix_hyperbolic_step}. The implementation of
the boundary conditions for $S\loH$ is explained in
\S\ref{Sec:boundary_conditions}.

%%%%%%%%%%%%%%%%%%%%%%%%%%%%%%%%%%%%%%%%%%%%%%%%%%%%%%%%%%%%%%%%%%%%%%%%%%%%%%%%

\subsection{Low-order step}  \label{Sec:low_order}
Let $t^n$ be the current time and let $\bu_h^n=\sum_{i\in \calV}
\bsfU_i^n\varphi_i$ be the current approximation which we assume to be
admissible, \ie $\bsfU_i^n\in \calA$ for all $i\in\calV$. For all
$i\in\calV$ and for all $j\in \calI^*(i)$, we consider the
Riemann problem with left state $\bsfU_i^n$, right state $\bsfU_j^n$, and
flux $\polf(\bw)\bn_{ij}$.  We denote by
$\lambda^{\max}(\bsfU_i^{n},\bsfU_j^{n},\bn_{ij})$ any upper bound on the
maximum wavespeed in this Riemann problem. Iterative techniques to compute
the maximum wavespeed are described in
\cite{Colella_Glaz_JCP_1985,Toro_2009}. In this manuscript we use the
inexpensive non-iterative guaranteed upper-bound thoroughly described in
\citep{Guermond_Popov_Fast_Riemann_2016,Clayton_Guermond_Popov_2021}. With
this estimate, we define the graph viscosity coefficient:
\begin{equation}
  d_{ij}\upLn = \max\big(\lambda^{\max}(\bsfU_i^{n},\bsfU_j^{n},\bn_{ij})
  \|\bc_{ij}\|_{\ell^2}, \lambda^{\max}(\bsfU_j^{n},\bsfU_i^{n},\bn_{ji})
  \|\bc_{ji}\|_{\ell^2}\big).
\end{equation}
Noticing that $\bc_{ij} =-\bc_{ji}$ if $i$ is an internal node
(because $\varphi_{i|\front}=0$ if $i\in\calV\upint$), we infer that
$\lambda^{\max}(\bsfU_i^{n},\bsfU_j^{n},\bn_{ij})
\|\bc_{ij}\|_{\ell^2}=
\lambda^{\max}(\bsfU_j^{n},\bsfU_i^{n},\bn_{ji})
\|\bc_{ji}\|_{\ell^2}$
for all $i\in\calV\upint$. This property allows for some computation
savings in the construction of $d\upLn$. For all
$i\in\calV\upint$ and all $i<j\in\calI(i)$, one computes
  $d_{ij}\upLn= \lambda^{\max}(\bsfU_i^{n},\bsfU_j^{n},\bn_{ij})
  \|\bc_{ij}\|_{\ell^2}$,
  and for all $i\in\calV\upbnd$ and all $i<j\in \calI(i)$, one
  computes
  $d_{ij}\upLn=\max(
  \lambda^{\max}(\bsfU_i^{n},\bsfU_j^{n},\bn_{ij})\|\bc_{ij}\|_{\ell^2},
  \lambda^{\max}(\bsfU_j^{n},\bsfU_i^{n},\bn_{ji})
  \|\bc_{ji}\|_{\ell^2} )$.  Finally one sets
$d_{ij}\upLn \leftarrow \max (d_{ij}\upLn, (d\upLn)\tr_{ij})$ for all
$j\in \calI^*(i)$, where $(d\upLn)\tr$ is the transpose
of $d\upLn$. The diagonal entries in $d\upLn$ are obtained by setting
$d_{ii}\upLn :=-\sum_{j\in\calI^*(i)}d_{ij}\upLn$.  This
technique saves almost half the computing time for $d_{ij}\upLn$
\citep[\S5.2.1]{maier2020massively}.  Once $d\upLn$ is known, the time-step size is defined by
\begin{align}\label{dtCfl}
  \dt_n := \text{c}_{\text{cfl}}\CROSS  \min_{i \in \calV} \big(-\tfrac{m_i}{2
  d_{ii}\upLn}\big),
\end{align}
where $0 < \text{c}_{\text{cfl}} \le 1$ is a user-defined constant.
The condition $\text{c}_{\text{cfl}} \le 1$ is shown in
\citep{Guermond_Popov_SINUM_2016} to be sufficient to guarantee that the
low-order method is invariant domain preserving. The time-step size is
computed at the first forward-Euler step of the SSPRK(3,3) algorithm, and this time-step size  is
used for the three stages of the hyperbolic update. Then according to the
Strang splitting algorithm~\eqref{continuous_Strang}, the time-step size
used in the parabolic update is $2\dt_n$, and the time-step size used in
the last hyperbolic update is again $\dt_n$. At the end of the entire
process the new time level is $t^{n}+2\dt_n$.

The low-order update produced by the forward-Euler step as defined in
\citep{Guermond_Popov_SINUM_2016} is
\begin{align}
\label{loworder}
 m_i (\bsfU_i\upLnp - \bsfU_i^{n}) = \dt_n \sum_{j\in\calI(i)}
\bsfF_{ij}\upL, \qquad \qquad
\bsfF_{ij}\upL \eqq  -\flux(\bsfU_j^{n})\bc_{ij}
+ d_{ij}\upLn (\bsfU_j^{n} - \bsfU_i^{n}).%
\end{align}%
The CFL condition $\text{c}_{\text{cfl}}\le 1$ guarantees that
$\bsfU\upLnp$ remains inside the invariant domain. Consequently, for our
choice of Lagrange elements (see \S\ref{Sec:fe}) the invariant domain
property holds true for the finite element function $\bu_h\upLnp$
\citep[Corollary~4.3]{Guermond_Popov_SINUM_2016}.
\begin{remark}[High aspect-ratio meshes]
  The time-step size $\dt_n$ determined by the theoretical estimate~\eqref{dtCfl} decreases
  significantly as the aspect ratio of the cells increases. This may be problematic
  when using meshes with high aspect-ratio cells to resolve thin
  boundary layers (see \S\ref{subse:airfoil}); in this case the aspect ratios can reach
  values up to 50:1 or more. For these configurations we have found
  that a better way to estimate  $\dt_n$ is to adaptively increase the value of
  $\text{c}_{\text{cfl}}$ in \eqref{dtCfl} beyond the limit of 1. This
  requires to additionally check whether the low-order solution
  $\bsfU\upLnp$ still remains in the admissible set, and if not to restart
  the time step with a smaller $\text{c}_{\text{cfl}}$ number. While this
  makes each time step slightly more expensive, the significant increase of the CFL number compensates for
  the otherwise increased cost of using a high aspect-ratio cells.
\end{remark}
To save arithmetic operations and prepare the ground for
the limiting step, we introduce the auxiliary states $\overline\bsfU_{ij}^n$
and rewrite \eqref{loworder} as follows:
\begin{align}
  \label{bar_state}
  &\overline\bsfU_{ij}^n \eqq \frac12 (\bsfU_i^{n} + \bsfU_j^{n})
  - \frac{1}{2d_{ij}\upLn} (\flux(\bsfU_j^{n})
- \flux(\bsfU_i^{n}) )\bc_{ij},\qquad \forall j\in\calI(i), \\
  \label{loworder_bis}
  & \bsfU_i\upLnp = \bsfU_i^{n}  + \frac{2\dt_n}{m_i}
\sum_{j\in\calI(i)} d_{ij}\upLn \overline\bsfU_{ij}^n.
\end{align}
The auxiliary states $\overline\bsfU_{ij}^n$ are essential to define the
bounds that must be guaranteed after limiting.  In particular, if one wants
to limit some quasi-concave functional $\Psi$, one has to compute the local
lower bound $\Psi_i^{\min}\eqq
\min_{j\in\calI(i)}\Psi(\overline\bsfU_{ij}^n)$. In the numerical
illustrations reported in the paper, limiting is done with the following
two functionals: $\Psi_\flat(\bsfU)=\varrho$ and $\Psi_\sharp(\bsfU)
=-\varrho$ (where recalling the notation introduced in \S\ref{Sec:model} we have set
$\varrho\eqq\rho(\bsfU)$). One uses $\Psi_\flat$ to enforce a local minimum
principle on the density and one uses $\Psi_\sharp$ to enforce a local
maximum principle. Additional limiting has to be done to ensure that the
specific internal energy is positive.  This is done in the case of a
$\gamma$-law equation of state by controlling the exponential of the
specific entropy, $\Phi(\bsfU) = \varepsilon(\bsfU) \varrho^{-\gamma}$ where
$\varepsilon(\bsfU) := E - \tfrac{|\bbm|^2}{2\rho}$ is the internal energy. For
theoretical reasons explained in
\citep[\S3.2]{Guermond_Nazarov_Popov_Tomas_SISC_2019}, the local lower
bound for this functional uses the states $\bsfU_j^n$ instead of the
auxiliary states $\overline\bsfU_{ij}^n$, \ie one defines
$\Phi_i^{\min}\eqq \min_{j\in\calI(i)}\Phi(\bsfU_j^n)$.

%%%%%%%%%%%%%%%%%%%%%%%%%%%%%%%%%%%%%%%%%%%%%%%%%%%%%%%%%%%%%%%%%%%%%%%%%%%%%%%%

\subsection{High-order step}
The computation of the provisional high-order update proceeds as for the
low-order update with two exceptions: (i) the graph viscosity is reduced;
(ii) the lumped mass matrix is replaced by an approximation of the
consistent mass matrix to correct third-order dispersive effects. More
precisely the high-order viscosity is defined as follows:
$d_{ij}\upHn = (\frac{\alpha_i +\alpha_j}{2}) d_{ij}\upLn$,
where $0 \leq \alpha_i \leq 1$ is an entropy-production indicator (see
\cite[\S3.4]{Guermond_Nazarov_Popov_Tomas_SISC_2019} and
\cite{maier2020massively} for some possible implementations). The key idea
is that $\alpha_i+\alpha_j$ is small in regions where the solution is
smooth and there is no entropy production. We introduce $\bsfF\upH$ to be
the vector of the high-order fluxes whose entries are $(d+2)$-valued and
defined for every $i\in\calV$ by
\begin{equation}
  \bsfF_{i}\upH \eqq \sum_{j \in \calI(i)}\bsfF_{ij}\upH \ \
\text{where} \ \
  \bsfF_{ij}\upH \eqq -\polf(\bsfU_j^n)\bc_{ij}
+ d_{ij}\upHn (\bsfU_j^{n} - \bsfU_i^{n}).
\end{equation}
Recalling that $\polM^{-1}\approx \polM_L^{-1}(\polI + \polB)$, the high-order
update is obtained by setting
\begin{equation}\label{highorderCompact}
  \polM_L (\bsfU\upHnp - \bsfU^{n})  \eqq  \dt_n (\polI + \polB) \bsfF\upH.
\end{equation}
Instead of using this expression, we proceed as in \citep[\S3.4]{GuerNaza2014}
to prepare the ground for limiting. Recalling that $b_{ii} = -\sum_{j\in
\calI^*(i)} b_{ji}$, we rewrite the high-order update
\eqref{highorderCompact} as follows:
\begin{equation}\label{highorderIndicial}
 m_i (\bsfU_i\upHnp - \bsfU_i^{n}) = \dt_n \bsfF_{i}\upH + \dt_n  \sum_{j\in \calI^*(i)} b_{ij} \bsfF_{j}\upH - b_{ji} \bsfF_{i}\upH.
\end{equation}
Now we subtract \eqref{loworder} from \eqref{highorderIndicial}
and obtain
\begin{equation}\label{highorder}
  m_i \bsfU_i\upHnp = m_i\bsfU_i\upLn +  \dt_n  \sum_{j\in \calI^*(i)} b_{ij} \bsfF_{j}\upH
- b_{ji} \bsfF_{i}\upH
  + (d_{ij}\upHn-d_{ij}\upLn) (\bsfU\upn_j-\bsfU\upn_i).
\end{equation}
The state $\bsfU\upHnp$ is a high-order approximation of $\bu(t^{n+1})$ if the viscosity
$d_{ij}\upHn$ is indeed small, but it may not be admissible. To save arithmetic operations,
one does not compute $\bsfU\upHnp$ since the actual
high-order update is obtained after limiting as explained in the next
subsection.

\subsection{Limiting} \label{Sec:limiting}
Recall that, as discussed at the end of \S\ref{Sec:low_order}, we want
the high-order update to satisfy
$\Psi_\flat(\bsfU_i^{n+1})\ge \Psi_{\flat,i}^{\min}$,
$\Psi_\sharp(\bsfU_i^{n+1})\ge \Psi_{\sharp,i}^{\min}$, and
$\Phi(\bsfU_i^{n+1})\ge \Phi_i^{\min}$ for all $i\in\calV$.  For this purpose
we rewrite \eqref{highorder} as follows:
\begin{align}
  \bsfU_i\upHnp &=\sum_{j\in \calI^*(i)}\lambda_i (\bsfU_i\upLn + \bsfP_{ij}^n),\\
  \text{with}\quad \bsfP_{ij}^n &\eqq  \frac{\dt_n}{m_i\lambda_i}\Big( b_{ij} \bsfF_{j}\upH - b_{ji} \bsfF_{i}\upH
  + (d_{ij}\upHn-d_{ij}\upLn) (\bsfU\upn_j-\bsfU\upn_i)\Big),\hspace{-3mm}
\end{align}
where $\lambda_i \eqq (\card(\calI(i)) -1)^{-1}$. This motivates
computing the final (flux-limited) solution as
\begin{align}
\label{LambdaArrangement}
 \bsfU_i^{n+1} = \bsfU_i\upLnp
+ \sum_{j \in \calI^*(i)} \lambda_i \limiter_{ij} \bsfP_{ij}^n,
\end{align}
where $\limiter_{ij} \in [0,1]$ for all $\{i,j\}$ are the limiters. Observe
that if $\limiter_{ij} = 1$ for all $j \in \calI^*(i)$ then
$\bsfU_i^{n+1} = \bsfU_i\upHnp$, and if $\limiter_{ij} = 0$ for all $j \in
\calI^*(i)$ then $\bsfU_i^{n+1} = \bsfU_i\upLnp$. For every
$j\in \calI^*(i)$, we define $\limiter_{j}^i$ to be the
largest number in $[0,1]$ that is such that
\begin{equation}
  \Psi_\flat(\bsfU_i\upLn + \limiter_j^i \bsfP_{ij}) \ge \Psi_{\flat,i}^{\min},\qquad
  \Psi_\sharp(\bsfU_i\upLn + \limiter_j^i \bsfP_{ij}) \ge \Psi_{\sharp,i}^{\min},\qquad
  \Phi(\bsfU_i\upLn + \limiter_j^i \bsfP_{ij}) \ge \Phi_{i}^{\min}.
\end{equation}
This number always exists since by construction $\limiter_j^i=0$ satisfies the
above three constraints. Finding this number (or a very close lower
estimate thereof) is quite simple and explained in
\citep{Guermond_Popov_Tomas_CMAME_2019, maier2020massively}. Then, in order
to maintain mass conservation, $\limiter_{ij}$ is defined by setting
$\limiter_{ij} = \min(\limiter_j^i, \limiter_i^j)$. This symmetry property,
together with the identity $m_i\lambda_i \bP_{ij} = -m_j\lambda_j\bP_{ji}$,
ensures that the mass of the high-order update is unchanged by limiting,
\ie
\begin{equation}
  \sum_{i\in\calV} m_i \bsfU_i^{n+1} = \sum_{i\in\calV} m_i
  \bsfU_i\upLnp.
  \label{mass_conservation_by_limiting}
\end{equation}
Replacing
$\limiter_j^i$ by $\min(\limiter_j^i, \limiter_i^j)$ does not
violate the invariant domain properties since $\calA$ is convex, \citep{Guermond_Popov_Tomas_CMAME_2019}.

Since $\bsfU_i^{n+1} =  \sum_{j \in\calI^*(i)} \lambda_i (\bsfU_i\upLnp +\limiter_{ij} \bsfP_{ij}^n) + \sum_{j \in
\calI^*(i)} \lambda_i (1-\limiter_{ij}) \bsfP_{ij}^n$, and
$\sum_{j \in\calI^*(i)} \lambda_i (\bsfU_i\upLnp +\limiter_{ij} \bsfP_{ij}^n)$ satisfies all the bounds,
one can repeat the above process and compute a new
set of limiters by
replacing $\bsfU_i\upLnp$ by $\sum_{j \in\calI^*(i)} \lambda_i (\bsfU_i\upLnp +\limiter_{ij} \bsfP_{ij}^n)$  and
$\bsfP_{ij}^n$ by $(1-\limiter_{ij}) \bsfP_{ij}^n$.
We have observed that this iterative limiting process must be
applied at least two times to reach optimal convergence.
All the simulations reported in the paper are done with two passes of
limiting.

Let $\bn$ denote the outward unit normal vector field on $\partial D$. To
properly formulate the conservation properties of the method after
limiting, we define an approximation of the normal vector and boundary mass
at every boundary node $i\in\calV\upbnd$  by setting
\begin{equation} \label{unit_normal_vector}
  \bn_i \eqq \frac{ \int_{\partial D} \varphi_i\,\bn \diff s}{m_i\upbnd},
  \qquad m_i\upbnd \eqq \Big\|\int_{\partial D} \varphi_i\,\bn \diff
  s\Big\|_{\ell^2}.%
\end{equation}%
\begin{lemma}[Balance of mass and admissibility after limiting]\label{Lem:boundary_flux}%
\phantom{The following holds true:}%
  \begin{enumerate}[font=\upshape,label=(\roman*)]
    \item
      For all $\bu_h\eqq \sum_{i\in\calV} \bsfU_i\varphi\in \bP(\calT_h)$, the
following holds true:
      $\int_\Dom \bu_h\diff x = \sum_{i\in\calV} m_i \bsfU_i$.
    \item Let $\bsfU^n$ be a collection of admissible states. Let
      $\bsfU\upnp$ be the update after one forward-Euler step and
      after limiting. Then $\bsfU^{n+1}$ is admissible under the
        condition $\text{\rm c}_{\text{\rm cfl}}\le 1$ and
      \begin{equation}
        \sum_{i\in\calV} m_i \bsfU_i^{n+1} + \dt_n \sum_{i\in\calV\upbnd}
        m\upbnd_i \polf(\bsfU_i^n)\bn_i = \sum_{i\in\calV} m_i \bsfU_i^{n}, \label{Eq:Lem:boundary_flux}
      \end{equation}
  \end{enumerate}
\end{lemma}

\begin{proof}
See \ref{Lem:boundary_flux_bis} in Appendix~\ref{sec:BCmath}.
\end{proof}

\begin{remark}[Literature] The convex limiting technique is a
  generalization of the Flux Corrected Transport that accommodates
  quasi-concave constraints. (Recall that FCT is by design adapted to
  affine constraints; see \eg
  \cite{Boris_books_JCP_1973,Zalesak_1979,KuzminLoehnerTurek2004}.)
  Convex limiting has been introduced in
  \citep{Guermond_Nazarov_Popov_Tomas_SISC_2019,
    Guermond_Popov_Tomas_CMAME_2019} for the Euler equations and
  general hyperbolic systems. We refer to
  \cite{maier2020massively,MaierTomas2020} for a detailed discussion
  of a high performance implementation of the hyperbolic solver part
  of the system.
\end{remark}

%%%%%%%%%%%%%%%%%%%%%%%%%%%%%%%%%%%%%%%%%%%%%%%%%%%%%%%%%%%%%%%%%%%%%%%%%%%%%%%%

\section{Euler boundary conditions}
\label{Sec:boundary_conditions}

In this section we describe how boundary conditions are enforced in
the hyperbolic step. To the best of our knowledge, the
  implementation details of the various boundary conditions considered
  in this section for continuous finite elements, and the associated theoretical results regarding
  conservation and admissibility are original.

\subsection{Overview}
Since the time stepping is explicit, the boundary conditions are enforced by
post-processing the approximation produced at the end of each stage of the
SSPRK(3,3) algorithm. The Butcher tableau of the explicit SSPRK(3,3) algorithm
is given in the left panel of \eqref{SPPRK33}. Given some ODE system $\partial_t
u = L(t,u)$ and $u^n\eqq u(t^n)$, the steps to approximate the solution to
$\partial_t u = L(t,u)$ at $t^{n+1}$ are shown in the right panel of
\eqref{SPPRK33}.
\begin{equation}
  \begin{array}{c|ccc}
\ST 0       &      0  &          &  \\
\ST      1  &      1  & 0        &  \\
\ST \frac12 &  \frac14& \frac14  & 0\\[2pt] \hline
\ST     & \frac16 &  \frac16 &  \frac23
\end{array}
\qquad\qquad
\begin{aligned}
w^{(1)} &\eqq u^n+ \dt_n L(t_n,u^n),\\
w^{(2)} &\eqq \tfrac34 u^n + \tfrac14(w^{(1)}+\dt_n L\big(t_n+\dt_n,w^{(1)})\big),\\
u^{n+1} &\eqq \tfrac13 u^n + \tfrac23(w^{(2)}+\dt_n L\big(t_n+\tfrac12\dt_n,w^{(2)})\big).
\end{aligned} \label{SPPRK33}
\end{equation}

The intermediate stages $w^{(1)}$, $w^{(2)}$ and the final stage
$u^{n+1}$ approximate $u$ at $t^{n+1}=t^n+\dt_n$, $t^n+\frac12\dt_n$, and
$t^n+\dt_n$, respectively. Hence, the time-dependent boundary conditions have to
be enforced on the intermediate stages $w^{(1)}$, $w^{(2)}$ and the final step
$u^{n+1}$ (using the corresponding collocation times).

We consider two types of boundary conditions: (i) Slip condition, also called
``reflecting'': $\bv\SCAL\bn=0$; (ii) Non-reflecting condition.  Let
$\front\los\subset \front$ be the boundary where one wants to enforce
the slip condition, and let $\front\lonr$ denote the complement of
$\front\los$ in $\front$, \ie $\front{\setminus}\front\los$. The index
$\lonr$ reminds us that $\front\lonr$ is a non-reflecting boundary
(either an inflow or an outflow boundary).  Let
$\calV\upbnd\los\subset\calV\upbnd$ be the collection of all the
boundary degrees of freedom $i$ such that
$\varphi_{i|\front\los}\not\equiv 0$.  Similarly,
$\calV\upbnd\lonr\subset\calV\upbnd$ is the collection of all the
boundary degrees of freedom $i$ such that
$\varphi_{i|\front\lonr}\not\equiv 0$.  We now define the normal
vectors associated with the degrees of freedom in $\calV\upbnd\los$
and $\calV\upbnd\lonr$:
\begin{equation} \label{def_ns_nio}
  \bn\ups_i \eqq \frac{\int_{\front\los}\varphi_i\bn \diff s}{
  \|\int_{\front\los}\varphi_i \bn \diff s \|_{\ell^2}},\qquad
  \bn\upnr_i \eqq \frac{\int_{\front\lonr}\varphi_i\bn \diff s}{
  \|\int_{\front\lonr}\varphi_i\bn \diff s \|_{\ell^2}}.
\end{equation}
Notice that although $\front\los\cap \front\lonr =\emptyset$, the two
index sets $\calV\upbnd\los$ and $\calV\upbnd\lonr$ may not be
disjoint. Hence, there may exists two notions of the normal vector at the
nodes sitting at the interface between $\front\los$ and
$\front\lonr$. Let us set
$m_i\ups\eqq \|\int_{\front\los}\varphi_i\bn \diff s \|_{\ell^2}$ and
$m_i\upnr\eqq \|\int_{\front\lonr}\varphi_i \bn \diff s
\|_{\ell^2}$.
Then~\eqref{unit_normal_vector} and \eqref{def_ns_nio} imply that
$m_i\upbnd \bn_i = m_i\ups\bn_i\ups + m_i\upnr\bn_i\upnr$.

In the following subsections, the symbol $\bsfU$ denotes
the state obtained at the end of one
forward-Euler step. This state has to be postprocessed to account for the boundary
conditions. It could be any one of the three states
$\bsfW^{(1)}$, $\bsfW^{(2)}$, or $\bsfU^{n+1}$. The postprocessed
state is denoted $\bsfU\upP$.

\subsection{Slip boundary condition} \label{Subsubsec:slip} We start
with the slip boundary condition.  Let $i\in \calV\upbnd\los$ and let
$\bsfU_i=(\varrho_i,\bsfM_i,\sfE_i)\tr$, \ie
$\varrho_i\eqq \rho(\bsfU_i)$, $\bsfM_i\eqq \bbm(\bsfU_i)$, and
$\sfE_i=E(\bsfU_i)$.  We enforce the slip boundary condition at $i$ by
setting
\begin{align}
  \bsfU_i\upP \eqq
(\varrho_i,\bsfM_i-(\bsfM_i\SCAL\bn_i\ups)\bn_i\ups,\sfE_i)\tr.
\label{slip_bc}
\end{align}

\subsection{Non-reflecting boundary condition}\label{Subsec:non_reflecting_bcs}
We now consider non-reflecting boundary conditions at
$i\in\calV\upbnd\lonr$.  To simplify the notation, we omit the node
index $_{i}$ since no argument regarding conservation properties is
made. We also write $\bn$ instead of $\bn_i\upnr$. We propose two
post-processing techniques: (i) one based on Godunov's method; (ii) the other
uses the characteristic variables (or proxies thereof).

\subsubsection{Godunov's method} \label{Sec:Godunov}
We assume that on the outer side of the boundary we are given some ideal,
admissible state $\bsfU_i\upD=(\varrho\upD,\bsfM\upD,\sfE\upD)$ related to the
far-field conditions, which we call Dirichlet state. Then we consider the
Riemann problem $\partial_t\bv + \partial_x(\polf(\bv)\bn)) = \bzero$ with left
data $\bsfU$ and right data $\bsfU\upD$. Let $G(\bn,\bsfU,\bsfU\upD)$ denote
the value of the solution of the Riemann problem at $x=0$. The
post-processing then consists of setting
\begin{equation}
\bsfU\upP=G(\bn,\bsfU,\bsfU\upD). \label{Godunov_bc}
\end{equation}
Notice that $\bsfU\upP$ is automatically admissible. Since this
  operation may be expensive, we propose an alternative approach in
  the next section.

\subsubsection{Characteristic
  variables} \label{Sec:Characteritic_variables} We now assume that
the equation of state is described by the $\gamma$-law and propose a
technique based on characteristic variable. The method is loosely
based on \cite[\S3]{Hedstrom_1979} and has some similarities with
\cite[\S2.6]{Demkowicz_etal_1990}, but instead of working on
increments as in \citep{Demkowicz_etal_1990} we directly work on the
characteristic variables. The key results of this section are
\eqref{Supersonic_inflow}-\eqref{Subsonic_inflow}-\eqref{Subsonic_outflow}-\eqref{Supersonic_outflow}.

We define $\varrho\eqq \rho(\bsfU)$,
$\bsfM\eqq \bbm(\bsfU)$,
$\sfP\eqq p(\bsfU)= (\gamma-1) \varrho e(\bsfU)$,
$\sfS(\bsfU) \eqq \varrho^{-\gamma} \sfP$, and set
\begin{equation}
  \bsfV\eqq\varrho^{-1}\bsfM,\qquad
  \sfV_n\eqq\bsfV\SCAL\normal,\qquad
  \bsfV^{\perp}\eqq\bsfV - (\bsfV\SCAL\normal)\normal,\qquad
   a\eqq\sqrt{\gamma\sfP\varrho^{-1}},\qquad
\end{equation}
We start by recalling that, although characteristic variables do not
exist in general for the one-dimensional system
$\partial_t \bv + \partial_x(\polf(\bv)\bn)) = \bzero$, characteristic
variables and characteristic speeds do exist under the assumption that
the flow is locally isentropic. Making this assumption, we obtain
\begin{equation}
\underbrace{\begin{cases}
\lambda_1(\bsfU,\normal) := \sfV_n - a,\\
C_1(\bsfU,\normal)\eqq\sfV_n - \frac{2 a}{\gamma-1}%
\end{cases}}_{\text{mutiplicity 1}}
\qquad
\underbrace{\begin{cases}
\lambda_2(\bsfU,\normal) := \sfV_n,\\
\sfS(\bsfU), \ \bsfV^\perp
\end{cases}}_{\text{mutiplicity $d$}}
\qquad
\underbrace{\begin{cases}
\lambda_3(\bsfU,\normal) := \sfV_n + a,\\
C_3(\bsfU,\normal)\eqq\sfV_n + \frac{2 a}{\gamma-1}.
\end{cases}}_{\text{mutiplicity 1}}
\end{equation}
%(Notice in passing that $C_1$ is a Riemann invariant associated with
%$\lambda_3$ and $C_3$ is a Riemann invariant associated with
%$\lambda_1$).
Since the eigenvalues are ordered, we  distinguish four different cases:
\begin{align*}
 \text{(i)}  && \text{supersonic inflow} &&\sfV_n<0 \text{ and } a< |\sfV_n|
  &&\lambda_1(\bsfU,\normal)\le  \lambda_2(\bsfU,\normal)\le\lambda_3(\bsfU,\normal)<0 \\
 \text{(ii)} && \text{subsonic inflow}   &&\sfV_n<0  \text{ and } |\sfV_n|\le a
  &&\lambda_1(\bsfU,\normal)\le \lambda_2(\bsfU,\normal) <0 \le\lambda_3(\bsfU,\normal) \\
\text{(iii)} && \text{subsonic outflow}  &&0\le \sfV_n \text{ and } |\sfV_n|<a
  &&\lambda_1(\bsfU,\normal)<0\le \lambda_2(\bsfU,\normal)\le\lambda_3(\bsfU,\normal) \\
\text{(iii)} && \text{supersonic outflow}&& 0\le \sfV_n \text{ and } a\le |\sfV_n|
  &&0\le \lambda_1(\bsfU,\normal)\le \lambda_2(\bsfU,\normal)\le
    \lambda_3(\bsfU,\normal).
\end{align*}

We assume that on the outer side of the boundary we are given some Dirichlet
state $\bsfU\upD\eqq (\varrho\upD,\bsfM\upD,\sfE\upD)$.  We are going to postprocess $\bsfU$ such that the
characteristic variables of the post-processed state $\bsfU\upP$ associated
with in-coming eigenvalues match those of the prescribed Dirichlet state,
while leaving the out-going characteristics unchanged. More
precisely, the proposed strategy consists of seeking $\bsfU\upP$ so that the
following holds true:
\begin{equation}
  C_l(\bsfU\upP) {=}
  \begin{cases}
    C_l(\bsfU\upD) & \text{if $\lambda_l(\bsfU,\normal\upnr)<0$,}\\
    C_l(\bsfU)     & \text{if $0\le \lambda_l(\bsfU,\normal\upnr)$,}
  \end{cases}
  \qquad l\in\{1,3\},
\end{equation}
\begin{equation}
  S(\bsfU\upP) {=}
\begin{cases}
    S(\bsfU\upD)  & \text{if $\lambda_2(\bsfU,\normal\upnr)<0$,}\\
    S(\bsfU)      & \text{if $0\le \lambda_2(\bsfU,\normal\upnr)$},
  \end{cases}
 \quad\qquad
(\bsfV\upP)^\perp {=}
\begin{cases}
    (\bsfV\upD)^\perp & \text{if $\lambda_2(\bsfU,\normal\upnr)<0$,}\\
     \bsfV^\perp    & \text{if $0\le \lambda_2(\bsfU,\normal\upnr)$}.
  \end{cases}\hspace{-3mm}
\end{equation}
(Note that the condition
$\lambda_2(\bsfU,\normal\upnr)<0$ is equivalent to
$|\lambda_1(\bsfU,\normal\upnr)| > |\lambda_3(\bsfU,\normal\upnr)|$.)
Recalling that we assumed that the evolution of the flow field is locally
isentropic, we now solve the above system in
the four cases identified above.

\paragraph{Supersonic inflow condition} \label{Subsubsec:sup_in}
Assume that
$\lambda_1(\bsfU,\normal)\le\lambda_2(\bsfU,\normal)\le\lambda_3(\bsfU,\normal)<0$.
Since all the characteristics enter the computational domain, the
post-processing consists of replacing $\bsfU$ by $\bsfU\upD$:
\begin{equation} \label{Supersonic_inflow}
  \bsfU\upP = \bsfU\upD.
\end{equation}

\paragraph{Subsonic inflow boundary} \label{Subsubsec:sub_in}
Assume that $\lambda_1(\bsfU,\normal)\le \lambda_2(\bsfU,\normal)<0 \le
\lambda_3(\bsfU,\normal)$. Then,
$\bsfU\upP$ is obtained by solving the system
\begin{align}
  C_1(\bsfU\upP) = C_1(\bsfU\upD),\qquad
  S(\bsfU\upP) = S(\bsfU\upD),\qquad
  (\bsfV\upP)^\perp = (\bsfV\upD)^\perp,\qquad
  C_3(\bsfU\upP) = C_3(\bsfU).
\end{align}
Notice that, as expected, $d+1$ Dirichlet conditions are enforced. This
gives $\sfV_n\upP = \frac12 (C_1(\bsfU\upD)+ C_3(\bsfU))$,
$\sfP\upP = S(\bsfU\upD) (\varrho\upP)^\gamma$, and
\begin{align}
  a\upP = \tfrac{\gamma-1}{4} (C_3(\bsfU)-C_1(\bsfU\upD))=
 \tfrac{\gamma-1}{4}\sfV_n + \frac{a}{2}
- \tfrac{\gamma-1}{4}\sfV_n\upD +
\tfrac{a\upD}{2}.
\end{align}
Notice that $0<a\upP$ if $\gamma\le 3$ and
$\frac{\gamma-1}{2}\sfV_n\upD \le a\upD$, which is always the case for
realistic $\gamma$-laws. (Here
$\frac{\gamma-1}{2}\sfV_n\upD \le a\upD$ is an admissibility condition
on the Dirichlet data.) Using
$(a\upP)^2 =\gamma \sfP\upP (\varrho\upP)^{-1}$ with
$\sfP\upP = S(\bsfU\upD) (\varrho\upP)^\gamma$, the post-processing
for a subsonic inflow boundary condition consists of setting:
\begin{align} \label{Subsonic_inflow}
  \begin{cases}
    \begin{aligned}
      \rho\upP &= \left(\frac{1}{\gamma S(\bsfU\upD)}\left(\frac{\gamma-1}{4}
      \big(C_3(\bsfU)-C_1(\bsfU\upD)\big)\right)^2\right)^{\frac{1}{\gamma-1}},
      \\
      \bsfM\upP &= \rho\upP\left((\bsfV\upD)^\perp + \sfV_n\upP\bn\right), \quad \text{with}\quad
      \sfV_n\upP = \frac12 \big(C_1(\bsfU\upD)+ C_3(\bsfU)\big),
      \\
      E\upP &= \frac{1}{\gamma-1} \sfP\upP + \frac{1}{2} \frac{\|\bsfM\upP\|_{\ell^2}^2}{\varrho\upP},
      \quad \text{with}\quad \sfP\upP = S(\bsfU\upD) (\varrho\upP)^\gamma.
    \end{aligned}
  \end{cases}
\end{align}

\paragraph{Subsonic outflow boundary}  \label{Subsubsec:sub_out}
Assume that $\lambda_1(\bsfU,\normal)<0\le \lambda_2(\bsfU,\normal)
\le\lambda_3(\bsfU,\normal)$.  Then, $\bsfU\upP$ is obtained by solving the
system
\begin{align}
  C_1(\bsfU\upP) = C_1(\bsfU\upD),\qquad
  (\bsfV\upP)^\perp = \bsfV^\perp,\qquad
  S(\bsfU\upP) =  S(\bsfU),\qquad
  C_3(\bsfU\upP) = C_3(\bsfU).
\end{align}
Notice that, as expected, only one Dirichlet condition is enforced.
This gives $\sfV\upP_n = \frac12 (C_1(\bsfU\upD)+ C_3(\bsfU))$, $\sfP\upP =
S(\bsfU) (\varrho\upP)^\gamma$, and
\begin{align*}
  a\upP = \frac{\gamma-1}{4} (C_3(\bsfU)-C_1(\bsfU\upD))=
 \frac{\gamma-1}{4}\sfV_n + \frac{a}{2}  -\frac{\gamma-1}{4}\sfV_n\upD + \frac{a\upD}{2}.
\end{align*}
Here again we have $0<a\upP$ if $\gamma\le 3$ and if the admissibility condition
$\frac{\gamma-1}{2}\sfV_n\upD \le a\upD$ holds true.
Using
$(a\upP)^2 =\gamma \sfP\upP (\varrho\upP)^{-1}$ with
$\sfP\upP = S(\bsfU) (\varrho\upP)^\gamma$, the post-processing consists of setting:
\begin{align}
  \label{Subsonic_outflow}
  \begin{cases}
    \begin{aligned}
      \rho\upP &= \left( \frac{1}{\gamma S(\bsfU)}\left(\frac{\gamma-1}{4}
      \big(C_3(\bsfU)-C_1(\bsfU\upD)\big)\right)^2\right)^{\frac{1}{\gamma-1}},
      \\
      \bsfM\upP &= \rho\upP\left(\bsfV^\perp + \sfV_n\upP\bn\right), \quad \text{with}\quad
      \sfV_n\upP = \frac12 \big(C_1(\bsfU\upD+C_3(\bsfU))\big),
      \\
      E\upP &= \frac{1}{\gamma-1} \sfP\upP + \frac{1}{2} \frac{\|\bsfM\upP\|_{\ell^2}^2}{\varrho\upP},
      \quad \text{with}\quad \sfP\upP = S(\bsfU) (\varrho\upP)^\gamma.
    \end{aligned}
  \end{cases}
\end{align}

\paragraph{Supersonic outflow boundary condition} \label{Subsubsec:sup_out}
Assume
that $0\le \lambda_1(\bsfU,\normal)\le \lambda_2(\bsfU,\normal)\le
  \lambda_3(\bsfU,\normal)$. Since all the characteristics exit the domain,
the post-processing consists of not doing anything:
\begin{equation}
\bsfU\upP = \bsfU. \label{Supersonic_outflow}
\end{equation}

\begin{remark}[Literature]
  It is established in \cite[\S3]{Hedstrom_1979} that appropriate
  non-reflecting boundary conditions for the one-dimensional Riemann
  problem $\partial_t \bv + \partial_n(\polf(\bv)\bn) = \bzero$ are
  $\partial_t C_1(\bsfU) + \frac{a}{\gamma-1} \partial_t s(\bsfU) =0$
  in the subsonic outflow situation and $\partial_t C_1(\bsfU) =0$
  plus $\partial_t s(\bsfU) =0$ (and $\partial_t \bsfV^\perp =\bzero$)
  in the subsonic inflow situation, where
  $s(\bsfU) = \log(e(\bsfU)^{\frac{1}{\gamma-1}}\varrho^{-1})$ is the
  specific entropy. Assuming that the Dirichlet data are
  time-independent, these conditions can be rewritten
  $\partial_t (C_3(\bsfU) - C_3(\bsfU\upD)) +
  \frac{a}{\gamma-1} \partial_t (s(\bsfU) - s(\bsfU\upD)) =0$ and so on.
  Then, under the assumption that the flow is locally isentropic at
  the boundary, these conditions exactly coincide with what is
  proposed above (notice that in this case $\sfS=(\gamma-1)e^{s}$).
  Indeed, by setting $C_1(\bsfU)_{|t=0}= C_1(\bsfU\upD)$,
  $S(\bsfU)_{|t=0}=S(\bsfU\upD)$, (and
  $\bsfV^\perp(\bsfU)_{|t=0}=\bsfV^\perp(\bsfU\upD)$), the condition
  $\partial_t C_1(\bsfU) =0$ yields $C_3(\bsfU\upP)=C_3(\bsfU\upD)$
  for the subsonic outflow situation (see \eqref{Subsonic_outflow}) and it yields
  $C_1(\bsfU\upP)=C_1(\bsfU\upD)$, $\sfS(\bsfU\upP)=\sfS(\bsfU\upD)$,
  (and $\bsfV^\perp(\bsfU\upP)_{|t=0}=\bsfV^\perp(\bsfU\upD)$) for the
  subsonic inflow situation (see \eqref{Subsonic_inflow}).

  No claim is made here about the optimality of the proposed
  artificial boundary conditions, in particular in regards to their
  absorbing properties. We refer the reader to \cite{Fosso_2012} and
  the abundant literature cited therein for other approaches used in
  the finite difference context.
\end{remark}

\subsubsection{Conservation and admissibility}
We collect in this subsection conservation and admissibility
properties of the post-processing method proposed above.
\begin{lemma}[Slip condition]\label{Lem:slip_bc}%
Let $i\in
\calV\upbnd\los$, let $\bsfU_i\in \calA$, and let $\bsfU_i\upP$ as defined in \eqref{slip_bc}.
  \begin{enumerate}[font=\upshape,label=(\roman*)]
    \item Then $\bsfU_i\upP$ is also admissible, meaning
$\bsfU_i\upP\in \calA$.
    \item Assume also that the equation of state derives from an
entropy $s$.  Then $s(\bsfU_i\upP) \ge s(\bsfU_i)$.
    \item For all $i\in \calV\upbnd\los{\setminus}\calV\upbnd\lonr$,
the mass flux and the total energy flux of the postprocessed solution
at $i$ is zero (\ie $\rho(\polf(\bsfU_i\upP)\bn_i)=0$ and
$E(\polf(\bsfU_i\upP)\bn_i)=0$).
  \end{enumerate}
\end{lemma}
\begin{proof}
See Lemma~\ref{Lem:slip_bc_bis} in Appendix~\ref{sec:BCmath}.
\end{proof}

\begin{lemma}[Non-reflecting
  condition]\label{Lem:subsonic_supersonic_bcs}%
  Let $i\in\calV\upbnd\lonr$ and let $\bsfU_i\in\calA$. Let
  $\bsfU_i\upP$ be defined either by \eqref{Godunov_bc} or by one the conditions
  \eqref{Supersonic_inflow}, \eqref{Subsonic_inflow},
  \eqref{Subsonic_outflow}, \eqref{Supersonic_outflow} (with the
  $\gamma$-law assumption, $\gamma\in (1,3]$, and the
  admissibility condition on the Dirichlet data $\frac{\gamma-1}{2}\sfV\upD \le a\upD$).  Then
  $\bsfU_i\upP\in \calA$.
\end{lemma}
\begin{proof}
  Direct consequence of the definitions \eqref{Supersonic_inflow},
  \eqref{Subsonic_inflow}, \eqref{Subsonic_outflow},
  \eqref{Supersonic_outflow}.
\end{proof}
We obtain the following result by combining
Lemma~\ref{Lem:subsonic_supersonic_bcs} and Lemma~\ref{Lem:slip_bc}.
\begin{corollary}[Admissibility]\label{Cor:admissibility_after_postproc_io_bcs}%
  The solution obtained at the end the RKSSP(3,3) algorithm after
  limiting and post-processing is admissible.
\end{corollary}

\begin{lemma}[Global conservation]\label{Lem:global_conservation}%
Assume that $\calV\los\upbnd =\calV\upbnd$ and $\bsfU^n$
satisfies the slip boundary condition (\ie $\bsfM_i^n\SCAL \bn_i=0$
for all $i\in\calV\upbnd$).  Then the solution obtained at the end the
RKSSP(3,3) algorithm after limiting and post-processing, say
$\bsfU^{n+1}$, satisfies
$\sum_{j\in\calV} m_j \varrho_j\upnp = \sum_{j\in\calV} m_j
\varrho_j\upn$
and
$\sum_{j\in\calV} m_j \sfE_j\upnp = \sum_{j\in\calV} m_j \sfE_j\upn$.
\end{lemma}

\begin{proof}
See Lemma~\ref{Lem:global_conservation_bis} in Appendix~\ref{sec:BCmath}.
\end{proof}

\section{Discretization of the parabolic problem}
\label{Sec:discretization_hyperbolic_limit}%

We describe in this section key details involved in the approximation
of the parabolic operator $S\loP$ (see
\eqref{parabolic_simplified}). Given an admissible field
$\bu_h^n=\sum_{i\in\calV} \bsfU_i^n\varphi_i$ at some time $t^n$ and
given some time step size $\dt_n$, we want to construct an
approximation of the solution to \eqref{parabolic_simplified} at
$t^{n+1}\eqq t^n+\dt_n$, say
$\bu_h^{n+1}=\sum_{i\in\calV}\bsfU_i^{n+1}\varphi_i$. Referring to
\eqref{continuous_Strang}, we recall that the time step used in the
parabolic problem is twice that used in the hyperbolic step, but to
simplify the notation we still call the time step size $\dt_n$ in this
entire section.  The important point here is that positivity of the
internal energy and conservation must be guaranteed. The steps
described in
\S\ref{Subsec:parabolic_update}--\S\ref{Subsec:internal_energy_update} are
summarised in Algorithm~\ref{alg:parabolic}
in Appendix~\ref{Sec:appendix_parabolic_step}.

\subsection{Density, velocity, and momentum update} \label{Subsec:parabolic_update}
Recalling that in \eqref{parabolic_simplified}
the density does not change in time, we set
\begin{equation}
\varrho_i^{n+1} = \varrho_i^n,\qquad \forall i\in\calV.
\end{equation}

We use the Crank-Nicolson technique for the time stepping in
\eqref{parabolic_momentum_simplified}.  The approximation in space is
done by using the Galerkin technique with the lumped mass matrix.  Let
$\{\be_k\}_{k\in\intset{1}{d}}$ denote the canonical Cartesian basis
of $\Real^d$, and for all $\bsfX\in \Real^d$, let
$\{\sfX_k\}_{k\in\intset{1}{d}}$ denote the Cartesian coordinates of
$\bsfX$.  The discrete problem then consists of seeking
$\bv_h^{n+\frac{1}{2}}\eqq \sum_{i\in\calV}
\sfV_{i,k}^{n+\frac{1}{2}}\varphi_i \be _k$ so that
\begin{equation}
\label{mt_parabolic_discrete} \varrho^{n}_i m_i
\sfV_{i,k}^{n+\frac{1}{2}} + \tfrac12 \dt_n
a(\bv_h^{n+\frac{1}{2}},\varphi_i\be_k) = m_i \sfM_{i,k}^{n} +
\tfrac12 \dt_n m_i \sfF_{i,k}^{n+\frac12},\qquad \forall i\in \calV, \
\forall k\in\intset{1}{d},
\end{equation}
with $ a(\bv,\bw) \eqq \int_\Dom\pols(\bv){:}\pole(\bw) \diff x$, and
$\sfF_{i,k}^{n+\frac12}\eqq \int_\Dom \varphi_i(\bx)
\be_k\SCAL\bef(\bx,t^{n+\frac12}) \diff x$.
We consider three types of boundary conditions on the velocity: (i)
the no-slip condition $\bv_{|\front}=\bzero$; (ii) the slip condition
$\bv\SCAL\bn=0$ and $\bn\CROSS(\pols(\bv)\bn)=\bzero$; (iii) and the
homogeneous Neumann boundary condition
$\pols(\bv)\bn_{|\front} = \bzero$.  The assembling is done in two
steps: (1) one first assembles the system with homogeneous Neumann
boundary conditions; (2) the correct boundary conditions are
implemented by post-processing the linear system.  After the essential
boundary conditions are enforced, and once this linear system is
solved (see \S\ref{Sec:linear_algebra}), the velocity and the momentum
are updated as follows:
\begin{equation}
  \label{vel_parabolic_discrete} \bsfV_i^{n+1} \eqq 2
\bsfV_i^{n+\frac{1}{2}} - \bsfV_i^{n},\qquad \bsfM_i^{n+1} \eqq
\varrho^{n+1}_i\bsfV_i^{n+1},\qquad \forall i\in\calV.
  \end{equation}

\subsection{Internal energy and total energy update}
\label{Subsec:internal_energy_update} The second step of the parabolic
solve consists of computing the specific internal energy (or
temperature) and updating the total energy. Before going into the
details, we discuss the boundary condition. We consider two types of
boundary conditions: (i) Dirichlet:
$T_{|\front}=T\upbnd$; (ii) homogeneous Neumann:
$\partial_n T = 0$. The assembling is done in two steps: (1) one first
assembles the system with homogeneous Neumann boundary conditions; (2)
the Dirichlet boundary conditions are implemented by post-processing
the linear system.

Here again we use the Crank-Nicolson technique for the time stepping
in \eqref{parabolic_internal_energy_simplified}, and the approximation
in space is done by using the Galerkin technique with the lumped mass matrix.
First, we compute the rate of specific internal energy production
caused by the viscous stress
\begin{equation}
  \label{def_Ki_nplusone}
  \sfK_i\upnph\eqq \frac{1}{m_i}\int_\Dom \pols(\bv_h\upnph) {:}
  \pole(\bv_h\upnph) \varphi_i \diff x,
  \qquad \forall i\in\calV.
\end{equation}
Second, we compute the specific internal energy of $\bu_h^n$ by setting set
$\sfe_i^{n}\eqq (\varrho^{n}_i)^{-1} \sfE_i^n -
\tfrac12\|\bsfV_i^{n}\|_{\ell^2}^2$ for all $i\in\calV$. Note that
$\sfe_i^{n}>0$ since we assumed that $\bu_h^n$ is admissible. Then, we seek
the specific internal energy, $e_h{\upnph}=
\sum_{i\in\calV}\sfe_i{\upnph}\varphi_i$, such that the following holds:
\begin{align}
  \label{high_int_energy_CN}
  & m_i \varrho_i^{n}(\sfe_i{\upnph} - \sfe_i^{n})+\tfrac12\dt_n
  b(e_h\upnph,\varphi_i)
  = \tfrac12 \dt_n m_i\sfK_i^{n+\frac{1}{2}}, \qquad \forall i\in \calV,
\end{align}
with $b(e,w) \eqq c_v^{-1}\kappa \int_\Dom \GRAD e\SCAL \GRAD w \diff x$.
Finally, we update the internal energy and the total energy:
\begin{equation}\label{energy_parabolic_discrete}
  \sfe_i\upnp = 2\sfe_i{\upnph} - \sfe_i^n, \qquad \sfE_i^{n+1} =
  \varrho_i^{n+1}\sfe_i^{n+1} + \tfrac12\varrho^{n}_i
  \|\bsfV_i^{n+1}\|_{\ell^2}^2,\qquad
  \forall i\in \calV.
\end{equation}

The algorithm is second-order accurate in time, but there
is no guarantee that the internal energy stays positive
because the Crank-Nicolson scheme is not positivity preserving. If it
happens that $\min_{i\in\calV} \sfe_i^{n+1}<0$, then limiting must be
applied. This is done as described in
\citep[\S5.3]{Guermond_Maier_Popov_Tomas_CMAME_2020}. We briefly
recall the technique. We compute a first-order,
invariant-domain-preserving (\ie positive) update of the internal
energy by seeking $e_h\upLnp=\sum_{i\in\calV} \sfe\upLnp \varphi_i$ so
that the following holds:
\begin{equation} \label{low_order_update_internal_energy}
  m_i \varrho_i^{n}(\sfe_i{\upLnp} - \sfe_i^{n})+ \dt_n
  b(e_h\upLnp,\varphi_i)
  = \tfrac12 \dt_n m_i\sfK_i^{n+\frac{1}{2}}, \qquad \forall i\in \calV.
\end{equation}
Let $e_h\upHnp\eqq \sum_{i\in\calV} \sfe\upHnp \varphi_i$ be the solution to \eqref{high_int_energy_CN}.
Subtracting \eqref{low_order_update_internal_energy} from \eqref{high_int_energy_CN} yields
\begin{align}
  &m_i \varrho_i^{n}(\sfe_i\upHnp - \sfe_i\upLnp) =
  \sum_{j\in\calI^*(i)}A_{ij},\\
&A_{ij}\eqq -\tfrac12\dt_n b(\varphi_j,\varphi_i) (\sfe_j\upHnp+
\sfe_j^{n} - 2 \sfe_j\upLnp -\sfe_i\upHnp -\sfe_i^{n} +2
\sfe_i\upLnp). \label{Aij_for_internal_energy}
\end{align}
The standard FCT limiting can be applied by setting
$m_i \varrho_i^{n}(\sfe_i\upHnp - \sfe_i\upLnp) =
\sum_{j\in\calI^*(i)} \limiter_{ij}A_{ij}$, see \eg
\citep{Boris_books_JCP_1973,Zalesak_1979,KuzminLoehnerTurek2004}. The
reader is referred to
\citep[\S5.3]{Guermond_Maier_Popov_Tomas_CMAME_2020} for the
computation of $\limiter_{ij}$. Once the internal energy is updated, the total energy can also be
updated by setting
\begin{equation}
  \label{energy_parabolic_discrete}
  \sfE_i^{n+1} = \varrho_i^{n+1}\sfe_i^{n+1} +
  \tfrac12\varrho^{n}_i \|\bsfV_i^{n+1}\|_{\ell^2}^2,\qquad \forall i\in \calV.
\end{equation}

The main properties of the method presented here are collected in the following statement.
\begin{lemma}[Positivity and conservation]
  \label{Lemma:parabolic_conservation} Let $\bsfU^{n}$ be an
admissible state. Let $\bsfU^{n+1}$ be the state constructed in the
parabolic substep.  Then, $\bsfU^{n+1}$ is an admissible state, \ie
$\bsfU_i^{n+1}\in\calA$ for all $i\in \calV$ and all $\dt_n$, and the
following holds for all $i\in\calV$ and all $\dt_n$:
  \begin{equation} \varrho_i^{n+1}= \varrho_i^n > 0, \qquad \qquad
\min_{j\in\calV} \sfe_j^{n+1}\ge \min_{j\in\calV} \sfe_j^{n} > 0,
\qquad \forall i\in\calV,
 \end{equation} Assume that the slip or the no-slip boundary condition
is enforced on the velocity everywhere on $\front$.  Assume that the
homogeneous Neumann boundary condition is enforced on the internal
energy everywhere on $\front$.  Then the following holds true for all
$\dt_n$:
 \begin{equation} \sum_{i\in\calV} m_i \sfE_i^{n+1} = \sum_{i\in\calV}
m_i \sfE_i^{n} + \sum_{i\in\calV} \dt m_i\bsfF_i^{n+\frac12}\SCAL
\bsfV_i^{n+\frac{1}{2}}.
      \label{eq2:lem:parabolic_energy}
  \end{equation}
\end{lemma}
\begin{proof}
  See \citep[Thm.~5.5]{Guermond_Maier_Popov_Tomas_CMAME_2020}.
\end{proof}

%\subsection{Boundary conditions for Navier-Stokes} \label{Sec:bcs_for_NS}

%

\subsection{Matrix-free geometric multigrid} \label{Sec:linear_algebra}

For the \texttt{deal.II}-based finite element implementation discussed
in \S\ref{Sec:validation_and_benchmarks} (see
\citep{maier2021ryujin}), the linear
systems~\eqref{mt_parabolic_discrete} and \eqref{high_int_energy_CN}
are solved iteratively using matrix-free operator
evaluations~\citep{Kronbichler2012}. The \emph{action} of the matrix
on a vector is implemented by redundantly computing the information
contained in the stencil on the fly through the finite element
integrals. On modern hardware, computations are less expensive than
data movement~\citep{Fischer2020,Kronbichler2018}, making a
matrix-free evaluation several times faster than a sparse-matrix
vector product due to the reduced memory traffic. This is especially
relevant for the vector-valued velocity: the block-structured
matrix-vector multiplication couples the velocity components, whereas
the cell-wise integrals in the matrix-free evaluation only couple
these components at the quadrature-point level without additional data
transfer. Besides yielding faster matrix-vector products, the
matrix-free approach also avoids the cost of matrix assembly. The
higher arithmetic intensity in the matrix-free evaluation is leveraged
by cross-element SIMD vectorization of the relevant
operations~\citep{Kronbichler2012,Kronbichler2019b}.

Regarding the selection of the iterative solvers, we choose among two
options. When the cell-based Reynolds number using the diameter of the
cells as length scale is above one, the mass matrix contribution
in~\eqref{parabolic_momentum_simplified} dominates due to the limit imposed
by the hyperbolic problem on the time step. This argument also holds for
the internal energy equation provided the Prandtl number is not too large.
In that case, the diagonal mass matrix is an optimal preconditioner for the
conjugate gradient algorithm, giving iteration counts below 10. The solver
can be further tuned for data locality on high-performance CPUs as
described in \cite{Kronbichler2021}. If the mesh becomes very fine for a
given viscosity level, the elliptic contributions become dominant instead.
Then, we equip the conjugate gradient solver with a geometric multigrid
preconditioner that steps into successively coarser levels
$l\in\intset{1}{L}$ where $l=1$ refers to the coarsest mesh and $l=L$
refers to the finest mesh, see \cite{Kronbichler2018} and
\cite{Clevenger2021} for details on the parallel scaling and performance. A
Chebyshev iteration of degree three (\ie three matrix-vector products)
around the point-Jacobi method is used for pre- and post-smoothing, using
parameters to smooth in a range $[0.08\hat{\lambda}_\text{max},
1.2\hat{\lambda}_\text{max}]$ with the maximal eigenvalue estimate
$\hat{\lambda}_{\text{max}}$ computed every four time steps by a Lanczos
iteration with 12 iterations. In order to improve parallel scaling of the
V-cycle, we limit the coarsening to the mesh levels
$l\in\intset{L_\text{min}}{L}$, with $L_\text{min}$ such that the
cell-based Reynolds number exceeds unity. On the coarse level, a few
iterations suffice to solve the system accurately, which is done by a
Chebyshev iteration aiming to reduce the residual by a factor of $10^3$
according to the a-priori error estimate. The multigrid solver typically
takes 3 to 5 iterations to converge. To increase the throughput, the
multigrid V-cycle is run in single
precision~\citep{Gropp2000,Kronbichler2019}.

%%%%%%%%%%%%%%%%%%%%%%%%%%%%%%%%%%%%%%%%%%%%%%%%%%%%%%%%%%%%%%%%%%%%%%%%%%%%%%%%
%%%%%%%%%%%%%%%%%%%%%%%%%%%%%%%%%%%%%%%%%%%%%%%%%%%%%%%%%%%%%%%%%%%%%%%%%%%%%%%%
%%%%%%%%%%%%%%%%%%%%%%%%%%%%%%%%%%%%%%%%%%%%%%%%%%%%%%%%%%%%%%%%%%%%%%%%%%%%%%%%

\section{Verification, validation and
  benchmarks} \label{Sec:validation_and_benchmarks} The method
described above has been implemented in a code called \texttt{ryujin}
which is freely available
online\footnote{\url{https://github.com/conservation-laws/ryujin}}
\citep{maier2021ryujin} under a permissible open source
license.\footnote{\url{https://spdx.org/licenses/MIT.html}} This code
is based on the finite element library \texttt{deal.II}
\citep{dealII92,dealIIcanonical} and uses mapped continuous $\polQ_1$
finite elements.
We discuss in this section a number of verification and benchmark
configurations to demonstrate that the algorithm described herein is
robust, accurate and scalable. In particular, we use a 2D shocktube
configuration proposed by \cite{Daru_Tenaud_2000, Daru_Tenaud_2009} to
demonstrate grid convergence; see
Section~\ref{subse:shocktube}. Following \citep{Daru_Tenaud_2020} and
to allow for rigorous quantitative comparisons with other research
codes, test vectors obtained from our computation of extrema of the
skin friction coefficient are made freely available
\citep{testvectors_2021}. We then demonstrate that the method can
reliably predict pressure coefficients on the well-studied
supercritical airfoil Onera OAT15a \citep{Deck_2005} in the
supercritical regime at Mach 0.73 in three-dimensions; see
Section~\ref{subse:airfoil}. Finally, a series of synthetic benchmarks
are presented to assess the performance of the compute-kernel and the
strong and weak scalability of our implementation
(Section~\ref{subse:scaling}).

\subsection{Verification}
\label{Sec:validation}
The hyperbolic kernel and the parabolic kernel of \texttt{ryujin} have
been been verified on various analytical solutions to ensure the
correctness of the implementation; see \eg
\citep{MaierTomas2020,maier2020massively}. We now demonstrate the
correctness of the full algorithm in 2D on a viscous shockwave problem
that has an exact solution described in \cite{Becker_1922}, see also
\citep[\S7.2]{Guermond_Maier_Popov_Tomas_CMAME_2020} and
\cite{Johnson_JFM_2013}.
With the parameters given in \citep[\S7.2;
Eqs.~(7.1)-(7.4)]{Guermond_Maier_Popov_Tomas_CMAME_2020}, we
approximate the Becker solution on a mesh sequence
$\calH$ of successively refined uniform meshes. The computational
domain is the unit square with Dirichlet boundary conditions on the
left and right boundaries, and periodic boundary conditions on the
upper and lower boundaries. We slightly deviate from
\citep{Guermond_Maier_Popov_Tomas_CMAME_2020} by choosing the velocity of the Galilean frame
to be $v_\infty=0.125$ and choosing the CFL
number 0.3. The source code and the parameter files are archived on
the online platform \emph{Zenodo}; see
\citep{maier2021ryujin,testvectors_2021}. As evidenced in
Table~\ref{tab:becker-validation}, we observe second-order convergence
in space and time in the maximum norm.
%%%%%%%%%%%%%%%%%%%%%%%%%%%%%%%%%%%%%%%%%%%%%%%%%%%%%%%%%%%%%%%%%%%%%%%%%%%%%%%%
\begin{table}[H]\small \centering
  \nprounddigits{2}
  \begin{tabular}{rcccccc}
    \toprule
    gridpoints  & $\delta_1$  & rate & $\delta_2$  & rate & $\delta_\infty$ & rate \\[0.5em]
    4225        & \numprint{4.68183e-3}   & --   & \numprint{5.96283e-3}  & --   & \numprint{1.28885e-2}    & --   \\
    16641       & \numprint{3.87919e-4}   & 3.59 & \numprint{9.33675e-5}  & 2.68 & \numprint{2.89148e-3}    & 2.16 \\
    66049       & \numprint{8.7326e-5}    & 2.15 & \numprint{2.34604e-5}  & 1.99 & \numprint{8.0178e-4}     & 1.85 \\
    263169      & \numprint{2.16887e-5}   & 2.01 & \numprint{5.89731e-5}  & 1.99 & \numprint{2.09718e-4}    & 1.93 \\
    \bottomrule
  \end{tabular}
  \caption{Convergence study of the approximation of the Becker
    solution \citep{Becker_1922} on a mesh sequence $\calH$ of successively
    refined uniform meshes with CFL=0.3. The columns labeled
    ``$\delta_p$'' show the $L^p$-norm of the consolidated  error
    \citep[Eq.~(7.4)]{Guermond_Maier_Popov_Tomas_CMAME_2020}.}%
  \label{tab:becker-validation}%
\end{table}

\subsection{Non-reflecting conditions} \label{Subsec:non_reflecting_boundary_conditions} We now evaluate the
performance of the non-reflecting boundary conditions described in
\S\ref{Subsec:non_reflecting_bcs} by using a series of tests proposed
in \cite{Fosso_2012}. We solve the Euler equations
in the domain $\Dom\eqq [-1,1]^2$ with the initial data
\begin{align}
\rho_0(\bx) &= \rho_\infty, \\
\bv_0(\bx) &= \bv_\infty + \overline v_\infty r_0^{-1} \psi(\|\bx - \bx_0\|_{\ell^2})
\polA (\bx - \bx_0),\\
p_0(\bx) & = p_\infty - \tfrac12 \rho_\infty\overline v_\infty^2
\psi^2(\|\bx - \bx_0\|_{\ell^2}),
\end{align}
with $\bv_\infty=(v_\infty,0)\tr$,
$\psi(r) \eqq e^{\frac12(1-\frac{r^2}{r_0^2})}$ and
$\polA \eqq \big(\begin{smallmatrix}0 & -1 \\1 &
  0\end{smallmatrix}\big)$.
We take $v_\infty=1$ and $\rho_\infty=1$. We use the velocity
perturbation $\overline v_\infty$ and the Mach number $M_\infty$ as
parameters. The pressure $p_\infty$ is defined to be
$\frac{\rho_\infty}{\gamma} a_\infty^2$ where
$a_\infty \eqq \frac{v_\infty}{M_\infty}$ is the sound speed. Four
cases are considered: (i) $M_\infty=0.5$, $\overline v_\infty=0.75$;
(ii) $M_\infty=0.5$, $\overline v_\infty=0.25$; (iii) $M_\infty=0.05$,
$\overline v_\infty=0.75$; (iv) $M_\infty=0.05$,
$\overline v_\infty=0.25$.
The simulations are done on a $80\CROSS 80$ mesh.
We enforce the non-reflecting boundary condition on the four sides of
the domain. We test the Riemann solution technique described in
\S\ref{Sec:Godunov} and the method based on the characteristic
variables described in \S\ref{Sec:Characteritic_variables}.
All the tests are done with $\text{CFL}=0.75$.
\begin{figure}[h]
\subfloat[$M_\infty=0.5$, $\overline v_\infty=0.75$. Left: $\delta_1(t)$; Right: $\delta_2(t)$]{%
  \includegraphics[width=0.24\textwidth]{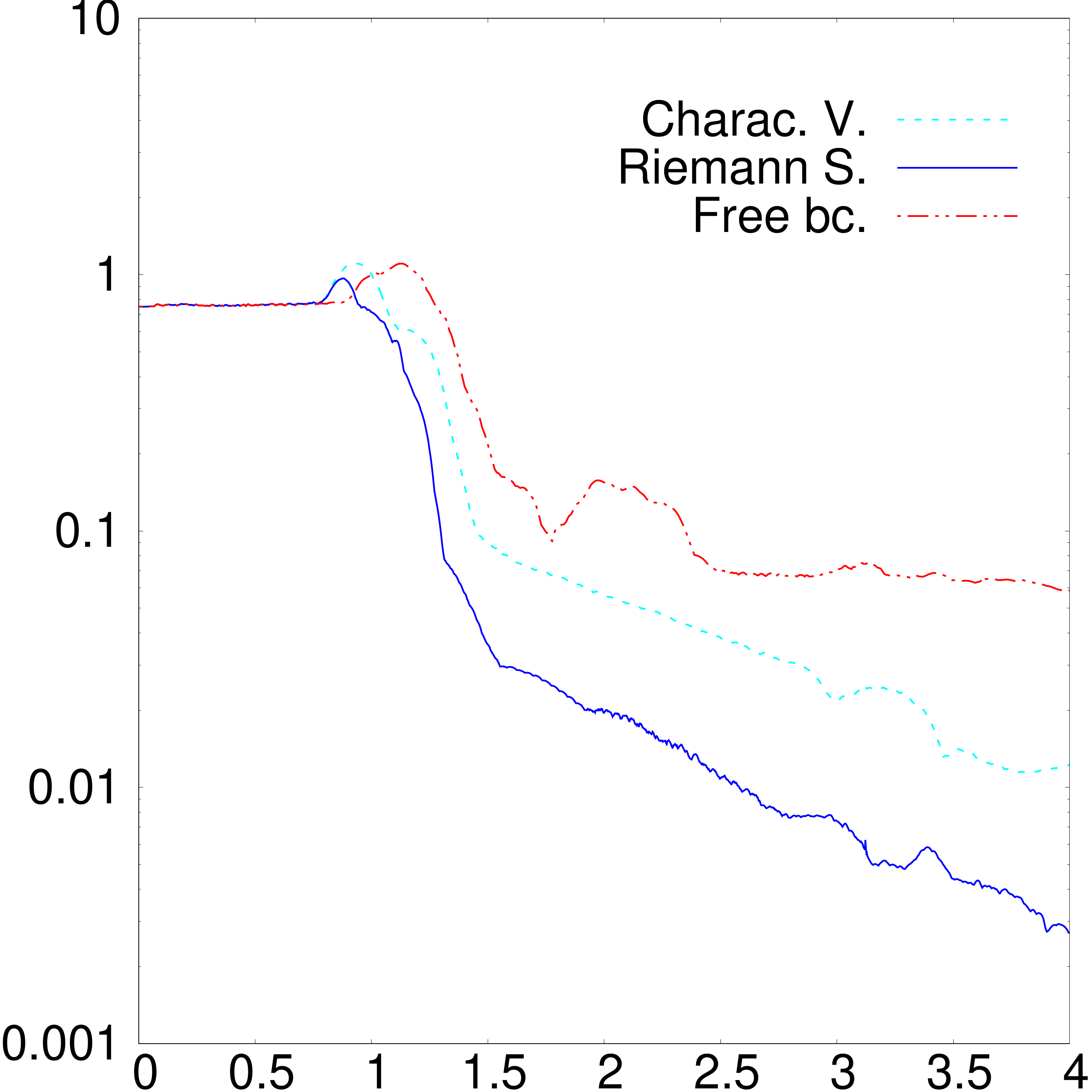}\hfil
\includegraphics[width=0.24\textwidth]{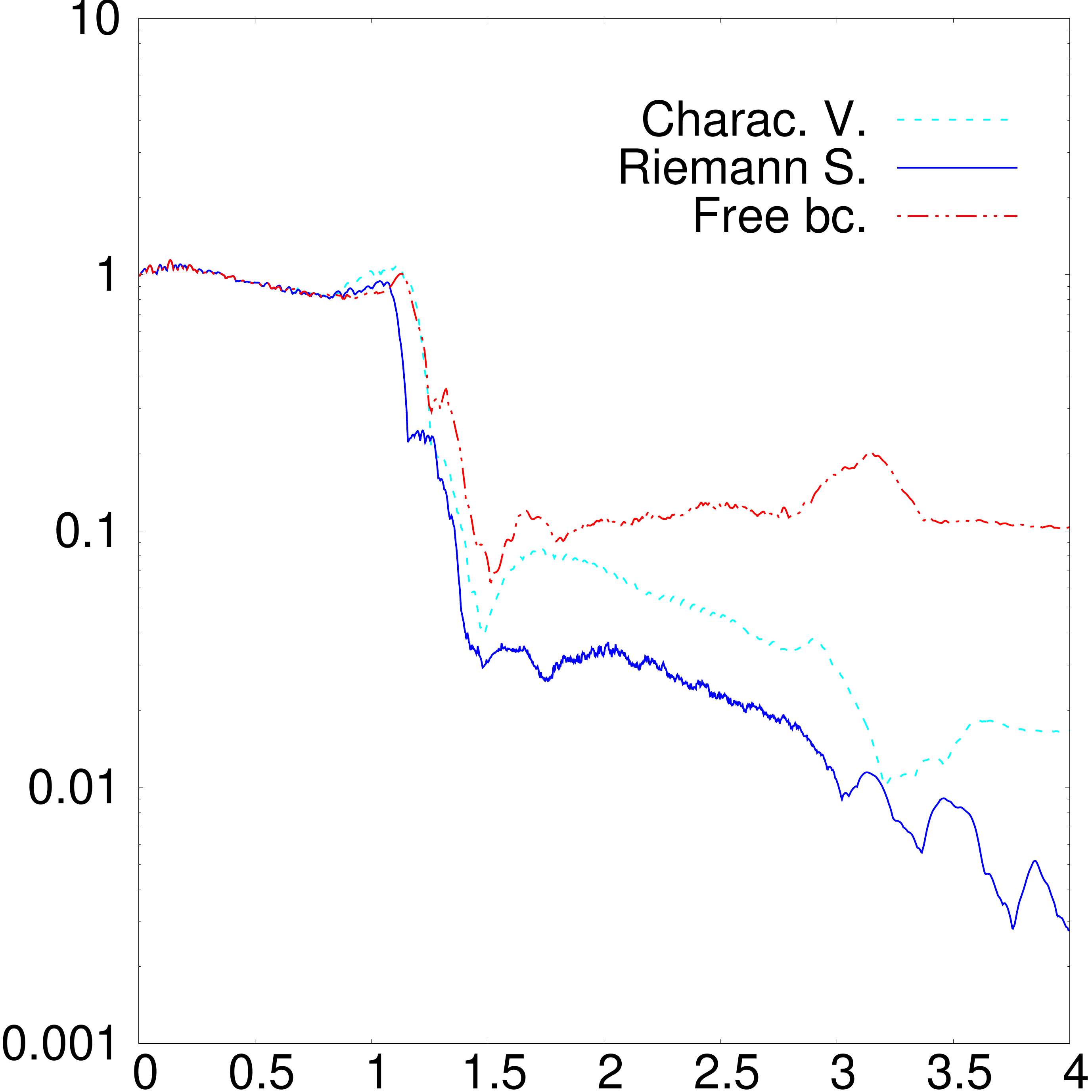}}\hfil
\subfloat[$M_\infty=0.5$, $\overline v_\infty=0.25$. Left: $\delta_1(t)$; Right: $\delta_2(t)$]{%
  \includegraphics[width=0.24\textwidth]{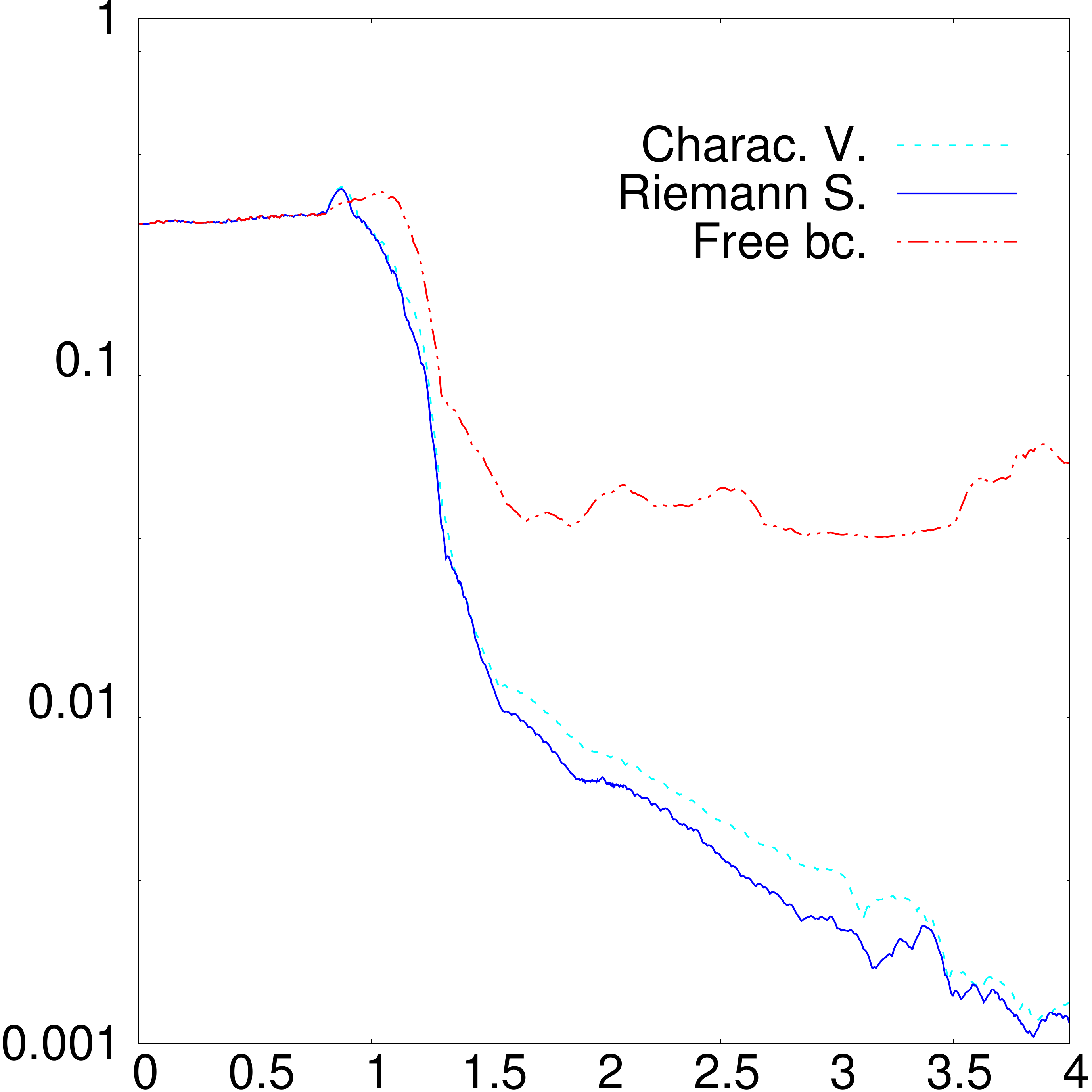}\hfil
\includegraphics[width=0.24\textwidth]{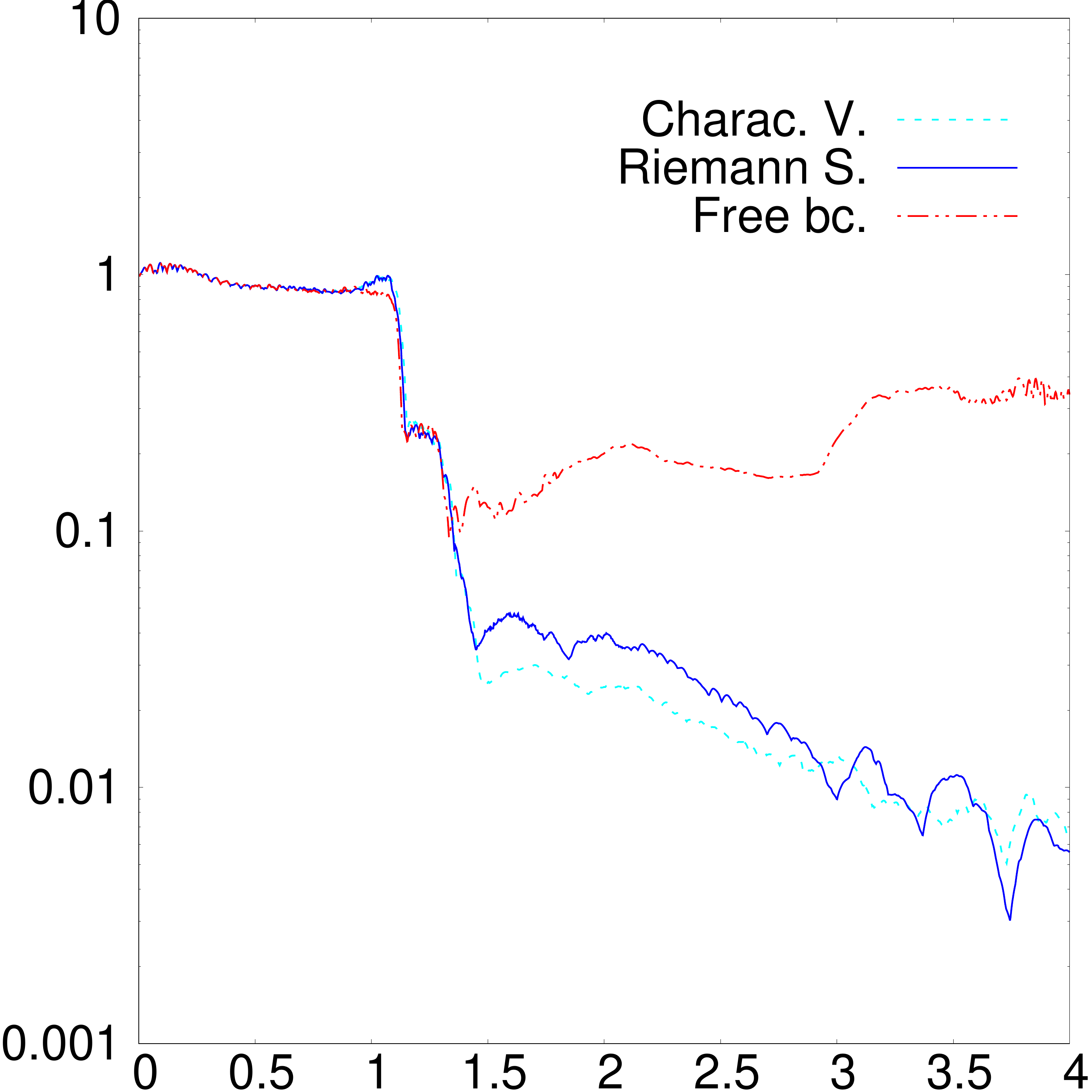}} \\
\subfloat[$M_\infty=0.05$, $\overline v_\infty=0.75$. Left: $\delta_1(t)$; Right: $\delta_2(t)$]{%
  \includegraphics[width=0.24\textwidth]{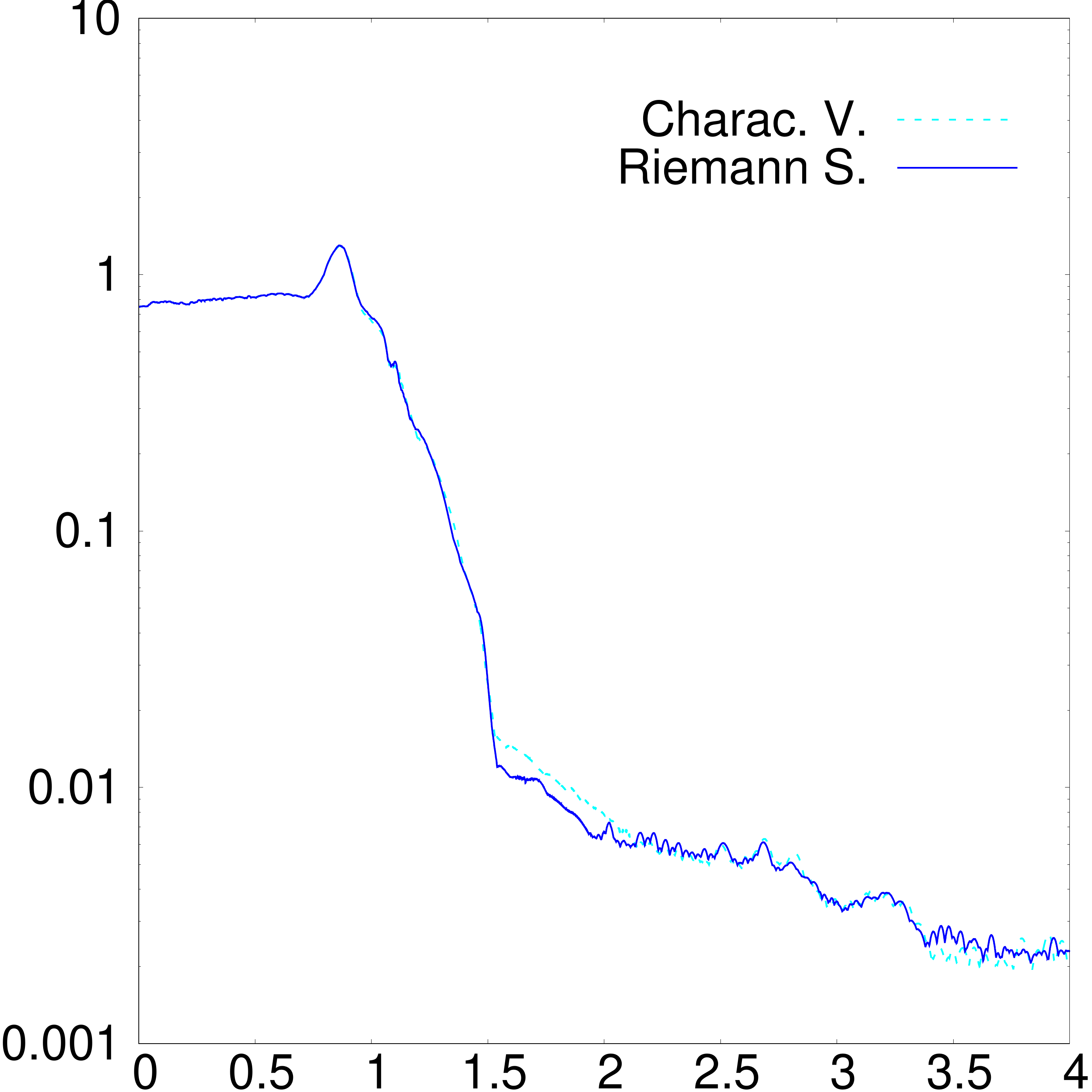} \hfil
\includegraphics[width=0.24\textwidth]{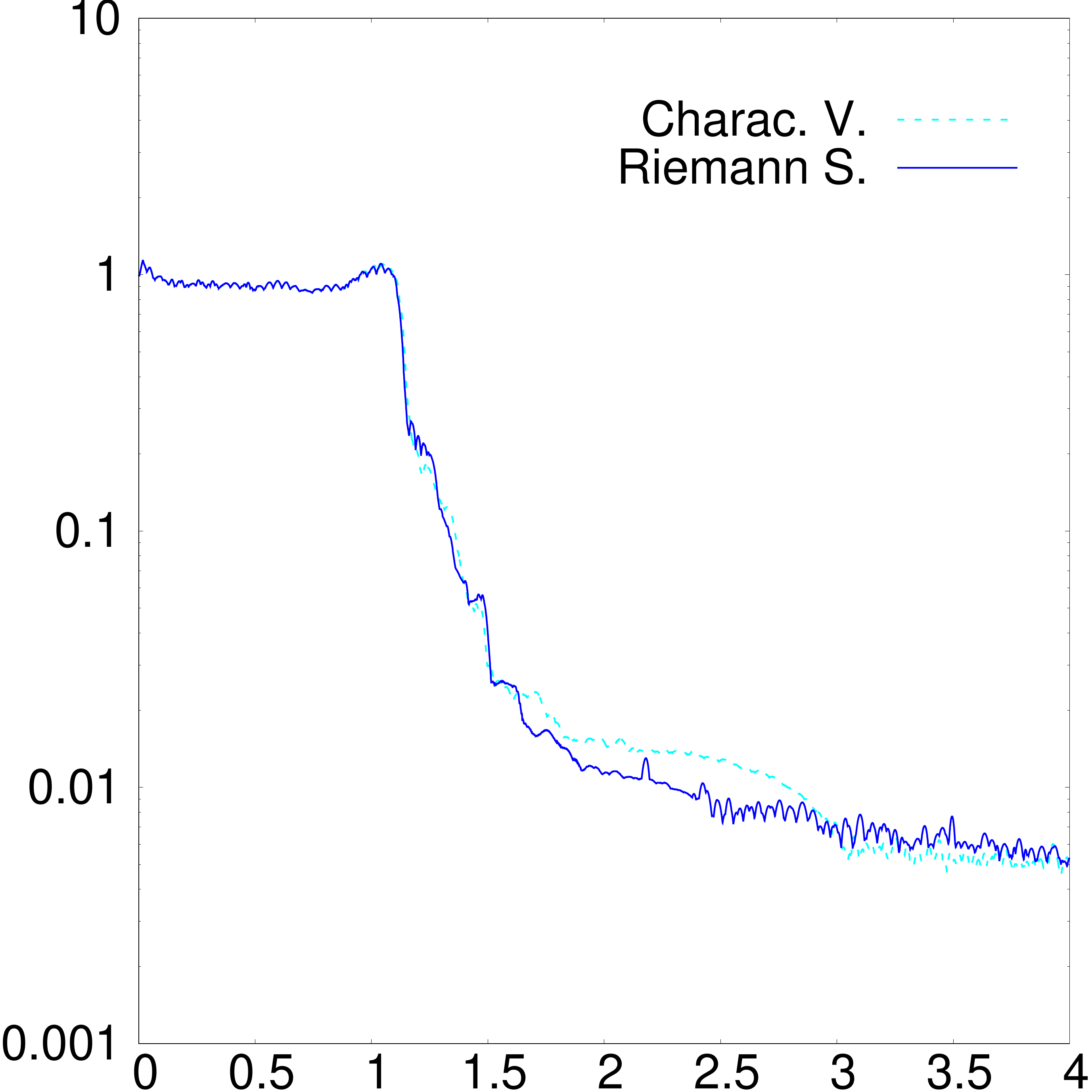}} \hfil
\subfloat[$M_\infty=0.05$, $\overline v_\infty=0.25$. Left: $\delta_1(t)$; Right: $\delta_2(t)$]{%
  \includegraphics[width=0.24\textwidth]{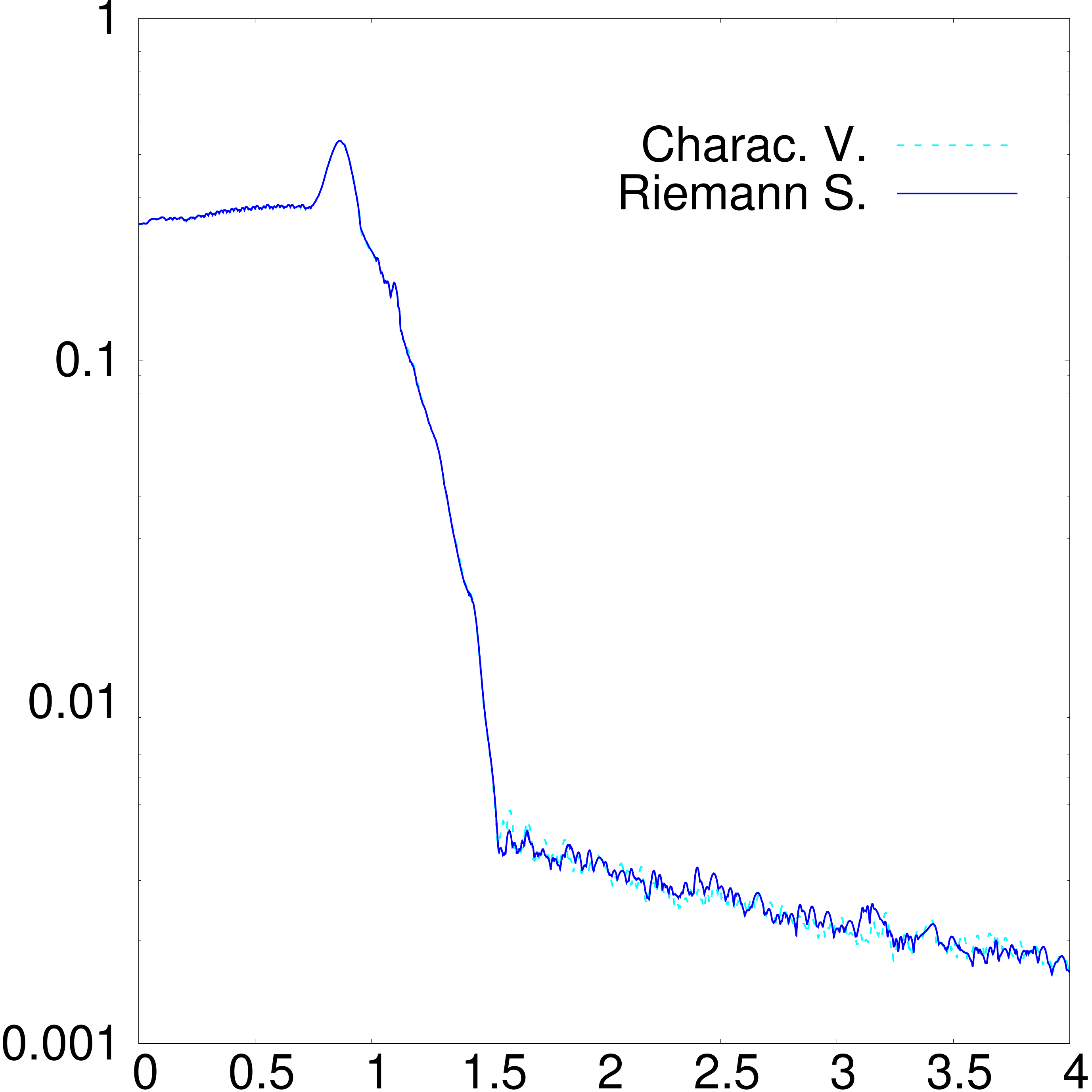} \hfil
\includegraphics[width=0.24\textwidth]{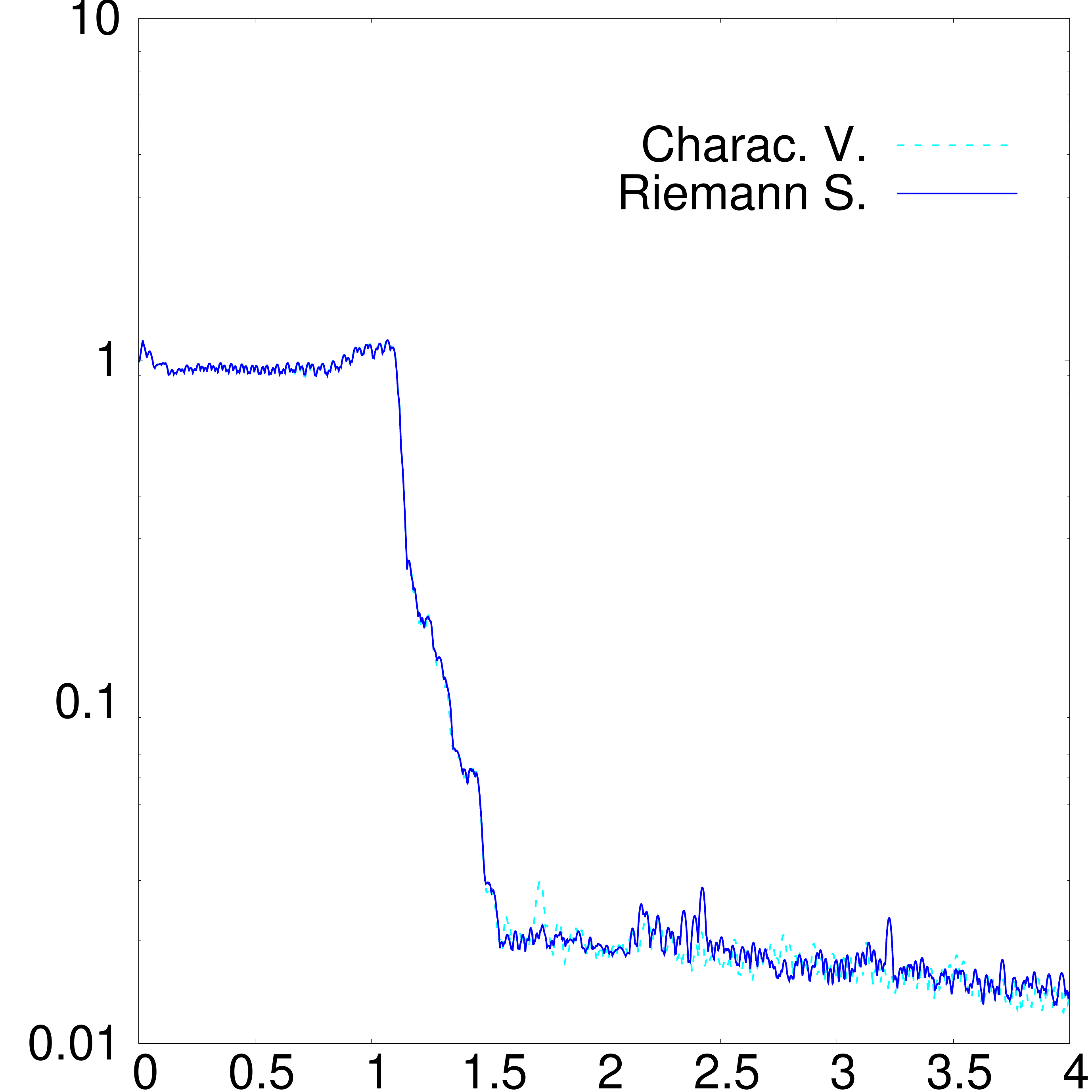}}
\caption{Tests on the non-reflecting boundary conditions.} \label{Fig:Poinsot_test}
\end{figure}
We show in Figure~\ref{Fig:Poinsot_test} the quantities
$\delta_1(t)\eqq
\frac{\|\bv(\cdot,t)-\bv_\infty\|_{\bL^\infty(\Dom)}}{\|\bv_\infty\|_{\bL^\infty(\Dom)}}$
and
$\delta_2(t)\eqq
\frac{\|\ROT\bv(\cdot,t)\|_{\bL^\infty(\Dom)}}{\|\ROT\bv_0\|_{\bL^\infty(\Dom)}}$
as functions of time over the time interval $[0,4]$.  The label
``Riemann S.\@'' refers to the method from \S\ref{Sec:Godunov} and the
label ``Charac.\@ V.\@'' refers to the method from
\S\ref{Sec:Characteritic_variables}.
Since the center of the vortex crosses the outflow boundary
at $t=1$, the faster $\delta_1$ and $\delta_2$ go to zero as $t$
grows, the better the non-reflecting properties of the boundary
condition are.  We observe that the method using the exact solution to
a Riemann problem is slightly more
efficient than that using the characteristics variables when the
velocity perturbation is large ($M_\infty=0.5$,
$\overline v_\infty=0.75$).
The method using the characteristic
variables performs as well as the method using the Riemann solution in
the cases (ii)-(iii)-(iv).
Overall the Riemann solution method has properties similar to
the method labeled ``OC2'' in \cite{Fosso_2012}.

\subsection{2D shocktube benchmark}
\label{subse:shocktube}

We now illustrate the accuracy of the proposed algorithm by testing it
against a challenging two-dimensional benchmark problem introduced in
the literature by \cite{Daru_Tenaud_2000,Daru_Tenaud_2009}. The
configuration is a shocktube problem in a square cavity
$\Dom\eqq (0,1)^2$ where a shock interacts with a viscous boundary
layer. A lambda shock is formed as a result of this interaction;
see Figure~\ref{fig:shocktube}(a).
\begin{figure}[htbp]
    \setlength\fboxsep{0pt}
    \setlength\fboxrule{0.4pt}
  \begin{center}
    \subfloat[]{
     \fbox{\includegraphics[width=0.42\textwidth]{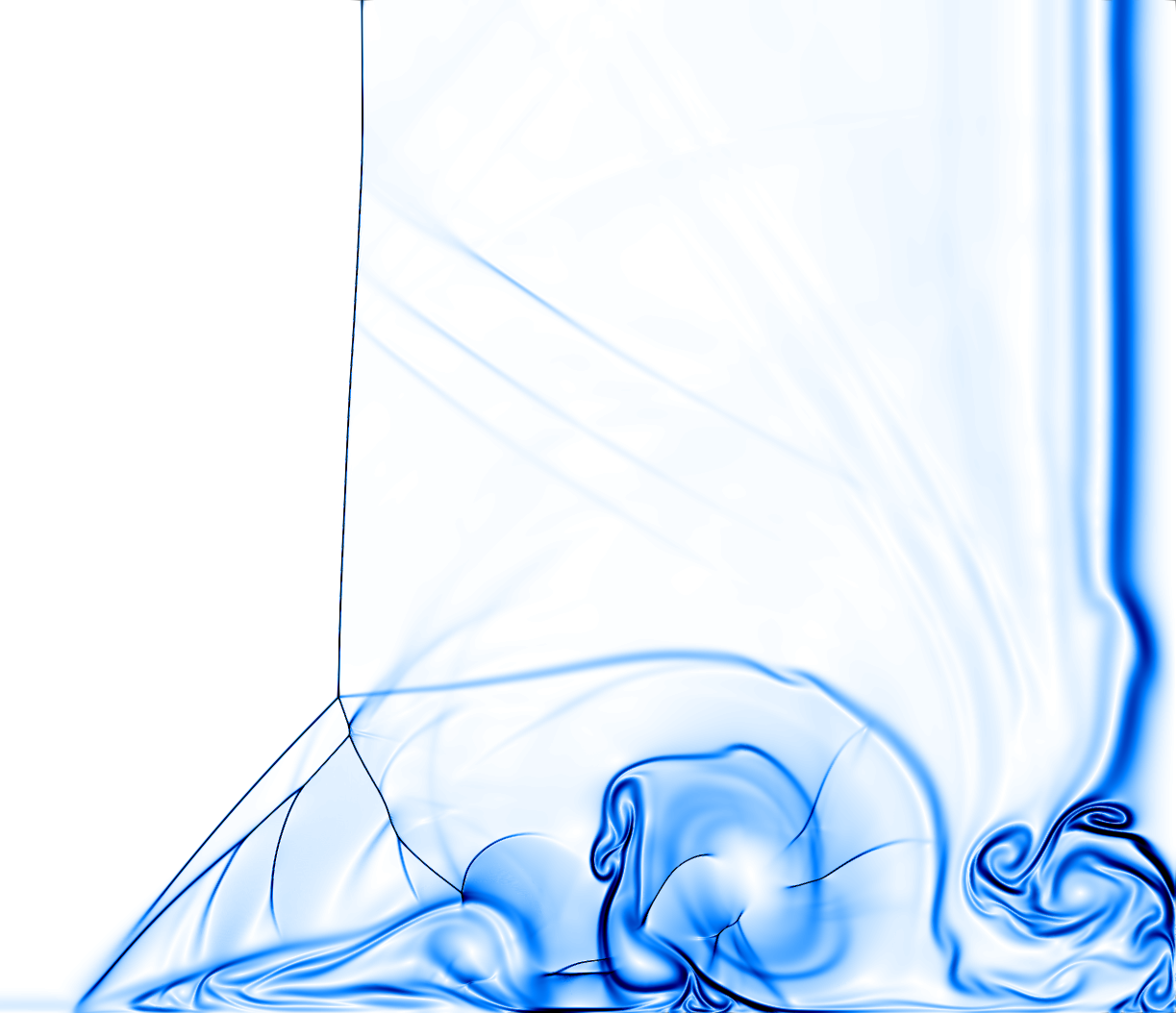}}
    }
    \subfloat[]{\input{c_f.tex}}
  \end{center}
  \caption{The 2D shocktube benchmark: (a) zoom ($[0.5,1]\times[0,0.5]$) of
    a schlieren plot at $t=1.00$ showing the interactions of the
    reflected shock wave with the viscous boundary layer. (b) postprocessed
    skin friction coefficient $C\lof$ at time $t=1.00$. The continuous green
    line is the numerical result with the finest level L14. The extrema
    are computed by means of extrapolation.
    }
  \label{fig:shocktube}
\end{figure}
The fluid is initially at rest at $t=0$. There are two different
states separated by a diaphragm at $\{x=\frac12\}$. The density,
velocity and pressure on the left-hand side of the diaphragm are
$\rho_L=120$, $v_L=0$, $p_L=\rho_L/\gamma$. The states on the
right-hand side are $\rho_R=1.2$, $v_R=0$, $p_R=\rho_R/\gamma$.  A
shock, a contact discontinuity and an expansion wave are created after
the diaphragm is broken. The viscous shock and the contact wave move
to the right, whereas the expansion wave moves to the left. A thin
viscous boundary layer is created on the bottom wall of the cavity as
the shock and the contact waves progress toward the right wall.  The
shock reaches the right wall at approximately $t\approx 0.2$, is
reflected, and moves back to the left thereafter. The contact
discontinuity remains stationary close to the right wall after it
interacted with the shock. A lambda shock is created as the shock
interacts with the viscous boundary layer.  The various mechanisms at
play in this problem are described in
\citep[\S6]{Daru_Tenaud_2000} and \citep[\S5 \&
\S6]{Daru_Tenaud_2009}.

The equation of state is $p=(\gamma-1)\rho e$ with
$\gamma=1.4$. The dynamics viscosity is $\mu=10^{-3}$ and the bulk
viscosity is set to $0$. The Prandtl number is
$Pr=\frac{\mu c_p}{\kappa}=0.73$, where $c_p=\frac{\gamma}{\gamma-1}$.
The no-slip boundary condition is enforced on the
velocity. The homogeneous Neumann
boundary condition is enforced on the internal energy, \ie the tube
is thermally insulated.  Due to symmetry, the computation is done in
the half domain $(0,1)\CROSS (0,\frac12)$, and the boundary conditions
$\bv\SCAL \bn=0$, $(\pole(\bu)\bn)\SCAL \btau=0$ are enforced at
$\{y=\frac12\}$, where $\btau$ is one of the two tangent unit vectors
along the boundary $\{y=\frac12\}$. Denoting the vertical component of
the velocity by $v_y$, the above boundary condition amounts to
enforcing $v_y=0$ and $\partial_y v_y =0$.

As in \citep{Daru_Tenaud_2000,Daru_Tenaud_2009}, we compute the skin
friction coefficient on the bottom wall $[0,1]\CROSS\{0\}$.  We
estimate this quantity for all the degrees of freedom
$i\in\calV\upbnd\los$ by using the following expression:
\begin{equation}
 C\lof(\bx_i)\eqq
  \frac{1}{\tfrac{1}{2}\rho_\infty\|\bv_\infty\|_{l^2}^2}
  \,
  \Big(\frac{1}{\tilde m^\partial_i}
  \int_{\front\los} \varphi_i\,\btau\SCAL(\pols(\bv_h)\bn) \diff s\Big)_\nu,
  \qquad
  \tilde m^\partial_i \eqq \int_{\front\los} \varphi_i \diff s.
\end{equation}
Our goal is to observe the convergence of this quantity as the meshes
are refined.  To this end, we compute the local maxima and minima of
$C\lof$ for a mesh sequence composed of successively refined
uniform meshes ranging from refinement level L\,11 (8 M grid points) to
L\,14 (512 M grid points). The results for all these grids are
  shown in Figure~\ref{fig:shocktube-detailed} in the
  Appendix~\ref{Sec:appendix_Daru_Tenaud}. Convergence is observed as
  the mesh size goes to zero.  The results obtained on the finest mesh
level L\,14 are shown in Figure~\ref{fig:shocktube}(b). The values of
the local minima and local maxima of $C\lof$ obtained on the grids
L\,11, L\,12, L\,13, and L\,14 are reported in
Table~\ref{tab:shocktube_extrema} and identified with symbols in  Figure~\ref{fig:shocktube-detailed}
in the Appendix~\ref{Sec:appendix_Daru_Tenaud}.
We also show in this table the extrapolated values of $C\lof$ that
have been obtained by using the open source software \texttt{gnuplot}
(see \url{http://www.gnuplot.info/}).  This is done by fitting the
four values of $C\lof$ obtained on the grids L\,11, L\,12, L\,13, and
L\,14 with the linear function $a + b\,h$ using the nonlinear
Levenberg-Marquardt least squares technique. The 12th and 13th extrema
(global maximum and minimum) have been extrapolated with convergence
orders that are slightly higher, \ie $a + b\,h^{1.5}$ and
$a + b\,h^{1.2}$, respectively.  We report the fit value $a$ and the
asymptotic standard error in the column labeled ``extrapolated'' in
Table~\ref{tab:shocktube_extrema} in the
  Appendix~\ref{Sec:appendix_Daru_Tenaud}. The results suggest that
the computation is indeed in the asymptotic regime with the normalized
asymptotic standard error
$\frac{\Delta C\lof}{\Delta C_{\textup{f,norm}}}$ smaller than
$0.2\,\%$ where
$\Delta C_{\textup{f,norm}}\eqq\max\,C\lof(\text{extrap.})  -
\min\,C\lof(\text{extrap.})$.
We also give in this table the values of $C\lof$ reported in
\citep{Daru_Tenaud_2009,Daru_Tenaud_2020}. These results have been
obtained with a 7th-order method called OSMP7 on a $2000\CROSS 4000$
uniform grid.  The relative deviation reported in the last column of
Table~\ref{tab:shocktube_extrema} is defined to be
$\frac{C\lof(\text{OSMP7}) - C\lof(\text{extrap.})}{\Delta
  C_{\textup{f,norm}}}$.
We report generally good agreement with the OSMP7 results with a
relative deviation that is generally less below $1\,\%$. There are
some noticeable differences though on the 13th, 15th, and 17th extrema
where the deviations are $-2.77\,\%$, $2.86\,\%$, and $3.99\,\%$,
respectively. We conjecture however that the results reported
  here are probably slightly more accurate than those from
  \citep{Daru_Tenaud_2009} since our finest grid is significantly  finer that 
those used therein.
Following the initiative of \citep{Daru_Tenaud_2020} and to facilitate
rigorous quantitative comparisons with other research codes, we
provide more detailed test vectors of the skin friction coefficient at
\citep{testvectors_2021}. We also reiterate the suggestion made in
\citep{Daru_Tenaud_2009} that this problem is an interesting benchmark
that should be used more systematically in the literature to test
compressible Navier-Stokes codes.

\subsection{Onera OAT15a airfoil}
\label{subse:airfoil}

The Onera OAT15a airfoil has been extensively studied in the literature.
Experimental results for the supercritical flow regime at Mach 0.73 and at
Reynolds number $3\CROSS 10^6$ have been published in
\cite{Jacquin_2005,Jacquin_2009}. Various numerical simulations of this
configuration are also available in the literature, see \eg
\cite{Deck_2005,Deck_Renard_2020,Nguyen_Terrana_Peraire_AIAA_2020}. We are
interested in predicting the pressure coefficient $C\lop$
\begin{align*}
  C\lop(\bx) \eqq\frac{\langle p(\bx,t)\rangle_t-p_\infty}{\tfrac{1}{2}\rho_\infty v_\infty^2},
\end{align*}
at Mach 0.73 with
an angle of attack of $3.5^\circ$ (deg). The symbol $\langle\cdot\rangle_t$ denotes the time average.
Here, $\rho_\infty$, $v_\infty$ and $p_\infty$ are the free-stream
density, velocity and pressure values. We take
$\rho_\infty = 1.225\,\text{kg}/{\text{m}}^3$,
$v_\infty = 248.42\,\text{m}/\text{s}$,
$p_\infty=1.013\times10^5\,\text{N}/{\text{m}}^2$ in our
computation. The ideal gas equation of state is used with
$\gamma=1.401$. We set the shear viscosity to
$\mu=1.789\times10^{-5} \text{Ns}/{\text{m}}^2$ and the bulk viscosity
to $\lambda=0$. The thermal conductivity is set to
$c_v^{-1}\kappa = 3.616\times10^{-5}\text{Ns}/{\text{m}}^2$.

The airfoil is shown in Fig.~\ref{fig:oat15a-c_p}. The chord length is
$23\,\text{cm}$ as in the experiment~\citep{Jacquin_2005,Jacquin_2009}.
In order to resolve the boundary
layer, a graded mesh is constructed using a ``manifold'' mapping
\citep{Heltai_2019}. As our primary objective is to compute the
pressure coefficient, a moderate grading using $x\mapsto x^2$
is chosen. This yields a 2D mesh with about 0.5 million
quadrilaterals. The near wall resolution for this mesh is
approximately $\Delta x \approx 250\,\mathrm{\mu m}$,
$\Delta y \approx 50\,\mathrm{\mu m}$. The 2D mesh is then extruded
with 513 repetitions resulting in a 3D mesh composed of about
274~million grid points. The resolution in the $z$-direction is
$\Delta z \approx 90\,\mathrm{\mu m}$. The ratio of the extension of the foil in
the transversal direction to the chord length is $0.20$.

The no-slip boundary condition is enforced on the airfoil.  The slip
boundary condition is enforced on the two vertical planes bounding the
computational domain in the $z$-direction.  The non-reflecting
boundary condition using the characteristics variables and described
in \S\ref{Sec:Characteritic_variables} is used for the outer boundary.
The fluid is assumed to be thermally insulated on all the boundaries.
Notice that deviating from the
experimental setup \citep[\S2]{Jacquin_2009} and other numerical
configurations \citep[Fig.~2]{Nguyen_Terrana_Peraire_AIAA_2020}, we do
not enfore a boundary-layer transition at $x/c=0.07$ with a
(numerical) trip wire.

\begin{figure}[ht]
  \begin{minipage}[b]{0.49\textwidth} \centering
  \input{oat15a-3d-c_p.tex} \\
    \small (a)
  \end{minipage}
  \begin{minipage}[b]{0.49\textwidth} \centering
   \includegraphics[width=\textwidth]{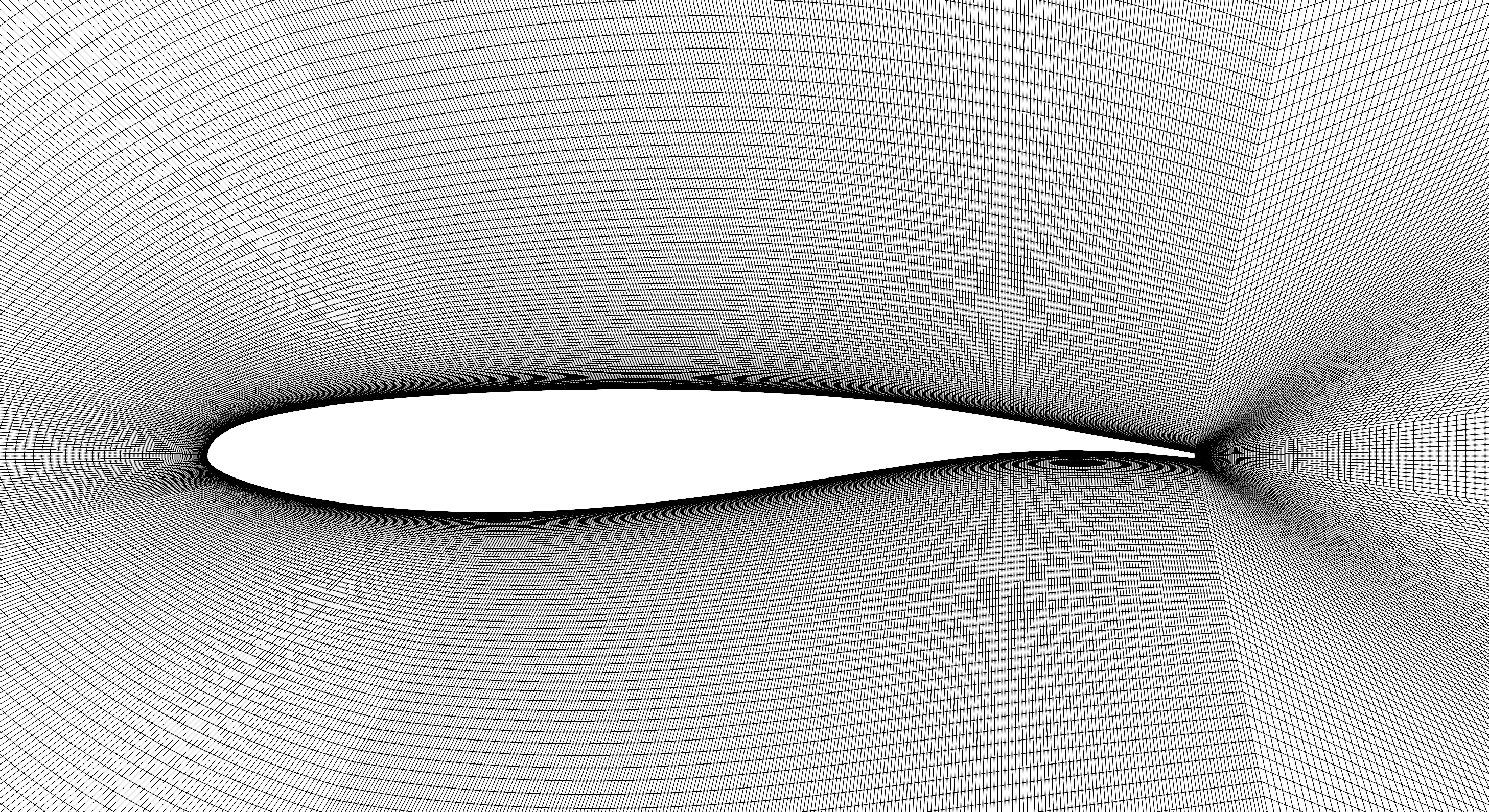} \\
\vspace*{\baselineskip}
\small (b)
  \end{minipage}
  \caption{The Onera OAT15a airfoil depicted with a graded mesh
  used in our computations. The actual hexahedral 3D mesh is created by
  extruding the quadrilateral 2D mesh in the $z$-direction by a set number of
  repetition.}
  \label{fig:oat15a-c_p}
\end{figure}
The $C\lop$ profile reported in Fig.~\ref{fig:oat15a-c_p} is computed
by averaging over 900 temporal snapshots sampled from the computation
every $\Delta t = 1.0\CROSS 10^{-4}$ and by eventually averaging the
results in the $z$-direction. The experimental results have been
provided to us by Prof. J. Peraire and have been extracted from
\citep[Fig.~7]{Jacquin_2009} using the online tool
\url{https://automeris.io/WebPlotDigitizer}. We observe that the
agreement with the experimental results is quite good.

\subsection{Compute-kernel benchmark and 3D %weak and
  strong scaling results}
\label{subse:scaling}

In this section, we report tests assessing the strong and weak
scalability of \texttt{ryujin}.
All the experiments are performed on the supercomputer
SuperMUC-NG,\footnote{\url{https://top500.org/system/179566/},
 retrieved on April 15, 2021.} using up to 2,048 nodes of $2$ Intel
Xeon Platinum 8174 CPUs (24 cores each). The CPU cores are operated at
a fixed clock frequency of 2.3 GHz. The C++ implementation is compiled
with the GNU compiler, using the option
\texttt{-march=skylake-avx512 -O3 -funroll-loops} to target the
specific machine with AVX-512 vectorization. The theoretical peak
performance of a node is 3.5 TFlop/s, and the memory bandwidth of a node
is 205 GB/s according to the STREAM benchmark. The
detailed node-level analysis of \cite{maier2020massively} has shown that the memory
bandwidth and evaluations of transcendental functions are the dominating
costs. In the following experiments, we show that the fast node-level
performance comes along with optimal scalability to large node counts,
where the simulation is partitioned among the participating processes
with Morton space filling curves using the \texttt{p4est} library by
\cite{Burstedde11} wrapped into \texttt{deal.II}
\citep{Bangerth11}. Parallelization is based on MPI, using the Intel
Omni-Path protocol of SuperMUC-NG for the inter-node communication.

\pgfplotstableread{
  nodes  dofs      nts time
  1      4360200   820 1709.32
  2      4360200  1000 1083.73
  4      4360200  1000  570.14
  8      4360200  1000  301.17
  16     4360200  1000  167.71
  32     4360200  1000   90.79
  64     4360200  1000   53.54
  128    4360200  1000   33.45
  256    4360200  1000   19.39
  512    4360200  1000   13.01
  1024   4360200  1000    9.11
  2048   4360200  1000    8.35
}\tableStrongNStokesZ
\pgfplotstableread{
  nodes  dofs      nts time    velrhs velsol enerhs enesol copy
  4      34479120  380 1536.42 2.58   67.37  12.70  41.19  2.69
  8      34479120  820 1689.20 2.76   73.60  13.84  43.43  3.16
  16     34479120 1000 1087.13 1.39   44.45  8.55   25.95  2.09
  32     34479120 1000  568.68 0.51   22.75  4.42   13.16  1.16
  64     34479120 1000  334.32 0.20   15.46  2.18   6.63   0.84
  128    34479120 1000  172.04 0.10   8.54   0.99   3.64   0.66
  256    34479120 1000   93.35 0.06   7.61   0.40   1.58   0.33
  512    34479120 1000   54.92 0.04   4.26   0.19   0.94   0.20
  1024   34479120 1000   31.99 0.03   2.14   0.12   0.78   0.15
  2048   34479120 1000   20.12 0.02   1.77   0.08   0.85   0.11
}\tableStrongNStokesA
\pgfplotstableread{
  nodes  dofs       nts time
  32     274229280  980 3929.86
  64     274229280 1000 2114.21
  128    274229280 1000 1108.65
  256    274229280 1000  592.13
  512    274229280 1000  334.09
  1024   274229280 1000  180.53
  2048   274229280 1000  100.67
}\tableStrongNStokesB
\pgfplotstableread{
  nodes  dofs        nts time
  256    2187432000 1000 4127.92
  512    2187432000 1000 2129.07
  1024   2187432000 1000 1136.00
  2048   2187432000 1000  610.53
}\tableStrongNStokesC
\pgfplotstableread{
  nodes  dofs         nts time
  2048   17473872000 1000 4483.64
}\tableStrongNStokesD

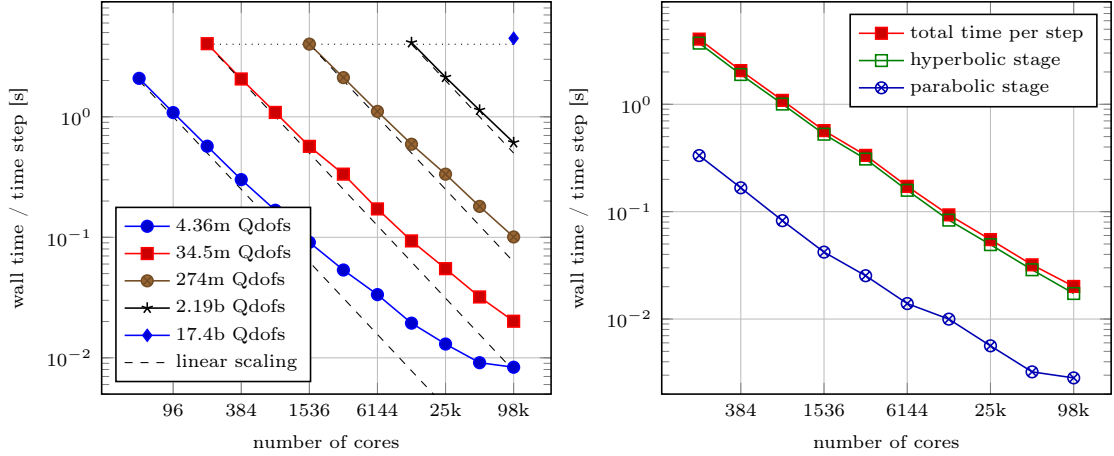
\begin{figure}[h]
  \centering
    \begin{tikzpicture}
      \begin{loglogaxis}[
        width=0.5\textwidth,
        height=0.45\textwidth,
        xlabel={number of cores},
        ylabel={wall time / time step [s]},
        tick label style={font=\scriptsize},
        label style={font=\scriptsize},
        legend style={font=\scriptsize},
        legend cell align=left,
        legend pos={south west},
        grid,
        mark size=2.2,
        xtick={96,384,1536,6144,24576,98304},
        xticklabels={96,384,1536,6144,25k,98k},
        ymin=5e-3,ymax=9.0e+0,
        semithick
        ]
        \addplot table[x expr={\thisrowno{0}*48}, y expr={\thisrowno{3}/\thisrowno{2}}] {\tableStrongNStokesZ};
        \addlegendentry{4.36m Qdofs};
        \addplot table[x expr={\thisrowno{0}*48}, y expr={\thisrowno{3}/\thisrowno{2}}] {\tableStrongNStokesA};
        \addlegendentry{34.5m Qdofs};
        \addplot table[x expr={\thisrowno{0}*48}, y expr={\thisrowno{3}/\thisrowno{2}}] {\tableStrongNStokesB};
        \addlegendentry{274m Qdofs};
        \addplot table[x expr={\thisrowno{0}*48}, y expr={\thisrowno{3}/\thisrowno{2}}] {\tableStrongNStokesC};
        \addlegendentry{2.19b Qdofs};
        \addplot table[x expr={\thisrowno{0}*48}, y expr={\thisrowno{3}/\thisrowno{2}}] {\tableStrongNStokesD};
        \addlegendentry{17.4b Qdofs};
        \addplot[dashed,black,thin] coordinates {
            (192,  4.0    )
            (98304,4.0/512)
          };
        \addplot[dashed,black,thin] coordinates {
            (48,   2.0    )
            (49152,2.0/1024)
          };
        \addlegendentry{linear scaling};
        \addplot[dashed,black,thin] coordinates {
            (1536, 4.0   )
            (98304,4.0/64)
          };
        \addplot[dashed,black,thin] coordinates {
            (12288,4.0  )
            (98304,4.0/8)
          };
        \addplot[dotted,black,thin] coordinates {
            (192,  4.0)
            (98304,4.0)
          };
        \end{loglogaxis}
      \end{tikzpicture}
      \begin{tikzpicture}
      \begin{loglogaxis}[
        width=0.5\textwidth,
        height=0.45\textwidth,
        xlabel={number of cores},
        ylabel={wall time / time step [s]},
        tick label style={font=\scriptsize},
        label style={font=\scriptsize},
        legend style={font=\scriptsize},
        legend cell align=left,
        legend pos={north east},
        grid,
        mark size=2.2,
        xtick={96,384,1536,6144,24576,98304},
        xticklabels={96,384,1536,6144,25k,98k},
        ymin=2e-3,ymax=9.0e+0,
        semithick
        ]
        \addplot[color=red,mark=square*,every mark/.append style={fill=red!80!black}] table[x expr={\thisrowno{0}*48}, y expr={\thisrowno{3}/\thisrowno{2}}] {\tableStrongNStokesA};
        \addlegendentry{total time per step};
        \addplot[color=green!50!black,mark=square] table[x expr={\thisrowno{0}*48}, y expr={(\thisrowno{3}-\thisrowno{4}-\thisrowno{5}-\thisrowno{6}-\thisrowno{7}-\thisrowno{8})/\thisrowno{2}}] {\tableStrongNStokesA};
        \addlegendentry{hyperbolic stage};
        \addplot[color=blue!70!black,mark=otimes] table[x expr={\thisrowno{0}*48}, y expr={(\thisrowno{4}+\thisrowno{5}+\thisrowno{6}+\thisrowno{7}+\thisrowno{8})/\thisrowno{2}}] {\tableStrongNStokesA};
        \addlegendentry{parabolic stage};
        \end{loglogaxis}
      \end{tikzpicture}
      \caption{Scaling analysis of the 3D Onera OAT15a airfoil  on up to
        2,048 nodes (\ie 98,304 cores) of SuperMUC-NG. In the left panel,
        the strong and weak scaling of five different problem sizes is
        assessed, involving up to 17.4 billion grid points (Qdofs) (\ie 85
        billion unknowns). In the right panel, the time per time step for
        the case with 34.5 million grid points is broken down to show the
        contributions of the hyperbolic and parabolic parts.}
  \label{fig:strong_scaling}
\end{figure}

In a first series of tests, we assess the strong scalability
performance of \texttt{ryujin} on the Onera OAT15a airfoil studied in
\S\ref{subse:airfoil}. The left panel in
Figure~\ref{fig:strong_scaling} shows the time spent per time step
versus the number of cores for five different problem sizes ranging
from 4.4 million to 17.4 billion grid points (recall that there are 48
cores per node).  In order to represent the steady-state performance,
the reported time per time step is obtained from an experiment that
measures the run time to complete 1000 time steps and then divides
this time by the number of time steps.
The case with~4.36 million grid points runs into communication latency
for 1024 nodes (49k cores). At this point, the number of grid points
per node is about 4,257, and the number of grid points per core is
about~89. For larger problem sizes per node, the scaling is almost
ideal, with a slight loss in efficiency as the share of
SIMD-vectorized work in the algorithm by \cite{maier2020massively} is
reduced near the interfaces between different parallel processes, and
by some slight load imbalance in the dynamic steps of the computation
as the number of cores increases. The performance achieved is
nevertheless excellent. For example, the computation with 274 million
grid points on 98k cores achieves a parallel efficiency of 63\% when
using the computation on 1,536 cores as baseline; the turnaround time
is then almost one million time steps per day. The case with 2.2
billion grid points on 98k cores achieves an efficiency of 85\%
against the computation on 12,288 cores.  Note in passing that
this series of tests also demonstrates excellent weak scalability. For
instance, counting from the top of the panel, the first red square,
the first brown circle, the first black star, and the unique blue
diamond are perfectly aligned along the horizontal dotted line. This
means that the run time per time step remains constant as the number
of grid point and the number of cores are both multiplied by 8.

We show in the right panel of Figure~\ref{fig:strong_scaling} the time per time step
for the parabolic step and for the hyperbolic step using the mesh composed of 34.5 millions grid points.
This breakdown of the timings
reveals that the hyperbolic stage dominates the run time in this case. This is
because the Reynolds number in this 3D benchmark is so large
that the conjugate gradient solver with a simple inverse mass matrix
preconditioner as described in \S\ref{Sec:linear_algebra} is efficient. Even
for the largest case with 17 billion points, we observe that the velocity solver converges in
2 iterations and the energy solver takes 3 iterations.

\pgfplotstableread{
  nodes  dofs       nts time    velrhs velsol enerhs enesol copy
  4      134242305 1000 8940.50 113.25 2509.2 38.54  2687.8 15.50
  8      134242305 1000 4537.77 57.56  1232.4 19.33  1408.4 8.45
  16     134242305 1000 2313.05 28.74  591.14 9.43   735.55 4.41
  32     134242305 1000 1166.84 14.71  275.67 4.60   394.07 2.27
  64     134242305 1000 600.27  7.41   130.42 2.31   214.55 1.18
  128    134242305 1000 317.62  3.62   65.87  1.16   120.64 0.58
  256    134242305 1000 178.62  1.93   37.73  0.60   70.44  0.24
  512    134242305 1000 102.29  1.08   23.86  0.29   39.74  0.11
  1024   134242305 1000 58.08   0.59   15.30  0.16   22.47  0.05
  2048   134242305 1000 41.85   0.36   15.51  0.11   14.49  0.04
}\tableStrongMultigrid

\begin{figure}[h]
  \strut\hfill
  \begin{tikzpicture}
      \begin{loglogaxis}[
        width=0.5\textwidth,
        height=0.45\textwidth,
        xlabel={number of cores},
        ylabel={wall time / time step [s]},
        tick label style={font=\scriptsize},
        label style={font=\scriptsize},
        legend style={font=\scriptsize},
        legend cell align=left,
        legend pos={north east},
        grid,
        mark size=2.2,
        xtick={96,384,1536,6144,24576,98304},
        xticklabels={96,384,1536,6144,25k,98k},
        ymin=8e-3,ymax=1.2e+1,
        semithick
        ]
        \addplot[color=red,mark=square*,every mark/.append style={fill=red!80!black}] table[x expr={\thisrowno{0}*48}, y expr={\thisrowno{3}/\thisrowno{2}}] {\tableStrongMultigrid};
        \addlegendentry{total time per step};
        \addplot[color=green!50!black,mark=square] table[x expr={\thisrowno{0}*48}, y expr={(\thisrowno{3}-\thisrowno{4}-\thisrowno{5}-\thisrowno{6}-\thisrowno{7}-\thisrowno{8})/\thisrowno{2}}] {\tableStrongMultigrid};
        \addlegendentry{hyperbolic stage};
        \addplot[color=blue!70!black,mark=otimes] table[x expr={\thisrowno{0}*48}, y expr={(\thisrowno{4}+\thisrowno{5}+\thisrowno{6}+\thisrowno{7}+\thisrowno{8})/\thisrowno{2}}] {\tableStrongMultigrid};
        \addlegendentry{parabolic stage};
        \addplot[dashed,black,thin] coordinates {
            (192,  3.5)
            (98304,3.5/512)
          };
        \end{loglogaxis}
      \end{tikzpicture}
      \hfill\strut
      \caption{Analysis of strong scaling of 2D shocktube test with 134
        million grid points and multigrid solvers for velocity and energy.}
  \label{fig:strong_scaling_mg}
\end{figure}
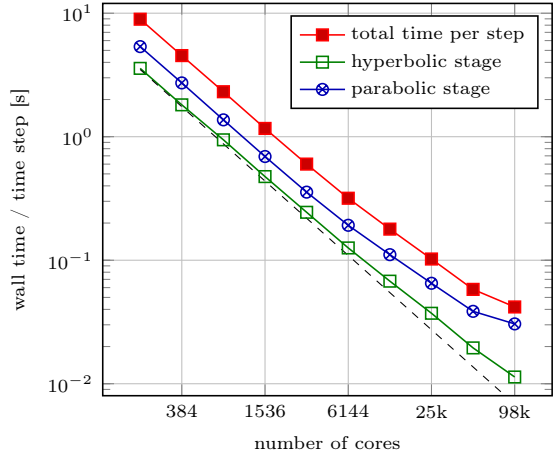

Figure~\ref{fig:strong_scaling_mg} shows a strong scaling experiment done
on the 2D shocktube discussed in \S\ref{subse:shocktube}. The simulations
are done on the mesh refinement level L13 with 134 million grid points.
Here, the flow is locally dominated by viscous effects, and therefore the
geometric multigrid solver described in \S\ref{Sec:linear_algebra} is used
for optimal performance. While the contribution to the run time per time
step of the viscous step is significantly higher when the multigrid solver
is activated than when using the plain conjugate gradient solver, the
parallel scaling is still excellent in this case. We observe that the
strong scaling still holds after multiplying the initial core count by
$2^8=256$. With $192\CROSS 2^8$ cores, the number of grid point per node is
about 131,000. The performance deteriorates though when the initial core count
is multiplied by $2^9=512$.  The excellent strong scaling performance is
due to the algorithms developed in \cite{Kronbichler2018} and
\cite{Clevenger2021}.

In summary, the above scaling experiments confirm that the proposed
algorithms have good node-level performance.  They are suitable to solve
large-scale problems on large-scale computers in various application
scenarios.

\section*{Acknowledgments} The authors thank C. Tenaud and V. Daru for
stimulating discussions and for sharing their
data~\citep{Daru_Tenaud_2020}. They also thank J. Peraire and his
collaborators for sharing the experimental data used in
Figure~\ref{fig:oat15a-c_p}. The authors finally thank S. Deck for
providing the OART15a geometry file. They also thank him for providing
a file containing experimental data, which, unfortunately, they have
not been able to read.

This material is based upon work supported in part by the National Science
Foundation via grants DMS 1619892 (BP \& JLG), DMS 1620058 (JLG) and DMS
1912847 \& 2045636 (MM); the Air Force Office of Scientific Research, USAF, under
grant/contract number FA9550-18-1-0397 (JLG\&BP); and by the Army Research Office
under grant/contract number W911NF-19-1-0431 (JLG\&BP).  IT was supported by
LDRD-21-1189 project \#23796 from Sandia National Laboratories.
MK was supported by the Bayerisches Kompetenznetzwerk f\"ur
Technisch-Wissenschaftliches Hoch- und H\"ochstleistungsrechnen (KONWIHR). The
authors gratefully acknowledge the Gauss Centre for Supercomputing
e.V.~(\texttt{www.gauss-centre.eu}) for funding this project by providing
computing time on the GCS Supercomputer SuperMUC at Leibniz Supercomputing
Centre (LRZ, \texttt{www.lrz.de}) through project id pr83te.

%%%%%%%%%%%%%%%%%%%%%%%%%%%%%%%%%%%%%%%%%%%%%%%%%%%%%%%%%%%%%%%%%%%%%%%%%%%%%%%%
%%%%%%%%%%%%%%%%%%%%%%%%%%%%%%%%%%%%%%%%%%%%%%%%%%%%%%%%%%%%%%%%%%%%%%%%%%%%%%%%
%%%%%%%%%%%%%%%%%%%%%%%%%%%%%%%%%%%%%%%%%%%%%%%%%%%%%%%%%%%%%%%%%%%%%%%%%%%%%%%%

%\clearpage % TODO: Remove
%\newpage

\appendix

\section{Boundary fluxes and boundary conditions}\label{sec:BCmath}
We collect in this appendix proofs of results stated in the body of
the paper. The statements are repeated for clarity.

\begin{lemma}[Balance after limiting] \label{Lem:boundary_flux_bis}%
  \phantom{It holds:}
  \begin{enumerate}[font=\upshape,label=(\roman*)]
    \item
      For all $\bu_h\eqq \sum_{i\in\calV} \bsfU_i\varphi\in \bP(\calT_h)$, the
      following holds true:
      $\int_\Dom \bu_h\diff x = \sum_{i\in\calV} m_i \bsfU_i$.
    \item
      Let $\bsfU^n$ be a collection of admissible states. Let $\bsfU\upnp$ be
      the update after one forward-Euler step and after limiting. Then the
      following balance identity holds:
      \begin{equation}\label{Eq:Lem:boundary_flux_bis}%
        \sum_{i\in\calV} m_i \bsfU_i^{n+1} + \dt_n \sum_{i\in\calV\upbnd}
        m\upbnd_i \polf(\bsfU_i^n)\bn_i = \sum_{i\in\calV} m_i \bsfU_i^{n}.
      \end{equation}
  \end{enumerate}
\end{lemma}

\begin{proof}
\textup{(i)}
Using the definition $m_i\eqq \int_\Dom \varphi_i\diff x$, we have
\begin{align*}
  \int_\Dom \bu_h\diff x = \sum_{i\in\calV} \bsfU_i \int_\Dom \varphi_i \diff x =
  \sum_{i\in\calV} m_i \bsfU_i.
\end{align*}
\textup{(ii)}
Recall that the definition of $\bsfU\upL$ implies that
\begin{align*}
  \sum_{i\in\calV} m_i \bsfU_i\upLnp + \dt_n \sum_{i\in\calV} \sum_{j\in\calI(i)} \polf(\bsfU_j^n)
  \int_\Dom \varphi_i\GRAD\varphi_j\diff x - \dt_n
  \underbrace{\sum_{i\in\calV}\sum_{j\in\calI(i)} d_{ij}\upLnp
  (\bsfU_j^n-\bsfU_i^n)}_{=\;0}
  = \sum_{i\in\calV} m_i \bsfU_i^{n}.
\end{align*}
Using that limiting conserves the total mass,
$\sum_{i\in\calV} m_i \bsfU_i\upn = \sum_{i\in\calV} m_i
\bsfU_i\upLnp$,
see \eqref{mass_conservation_by_limiting}, the partition of unity
property $\sum_{i\in\calV}\varphi_i=1$, and the definition of
$\{m_j\upbnd\}_{j\in\calV\upbnd}$ and $\{\bn_j\}_{j\in\calV\upbnd}$ yields
\begin{align*}\sum_{i\in\calV} m_i \bsfU_i^{n}
  &=\sum_{i\in\calV} m_i \bsfU_i\upnp + \dt_n \sum_{j\in\calV} \polf(\bsfU_j^n)\int_\Dom \GRAD\varphi_j \diff x \\
  &=\sum_{i\in\calV} m_i \bsfU_i\upnp + \dt_n \sum_{j\in\calV} \polf(\bsfU_j^n)\int_{\front} \varphi_j \bn \diff s\\
  &=\sum_{i\in\calV} m_i \bsfU_i\upnp + \dt_n \sum_{j\in\calV} m_j\upbnd \polf(\bsfU_j^n) \bn_j.
\end{align*}
The assertion is proved. 
\end{proof}

\begin{lemma}[Slip condition]\label{Lem:slip_bc_bis}
Let $i\in
\calV\upbnd\los$, let $\bsfU_i\in \calA$, and let $\bsfU_i\upP$ as defined in \eqref{slip_bc}.
  \begin{enumerate}[font=\upshape,label=(\roman*)]
    \item Then $\bsfU_i\upP$ is also admissible, meaning
$\bsfU_i\upP\in \calA$.
    \item Assume also that the equation of state derives from an
entropy $s$.  Then $s(\bsfU_i\upP) \ge s(\bsfU_i)$.
    \item For all $i\in \calV\upbnd\los{\setminus}\calV\upbnd\lonr$,
the mass flux and the total energy flux of the postprocessed solution
at $i$ is zero (\ie $\rho(\polf(\bsfU_i\upP)\bn_i)=0$ and
$E(\polf(\bsfU_i\upP)\bn_i)=0$).
  \end{enumerate}
\end{lemma}

\begin{proof}
\textup{(i)}
Let us denote $\bsfU_i\upP \qqe (\varrho_i\upP,\bsfM_i\upP,\sfE_i\upP)$.
Then $\varrho_i\upP=\varrho_i$, which implies that $\varrho_i\upP>0$ since
$\bsfU_i\in\calA$. We also have
\begin{align*}
  \varepsilon(\bsfU\upP_i)
  &= \sfE_i\upP - \tfrac{1}{2\varrho_i\upP} (\bsfM_i\upP)^2 \\
  &= \sfE_i - \tfrac{1}{2\varrho_i} ( (\bsfM_i)^2 - (\bsfM_i\SCAL\bn_i)^2)
    \ge
    \sfE_i - \tfrac{1}{2\varrho_i} (\bsfM_i)^2
  = \varepsilon(\bsfU_i).
\end{align*}
That is $\varepsilon(\bsfU\upP_i)\ge \varepsilon(\bsfU_i)> 0$
because $\bsfU_i\in\calA$. Since $\rho(\bsfU_i\upP) = \rho(\bsfU_i)$ and, as
proved above, the internal energy $\varepsilon(\bsfU)$ stays positive, we
infer that the specific internal energy $e(\bsfU) = \varepsilon(\bsfU)/\rho$
remains positive too. This proves the first assertion. \\
\textup{(ii)} Let us make the change of variable $\sigma(\rho(\bsfU),e(\bsfU))
\eqq s(\bsfU)$. Using fundamental theorem of calculus
\[
s(\bsfU_i\upP) = \sigma(\varrho_i,e(\bsfU\upP)) =
\sigma(\varrho_i,e(\bsfU))
+ \int_{e(\bsfU)}^{e(\bsfU\upP)} \partial_e \sigma(\varrho_i,e) \diff e.
\]
But, in order for $\sigma$ to be physically realistic it has to satisfy
$\partial_{e} \sigma(\rho,e)>0$. Hence $s(\bsfU_i\upP)\ge
\sigma(\varrho_i,e(\bsfU)) \eqq s(\bsfU_i)$. \\
\textup{(iii)} Let $i\in \calV\upbnd\los{\setminus}\calV\upbnd\lonr$,
then $\bn_i\ups=\bn_i$.  Recall from \eqref{Eq:Lem:boundary_flux_bis}
that the boundary flux induced by $\bsfU_i\upP$ is
$m_i\upbnd\polf(\bsfU_i\upP)\bn_i$.  Let $\bsfV_i\upP$ be the velocity of the
post-processed state $\bsfU_i\upP$.  By definition,
$\rho(\polf(\bsfU_i\upP)\bn_i) = \varrho_i\upP\bsfV_i\upP\SCAL \bn_i$,
$E(\polf(\bsfU_i\upP)\bn_i) = \bsfV_i\upP\SCAL \bn_i (\sfE_i\upP +
\sfP_i\upP)$ and $\bsfV_i\upP\SCAL \bn_i=\bsfV_i\upP\SCAL \bn_i\ups=0$, whence the assertion.
\end{proof}

\begin{lemma}[Global conservation] \label{Lem:global_conservation_bis}
Assume that $\calV\los\upbnd =\calV\upbnd$ and $\bsfU^n$
satisfies the slip boundary condition (\ie $\bsfM_i^n\SCAL \bn_i=0$
for all $i\in\calV\upbnd$).  Then the solution obtained at the end the
RKSSP(3,3) algorithm after limiting and post-processing, say
$\bsfU^{n+1}$, satisfies
$\sum_{j\in\calV} m_j \varrho_j\upnp = \sum_{j\in\calV} m_j
\varrho_j\upn$
and
$\sum_{j\in\calV} m_j \sfE_j\upnp = \sum_{j\in\calV} m_j \sfE_j\upn$.
\end{lemma}

\begin{proof}
Let us assume now that $\calV\los\upbnd =\calV\upbnd$ and
$\bsfU^n$ satisfies the slip boundary condition at every boundary
node. Referring to \eqref{SPPRK33} for the notation, let us denote
$\bw_h^{(1)}\eqq \sum_{i\in\calV} \bsfW_i^{(1)}$ the update obtained
after the first forward-Euler step (high-order plus limiting). Then
using the identity \eqref{Eq:Lem:boundary_flux_bis}, we infer that
$\sum_{i\in\calV}m_i\rho(\bsfW_i^{(1)}) =
\sum_{i\in\calV}m_i\rho(\bsfU_i^{n})$.
Similarly, after post-processing $\bw^{(1)}_h$, the update
$\bw^{(2)}_h$ satisfies
\[
\sum_{i\in\calV}m_i\rho(\bsfW_i^{(2)}) = \tfrac34
\sum_{i\in\calV}m_i\rho(\bsfU_i^{n}) + \tfrac14 \sum_{i\in\calV}
m_i\rho((\bsfW_i^{(1)})\upP)) =  \tfrac34
\sum_{i\in\calV}m_i\rho(\bsfU_i^{n}) + \tfrac14 \sum_{i\in\calV}
m_i\rho((\bsfW_i^{(1)}))),
\]
\ie $\sum_{i\in\calV}m_i\rho(\bsfW_i^{(2)}) = \sum_{i\in\calV}m_i\rho(\bsfU_i^{n})$.
$\bw_h^{(3)}$ be update obtained at the final stage.
After post-processing $\bw_h^{(2)}$ we obtain
\[
\sum_{i\in\calV}m_i\rho(\bsfW_i^{(3)}) = \tfrac13
\sum_{i\in\calV}m_i\rho(\bsfU_i^{n}) + \tfrac23 \sum_{i\in\calV}
m_i\rho((\bsfW_i^{(2)})\upP)) =  \tfrac13
\sum_{i\in\calV}m_i\rho(\bsfU_i^{n}) + \tfrac23 \sum_{i\in\calV}
m_i\rho((\bsfW_i^{(2)}))),
\]
\ie
$\sum_{i\in\calV}m_i\rho(\bsfW_i^{(3)}) =
\sum_{i\in\calV}m_i\rho(\bsfU_i^{n})$.
Let $\bu_h\upnp\eqq \sum_i\bsfU_i\upnp\varphi_i$ be the final update after
post-processing, \ie $\bu_h\upnp = \calP(\bw_h^{(3)})$.  Then,
$\sum_{i\in\calV}m_i\rho(\bsfU_i\upnp) =
\sum_{i\in\calV}m_i\rho(\bsfW_i^{(3)}) =
\sum_{i\in\calV}m_i\rho(\bsfU_i^{n})$.
The argument for the conservation of the total energy is identical.
Actually, the argument holds for any explicit SSPRK technique.
\end{proof}

\clearpage

\section{Hyperbolic step} \label{Sec:appendix_hyperbolic_step}

\newcommand{\Ii}{\calI(i)}
\newcommand{\NN}{\mathcal{N}}
\newcommand{\bUni}{\bsfU^n_i}
\newcommand{\bUnj}{\bsfU^n_j}
\newcommand{\bUnij}{\overline{\bsfU}^n_{ij}}
\resizebox{0.90\textwidth}{!}{\begin{minipage}{\textwidth}
\begin{algorithm2e}[H]
  \DontPrintSemicolon
  \SetKwProg{Euler}{}{}{end}
 \Euler{}{
    \tcp*[l]{Step 1: compute off-diagonal $d_{ij}\upLn$ and $\alpha_i$}
    \For{$i\in\calV\upint$}{
      \texttt{compute indicator} $\alpha_i$\;
      \For{$j\in\Ii$, $j>i$}{
        $d_{ij}\upLn\;\leftarrow\;
        \lambda_{\max}(\bn_{ij},\bUni,\bUnj)\,\|\bc_{ij}\|_{\ell^2}$
      }
    }
    \For{$i\in\calV\upbnd$}{
      \texttt{compute indicator} $\alpha_i$\\
      \For{$j\in\Ii$, $j>i$}{
        $d_{ij}\upLn\;\leftarrow\;
        \max\,\left(\lambda_{\max}(\bn_{ij},\bUni,\bUnj)\,\|\bc_{ij}\|_{\ell^2},
        \lambda_{\max}(\bn_{ji},\bUnj,\bUni)\,\|\bc_{ji}\|_{\ell^2} \right)$
      }
    }
\tcp*[l]{Step 2: fill lower-diagonal part and compute $d_{ii}\upLn$ and $\dt_n$}
    $\dt_n\;\leftarrow\;+\infty$\;
    \For{$i=1$, \ldots, $\NN$}{
      \For{$j\in\Ii$, $j<i$}{
        $d_{ij}\upLn\;\leftarrow\;d_{ji}\upLn$\;
      }
      $d_{ii}\upLn\;\leftarrow\;-\sum_{j\in\Ii,j\not=i}d_{ij}\upLn$
      ;\quad
      $\dt_n\;\leftarrow\;\min\left(\dt_n,-c_{\text{cfl}}\frac{m_i}{2d_{ii}^{L,n}}\right)$\;
    }
\tcp*[l]{Step 3: low-order update, compute $\bsfF_i\upH$ and accumulate limiter bounds}
    \For{$i\in\calV$}{
      $\bU^{n+1}_i  \;\leftarrow\; \bU^{n}_i$, \quad $\bsfF_{i}\upH  \;\leftarrow\; \bzero$\;
      \For{$j\in\Ii$}{
        $d_{ij}\upHn \;\leftarrow\; d_{ij}\upLn\,\frac{\alpha_i^n+\alpha_j^n}{2}$
        ;\qquad
        $\bsfF\upH_i \;\leftarrow\; \bsfF\upH_i- \polf_j \SCAL\bc_{ij} + d_{ij}\upHn\big(\bUnj-\bUni\big)$\;
        $\bUnij \;\leftarrow\; \frac12\big(\bUni+\bUnj\big)-\frac1{2\,d_{ij}\upLn}\big(\polf_j-\polf_i\big)\SCAL\bc_{ij}$
        \;
        $\bsfU^{n+1}_i \;\leftarrow\; \bsfU^{n+1}_i + \frac{2\,\dt_n}{m_i}\,d\upLn_{ij}\bUnij$\;
        \texttt{accumulate local bounds from $\bUnij$}\;
      }
    }
\tcp*[l]{Step 4: compute $\bsfP_{ij}$ and $\limiter_{ij}$:}
    \For{$i\in\calV$}{
      \For{$j\in\Ii$}{
        $\bsfP_{ij}\;\leftarrow\;\frac{\dt_n}{\lambda_i m_i}\Big(\big(d_{ij}\upHn-d_{ij}\upHn\big)\big(\bUnj-\bUni\big)+b_{ij}\bsfF\upH_j-b_{ji}\bsfF\upH_i\Big)$\;
        \texttt{compute $l_{ij}$ from $\bU_i^{n+1}$, $\bsfP_{ij}$ and local
        bounds}\;
      }
    }
   $\ell \;\leftarrow\; \min(\ell,\ell\tr)$\;
    \For{pass $=\,1$, \ldots, number of limiter passes}{
      \tcp*[l]{Step 5, 6, \ldots: high-order update and recompute $l_{ij}$:}
      \For{$i\in\calV$}{
        \For{$j\in\Ii$}{
          $\bU^{n+1}_i \;\leftarrow\; \bU^{n+1}_i + \lambda_i
          \limiter_{ij} \bsfP_{ij}^n$\;
        }
      }
      \If{last round}{break}
      \For{$i\in\calV$}{
        \For{$j\in\Ii$}{
          $\bsfP_{ij}\;\leftarrow\; \big(1-\limiter_{ij}\big)\,\bsfP_{ij}$\;
          \texttt{compute $l_{ij}$ from $\bU_i^{n+1}$, $\bsfP_{ij}$ and local
          bounds}\;
        }
      }
    }
  }
\caption{High-order forward Euler step.}
  \label{alg:euler}
\end{algorithm2e}
\end{minipage}}

\clearpage
\newpage

\section{Parabolic step} \label{Sec:appendix_parabolic_step}

\begin{algorithm2e}[H]
  \DontPrintSemicolon
  \SetKwProg{Parabolic}{}{}{end}
 \Parabolic{}{
   \tcp*[l]{Step 1: momentum update}
   \texttt{assemble right-hand side in \eqref{mt_parabolic_discrete}}\;
   \texttt{solve \eqref{mt_parabolic_discrete}}\;
   \texttt{update momentum, \eqref{vel_parabolic_discrete}}\;
    \vspace{0.3em}
\tcp*[l]{Step 2: internal energy and total energy update}
  \texttt{assemble $\sfK\upnph$, \eqref{def_Ki_nplusone}}\;
   \texttt{solve \eqref{high_int_energy_CN}}\;
   \texttt{update internal energy, \eqref{energy_parabolic_discrete}}\;
      \vspace{0.3em}
\tcp*[l]{Step 3: check bounds and limit if bounds are violated}
\uIf{$\min_{i\in\calV} \sfe_i^{n+1}<0$}
{
   \texttt{solve for low-order solution $e_h\upLnp$, \eqref{low_order_update_internal_energy}}\;
   \texttt{compute limiting matrix $A_{ij}$, \eqref{Aij_for_internal_energy}}\;
   \texttt{Compute limiters $\limiter_{ij}$ using FCT}
 }
 \texttt{update total energy, \eqref{energy_parabolic_discrete}}\;
  \Return
}
\caption{High-order parabolic step.}
  \label{alg:parabolic}
\end{algorithm2e}

\clearpage
\newpage

\section{2D shocktube benchmark} \label{Sec:appendix_Daru_Tenaud}

\begin{figure}[H]
  \vspace{-3em}
  \begin{center}
    \rotatebox{-90}{\includegraphics{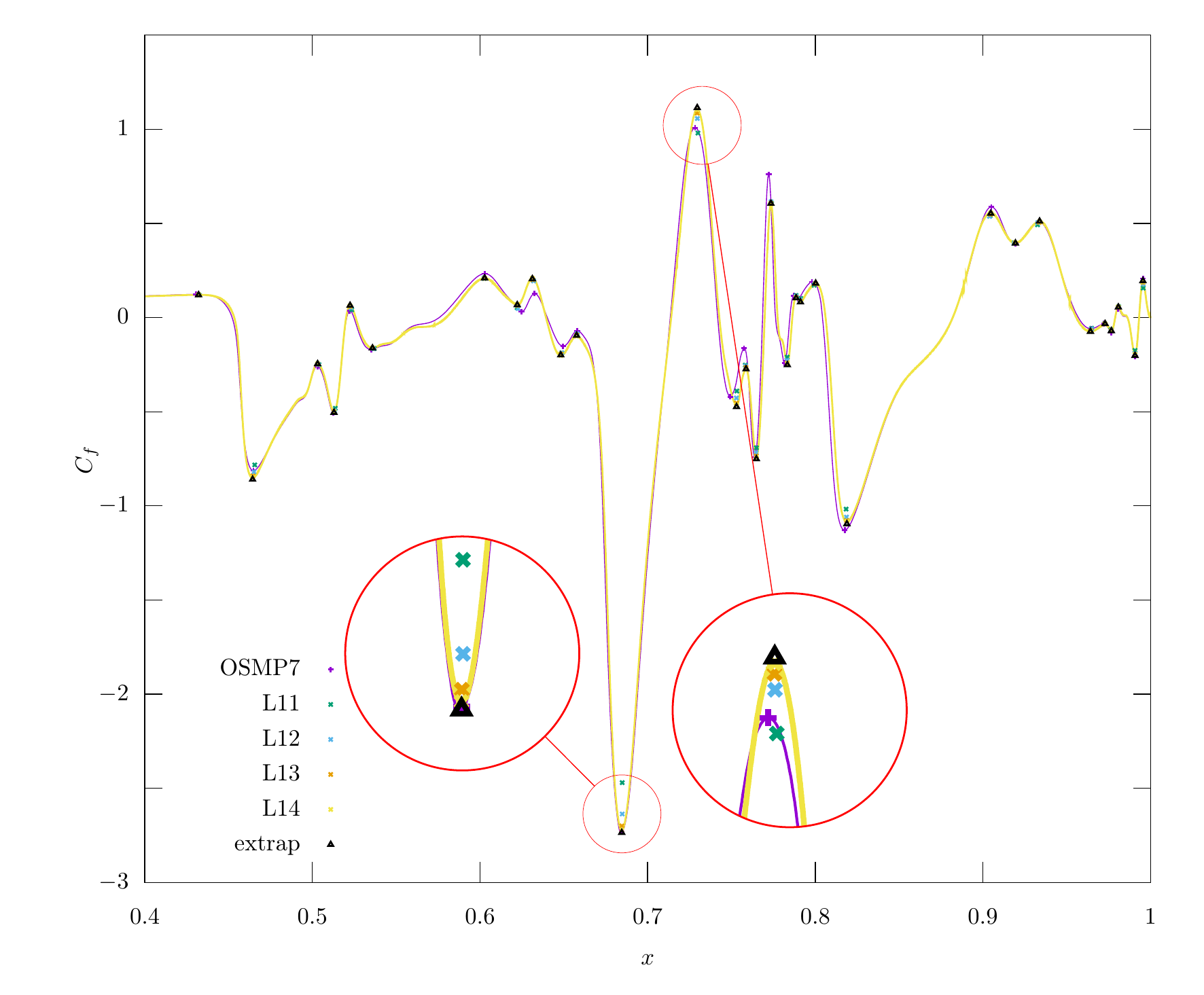}}
  \end{center}
  \vspace{-3em}
  \caption{The 2D shocktube benchmark: detailed plot of the skin friction
    coefficient $C\lof$ at time $t=1.00$. The continuous lines are for the
    finest level (L 14) of our computation and the OSMP7 scheme as reported
    in \citep{Daru_Tenaud_2009, Daru_Tenaud_2020}.
    The two insets show the convergence behavior in the global maximum and
    minimum, respectively.}
  \label{fig:shocktube-detailed}
\end{figure}

\begin{sidewaystable}
  \input{c_f-table.tex}  \caption{Computed extrema of the skin friction coefficient $C\lof$ for
    the 2D schocktube configuration for refinement levels L\,11 (8 M
    grid points) through L\,14 (512 M grid points). The extrapolated values
    are computed by fitting the linear function $a + b\,h$ to the
    four values obtained on the grids L\,11 to L\,14.
    We report the coefficient $a$ and the asymptotic standard
    error in the column labeled ``extrapolated''. Values for the OSMP7 scheme are
    taken from \cite{Daru_Tenaud_2009, Daru_Tenaud_2020}. The relative
    deviation reported in the last column is computed by
    $\big(C\lof(\text{OSMP7}) - C\lof(\text{extrap.})\big) /
    \big(\max\,C\lof(\text{extrap.}) - \min\,C\lof(\text{extrap.})\big)$}%
  \label{tab:shocktube_extrema}
\end{sidewaystable}

\clearpage
\bibliographystyle{abbrvnat}
\bibliography{ref_ns}

\end{document}

%% file: oat15a-3d-c_p.tex
\begin{tikzpicture}[gnuplot,xscale=0.55,yscale=0.55]
%% generated with GNUPLOT 5.4.1p1 (Gentoo revision r2) (Lua 5.1; terminal rev. Jun 2020, script rev. 114)
%% Thu 06 May 2021 04:28:11 PM CDT
\path (0.000,0.000) rectangle (12.500,8.750);
\gpcolor{color=gp lt color axes}
\gpsetlinetype{gp lt axes}
\gpsetdashtype{gp dt axes}
\gpsetlinewidth{0.50}
\draw[gp path] (1.504,8.441)--(11.947,8.441);
\gpcolor{color=gp lt color border}
\gpsetlinetype{gp lt border}
\gpsetdashtype{gp dt solid}
\gpsetlinewidth{1.00}
\draw[gp path] (1.504,8.441)--(1.684,8.441);
\draw[gp path] (11.947,8.441)--(11.767,8.441);
\node[gp node right] at (1.320,8.441) {\footnotesize $-2$};
\gpcolor{color=gp lt color axes}
\gpsetlinetype{gp lt axes}
\gpsetdashtype{gp dt axes}
\gpsetlinewidth{0.50}
\draw[gp path] (1.504,7.376)--(11.947,7.376);
\gpcolor{color=gp lt color border}
\gpsetlinetype{gp lt border}
\gpsetdashtype{gp dt solid}
\gpsetlinewidth{1.00}
\draw[gp path] (1.504,7.376)--(1.684,7.376);
\draw[gp path] (11.947,7.376)--(11.767,7.376);
\node[gp node right] at (1.320,7.376) {\footnotesize $-1.5$};
\gpcolor{color=gp lt color axes}
\gpsetlinetype{gp lt axes}
\gpsetdashtype{gp dt axes}
\gpsetlinewidth{0.50}
\draw[gp path] (1.504,6.311)--(11.947,6.311);
\gpcolor{color=gp lt color border}
\gpsetlinetype{gp lt border}
\gpsetdashtype{gp dt solid}
\gpsetlinewidth{1.00}
\draw[gp path] (1.504,6.311)--(1.684,6.311);
\draw[gp path] (11.947,6.311)--(11.767,6.311);
\node[gp node right] at (1.320,6.311) {\footnotesize $-1$};
\gpcolor{color=gp lt color axes}
\gpsetlinetype{gp lt axes}
\gpsetdashtype{gp dt axes}
\gpsetlinewidth{0.50}
\draw[gp path] (1.504,5.246)--(11.947,5.246);
\gpcolor{color=gp lt color border}
\gpsetlinetype{gp lt border}
\gpsetdashtype{gp dt solid}
\gpsetlinewidth{1.00}
\draw[gp path] (1.504,5.246)--(1.684,5.246);
\draw[gp path] (11.947,5.246)--(11.767,5.246);
\node[gp node right] at (1.320,5.246) {\footnotesize $-0.5$};
\gpcolor{color=gp lt color axes}
\gpsetlinetype{gp lt axes}
\gpsetdashtype{gp dt axes}
\gpsetlinewidth{0.50}
\draw[gp path] (1.504,4.180)--(11.947,4.180);
\gpcolor{color=gp lt color border}
\gpsetlinetype{gp lt border}
\gpsetdashtype{gp dt solid}
\gpsetlinewidth{1.00}
\draw[gp path] (1.504,4.180)--(1.684,4.180);
\draw[gp path] (11.947,4.180)--(11.767,4.180);
\node[gp node right] at (1.320,4.180) {\footnotesize $0$};
\gpcolor{color=gp lt color axes}
\gpsetlinetype{gp lt axes}
\gpsetdashtype{gp dt axes}
\gpsetlinewidth{0.50}
\draw[gp path] (1.504,3.115)--(11.947,3.115);
\gpcolor{color=gp lt color border}
\gpsetlinetype{gp lt border}
\gpsetdashtype{gp dt solid}
\gpsetlinewidth{1.00}
\draw[gp path] (1.504,3.115)--(1.684,3.115);
\draw[gp path] (11.947,3.115)--(11.767,3.115);
\node[gp node right] at (1.320,3.115) {\footnotesize $0.5$};
\gpcolor{color=gp lt color axes}
\gpsetlinetype{gp lt axes}
\gpsetdashtype{gp dt axes}
\gpsetlinewidth{0.50}
\draw[gp path] (1.504,2.050)--(11.947,2.050);
\gpcolor{color=gp lt color border}
\gpsetlinetype{gp lt border}
\gpsetdashtype{gp dt solid}
\gpsetlinewidth{1.00}
\draw[gp path] (1.504,2.050)--(1.684,2.050);
\draw[gp path] (11.947,2.050)--(11.767,2.050);
\node[gp node right] at (1.320,2.050) {\footnotesize $1$};
\gpcolor{color=gp lt color axes}
\gpsetlinetype{gp lt axes}
\gpsetdashtype{gp dt axes}
\gpsetlinewidth{0.50}
\draw[gp path] (1.504,0.985)--(11.947,0.985);
\gpcolor{color=gp lt color border}
\gpsetlinetype{gp lt border}
\gpsetdashtype{gp dt solid}
\gpsetlinewidth{1.00}
\draw[gp path] (1.504,0.985)--(1.684,0.985);
\draw[gp path] (11.947,0.985)--(11.767,0.985);
\node[gp node right] at (1.320,0.985) {\footnotesize $1.5$};
\gpcolor{color=gp lt color axes}
\gpsetlinetype{gp lt axes}
\gpsetdashtype{gp dt axes}
\gpsetlinewidth{0.50}
\draw[gp path] (1.504,0.985)--(1.504,8.441);
\gpcolor{color=gp lt color border}
\gpsetlinetype{gp lt border}
\gpsetdashtype{gp dt solid}
\gpsetlinewidth{1.00}
\draw[gp path] (1.504,0.985)--(1.504,1.165);
\draw[gp path] (1.504,8.441)--(1.504,8.261);
\node[gp node center] at (1.504,0.677) {\footnotesize $0$};
\gpcolor{color=gp lt color axes}
\gpsetlinetype{gp lt axes}
\gpsetdashtype{gp dt axes}
\gpsetlinewidth{0.50}
\draw[gp path] (3.593,0.985)--(3.593,8.441);
\gpcolor{color=gp lt color border}
\gpsetlinetype{gp lt border}
\gpsetdashtype{gp dt solid}
\gpsetlinewidth{1.00}
\draw[gp path] (3.593,0.985)--(3.593,1.165);
\draw[gp path] (3.593,8.441)--(3.593,8.261);
\node[gp node center] at (3.593,0.677) {\footnotesize $0.2$};
\gpcolor{color=gp lt color axes}
\gpsetlinetype{gp lt axes}
\gpsetdashtype{gp dt axes}
\gpsetlinewidth{0.50}
\draw[gp path] (5.681,0.985)--(5.681,8.441);
\gpcolor{color=gp lt color border}
\gpsetlinetype{gp lt border}
\gpsetdashtype{gp dt solid}
\gpsetlinewidth{1.00}
\draw[gp path] (5.681,0.985)--(5.681,1.165);
\draw[gp path] (5.681,8.441)--(5.681,8.261);
\node[gp node center] at (5.681,0.677) {\footnotesize $0.4$};
\gpcolor{color=gp lt color axes}
\gpsetlinetype{gp lt axes}
\gpsetdashtype{gp dt axes}
\gpsetlinewidth{0.50}
\draw[gp path] (7.770,0.985)--(7.770,8.441);
\gpcolor{color=gp lt color border}
\gpsetlinetype{gp lt border}
\gpsetdashtype{gp dt solid}
\gpsetlinewidth{1.00}
\draw[gp path] (7.770,0.985)--(7.770,1.165);
\draw[gp path] (7.770,8.441)--(7.770,8.261);
\node[gp node center] at (7.770,0.677) {\footnotesize $0.6$};
\gpcolor{color=gp lt color axes}
\gpsetlinetype{gp lt axes}
\gpsetdashtype{gp dt axes}
\gpsetlinewidth{0.50}
\draw[gp path] (9.858,0.985)--(9.858,7.645);
\draw[gp path] (9.858,8.261)--(9.858,8.441);
\gpcolor{color=gp lt color border}
\gpsetlinetype{gp lt border}
\gpsetdashtype{gp dt solid}
\gpsetlinewidth{1.00}
\draw[gp path] (9.858,0.985)--(9.858,1.165);
\draw[gp path] (9.858,8.441)--(9.858,8.261);
\node[gp node center] at (9.858,0.677) {\footnotesize $0.8$};
\gpcolor{color=gp lt color axes}
\gpsetlinetype{gp lt axes}
\gpsetdashtype{gp dt axes}
\gpsetlinewidth{0.50}
\draw[gp path] (11.947,0.985)--(11.947,8.441);
\gpcolor{color=gp lt color border}
\gpsetlinetype{gp lt border}
\gpsetdashtype{gp dt solid}
\gpsetlinewidth{1.00}
\draw[gp path] (11.947,0.985)--(11.947,1.165);
\draw[gp path] (11.947,8.441)--(11.947,8.261);
\node[gp node center] at (11.947,0.677) {\footnotesize $1$};
\draw[gp path] (1.504,8.441)--(1.504,0.985)--(11.947,0.985)--(11.947,8.441)--cycle;
\node[gp node center,rotate=-270] at (-0.2,4.713) {\footnotesize $C\lop$};
\node[gp node center] at (6.725,0.215) {\footnotesize $x/c$};
%\node[gp node right] at (10.479,8.107) {computed $C\lop$};
%\gpcolor{rgb color={0.580,0.000,0.827}}
\gpcolor{rgb color={0.000,0.620,0.451}}
%\draw[gp path] (10.663,8.107)--(11.579,8.107);
\draw[gp path] (1.504,2.386)--(1.507,2.525)--(1.510,2.666)--(1.513,2.810)--(1.517,2.955)%
  --(1.521,3.102)--(1.525,3.249)--(1.530,3.395)--(1.534,3.540)--(1.539,3.683)--(1.544,3.824)%
  --(1.549,3.961)--(1.554,4.094)--(1.560,4.224)--(1.566,4.348)--(1.571,4.467)--(1.577,4.580)%
  --(1.583,4.688)--(1.589,4.791)--(1.595,4.893)--(1.602,4.989)--(1.608,5.078)--(1.615,5.163)%
  --(1.621,5.245)--(1.628,5.323)--(1.634,5.397)--(1.641,5.468)--(1.648,5.535)--(1.655,5.599)%
  --(1.662,5.660)--(1.669,5.718)--(1.676,5.773)--(1.683,5.825)--(1.690,5.874)--(1.698,5.921)%
  --(1.705,5.966)--(1.713,6.009)--(1.720,6.050)--(1.728,6.090)--(1.736,6.127)--(1.744,6.163)%
  --(1.752,6.197)--(1.760,6.230)--(1.768,6.260)--(1.777,6.290)--(1.786,6.317)--(1.795,6.347)%
  --(1.804,6.379)--(1.815,6.394)--(1.829,6.425)--(1.843,6.463)--(1.856,6.493)--(1.870,6.527)%
  --(1.884,6.555)--(1.898,6.584)--(1.911,6.609)--(1.925,6.633)--(1.939,6.656)--(1.953,6.677)%
  --(1.966,6.696)--(1.980,6.715)--(1.994,6.732)--(2.008,6.749)--(2.021,6.764)--(2.035,6.778)%
  --(2.049,6.792)--(2.063,6.805)--(2.076,6.818)--(2.090,6.829)--(2.104,6.841)--(2.118,6.851)%
  --(2.131,6.862)--(2.145,6.872)--(2.159,6.881)--(2.172,6.890)--(2.186,6.899)--(2.200,6.907)%
  --(2.214,6.916)--(2.227,6.923)--(2.241,6.931)--(2.255,6.938)--(2.269,6.945)--(2.282,6.952)%
  --(2.296,6.958)--(2.310,6.964)--(2.324,6.970)--(2.337,6.976)--(2.351,6.981)--(2.365,6.986)%
  --(2.379,6.991)--(2.392,6.996)--(2.406,6.999)--(2.420,7.003)--(2.434,7.006)--(2.447,7.010)%
  --(2.461,7.013)--(2.475,7.016)--(2.489,7.020)--(2.502,7.023)--(2.516,7.026)--(2.530,7.029)%
  --(2.544,7.033)--(2.557,7.036)--(2.571,7.040)--(2.585,7.043)--(2.599,7.047)--(2.612,7.051)%
  --(2.626,7.054)--(2.640,7.058)--(2.654,7.062)--(2.667,7.065)--(2.681,7.069)--(2.695,7.072)%
  --(2.709,7.076)--(2.722,7.080)--(2.736,7.084)--(2.750,7.087)--(2.763,7.091)--(2.777,7.095)%
  --(2.791,7.099)--(2.805,7.102)--(2.818,7.106)--(2.832,7.109)--(2.846,7.113)--(2.860,7.116)%
  --(2.873,7.120)--(2.887,7.123)--(2.901,7.127)--(2.915,7.130)--(2.928,7.133)--(2.942,7.137)%
  --(2.956,7.140)--(2.970,7.143)--(2.983,7.145)--(2.997,7.148)--(3.011,7.150)--(3.025,7.152)%
  --(3.038,7.154)--(3.052,7.156)--(3.066,7.157)--(3.080,7.158)--(3.093,7.159)--(3.107,7.160)%
  --(3.121,7.160)--(3.135,7.161)--(3.148,7.161)--(3.162,7.162)--(3.176,7.162)--(3.190,7.162)%
  --(3.203,7.163)--(3.217,7.163)--(3.231,7.164)--(3.245,7.164)--(3.258,7.164)--(3.272,7.165)%
  --(3.286,7.165)--(3.300,7.165)--(3.313,7.165)--(3.327,7.165)--(3.341,7.164)--(3.355,7.164)%
  --(3.368,7.164)--(3.382,7.164)--(3.396,7.164)--(3.409,7.164)--(3.423,7.163)--(3.437,7.163)%
  --(3.451,7.163)--(3.464,7.163)--(3.478,7.163)--(3.492,7.163)--(3.506,7.162)--(3.519,7.162)%
  --(3.533,7.161)--(3.547,7.161)--(3.561,7.160)--(3.574,7.160)--(3.588,7.159)--(3.602,7.158)%
  --(3.616,7.158)--(3.629,7.157)--(3.643,7.156)--(3.657,7.156)--(3.671,7.155)--(3.684,7.154)%
  --(3.698,7.154)--(3.712,7.153)--(3.726,7.153)--(3.739,7.152)--(3.753,7.151)--(3.767,7.150)%
  --(3.781,7.149)--(3.794,7.148)--(3.808,7.147)--(3.822,7.146)--(3.836,7.145)--(3.849,7.144)%
  --(3.863,7.143)--(3.877,7.142)--(3.891,7.141)--(3.904,7.140)--(3.918,7.139)--(3.932,7.138)%
  --(3.946,7.137)--(3.959,7.136)--(3.973,7.135)--(3.987,7.134)--(4.000,7.133)--(4.014,7.132)%
  --(4.028,7.131)--(4.042,7.130)--(4.055,7.129)--(4.069,7.128)--(4.083,7.127)--(4.097,7.126)%
  --(4.110,7.125)--(4.124,7.125)--(4.138,7.124)--(4.152,7.123)--(4.165,7.122)--(4.179,7.121)%
  --(4.193,7.120)--(4.207,7.119)--(4.220,7.119)--(4.234,7.118)--(4.248,7.117)--(4.262,7.116)%
  --(4.275,7.116)--(4.289,7.115)--(4.303,7.114)--(4.317,7.113)--(4.330,7.113)--(4.344,7.112)%
  --(4.358,7.111)--(4.372,7.110)--(4.385,7.109)--(4.399,7.108)--(4.413,7.108)--(4.427,7.107)%
  --(4.440,7.106)--(4.454,7.105)--(4.468,7.105)--(4.482,7.104)--(4.495,7.103)--(4.509,7.102)%
  --(4.523,7.102)--(4.537,7.101)--(4.550,7.100)--(4.564,7.099)--(4.578,7.098)--(4.591,7.098)%
  --(4.591,7.097)--(4.591,7.098)--(4.591,7.097)--(4.591,7.098)--(4.591,7.097)--(4.591,7.098)%
  --(4.606,7.097)--(4.620,7.096)--(4.635,7.095)--(4.649,7.094)--(4.663,7.094)--(4.678,7.093)%
  --(4.692,7.092)--(4.706,7.091)--(4.721,7.091)--(4.735,7.090)--(4.749,7.089)--(4.764,7.089)%
  --(4.778,7.088)--(4.792,7.087)--(4.807,7.086)--(4.821,7.086)--(4.835,7.085)--(4.850,7.084)%
  --(4.864,7.083)--(4.878,7.082)--(4.893,7.082)--(4.907,7.081)--(4.922,7.080)--(4.936,7.079)%
  --(4.950,7.078)--(4.965,7.077)--(4.979,7.077)--(4.993,7.076)--(5.008,7.075)--(5.022,7.074)%
  --(5.036,7.073)--(5.051,7.072)--(5.065,7.071)--(5.079,7.070)--(5.094,7.069)--(5.108,7.068)%
  --(5.122,7.067)--(5.137,7.066)--(5.151,7.064)--(5.165,7.063)--(5.180,7.062)--(5.194,7.061)%
  --(5.208,7.059)--(5.223,7.058)--(5.237,7.056)--(5.252,7.054)--(5.266,7.052)--(5.280,7.051)%
  --(5.280,7.050)--(5.295,7.049)--(5.309,7.046)--(5.323,7.044)--(5.338,7.042)--(5.352,7.040)%
  --(5.366,7.037)--(5.381,7.034)--(5.395,7.031)--(5.409,7.028)--(5.424,7.025)--(5.438,7.022)%
  --(5.452,7.018)--(5.467,7.014)--(5.481,7.011)--(5.495,7.007)--(5.510,7.002)--(5.524,6.998)%
  --(5.538,6.993)--(5.553,6.988)--(5.567,6.983)--(5.582,6.978)--(5.596,6.972)--(5.610,6.966)%
  --(5.625,6.960)--(5.639,6.954)--(5.653,6.948)--(5.668,6.941)--(5.682,6.934)--(5.696,6.927)%
  --(5.711,6.920)--(5.725,6.913)--(5.739,6.905)--(5.754,6.897)--(5.768,6.889)--(5.782,6.880)%
  --(5.797,6.872)--(5.811,6.863)--(5.825,6.854)--(5.840,6.845)--(5.854,6.835)--(5.869,6.826)%
  --(5.883,6.816)--(5.897,6.806)--(5.912,6.796)--(5.926,6.786)--(5.940,6.776)--(5.955,6.766)%
  --(5.969,6.755)--(5.983,6.745)--(5.998,6.735)--(6.012,6.724)--(6.026,6.714)--(6.041,6.703)%
  --(6.055,6.693)--(6.069,6.682)--(6.084,6.672)--(6.098,6.661)--(6.112,6.651)--(6.127,6.640)%
  --(6.141,6.629)--(6.155,6.618)--(6.170,6.608)--(6.184,6.597)--(6.199,6.586)--(6.213,6.575)%
  --(6.227,6.565)--(6.242,6.554)--(6.256,6.544)--(6.270,6.533)--(6.285,6.523)--(6.299,6.512)%
  --(6.313,6.502)--(6.328,6.492)--(6.342,6.481)--(6.356,6.471)--(6.371,6.461)--(6.385,6.451)%
  --(6.399,6.440)--(6.414,6.430)--(6.428,6.420)--(6.442,6.410)--(6.457,6.399)--(6.471,6.389)%
  --(6.485,6.379)--(6.500,6.368)--(6.514,6.358)--(6.529,6.348)--(6.543,6.337)--(6.557,6.327)%
  --(6.572,6.316)--(6.586,6.306)--(6.600,6.295)--(6.615,6.284)--(6.629,6.273)--(6.643,6.262)%
  --(6.658,6.251)--(6.672,6.240)--(6.686,6.228)--(6.701,6.216)--(6.715,6.205)--(6.729,6.193)%
  --(6.744,6.181)--(6.758,6.169)--(6.772,6.157)--(6.787,6.145)--(6.801,6.133)--(6.816,6.121)%
  --(6.830,6.109)--(6.844,6.098)--(6.859,6.087)--(6.873,6.075)--(6.887,6.065)--(6.902,6.054)%
  --(6.916,6.043)--(6.930,6.033)--(6.945,6.022)--(6.959,6.012)--(6.973,6.002)--(6.988,5.991)%
  --(7.002,5.981)--(7.016,5.971)--(7.031,5.961)--(7.045,5.951)--(7.059,5.941)--(7.074,5.931)%
  --(7.088,5.922)--(7.102,5.913)--(7.117,5.903)--(7.131,5.894)--(7.146,5.885)--(7.160,5.876)%
  --(7.174,5.867)--(7.189,5.858)--(7.203,5.849)--(7.217,5.840)--(7.232,5.832)--(7.246,5.823)%
  --(7.260,5.815)--(7.275,5.807)--(7.289,5.799)--(7.303,5.791)--(7.318,5.783)--(7.332,5.776)%
  --(7.346,5.768)--(7.361,5.761)--(7.375,5.753)--(7.389,5.745)--(7.404,5.738)--(7.418,5.730)%
  --(7.432,5.723)--(7.447,5.716)--(7.461,5.709)--(7.476,5.701)--(7.490,5.695)--(7.504,5.688)%
  --(7.519,5.681)--(7.533,5.675)--(7.547,5.668)--(7.562,5.662)--(7.576,5.655)--(7.590,5.649)%
  --(7.605,5.643)--(7.619,5.636)--(7.633,5.629)--(7.648,5.623)--(7.662,5.617)--(7.676,5.611)%
  --(7.691,5.605)--(7.705,5.599)--(7.719,5.593)--(7.734,5.588)--(7.748,5.582)--(7.763,5.576)%
  --(7.777,5.570)--(7.791,5.564)--(7.806,5.558)--(7.820,5.553)--(7.834,5.547)--(7.849,5.541)%
  --(7.863,5.536)--(7.877,5.530)--(7.892,5.525)--(7.906,5.520)--(7.920,5.514)--(7.935,5.509)%
  --(7.949,5.504)--(7.963,5.499)--(7.978,5.494)--(7.992,5.489)--(8.006,5.484)--(8.021,5.479)%
  --(8.035,5.474)--(8.049,5.469)--(8.064,5.464)--(8.078,5.460)--(8.093,5.455)--(8.107,5.450)%
  --(8.121,5.446)--(8.136,5.441)--(8.150,5.436)--(8.164,5.432)--(8.179,5.427)--(8.193,5.422)%
  --(8.207,5.417)--(8.222,5.413)--(8.236,5.408)--(8.250,5.404)--(8.265,5.399)--(8.279,5.395)%
  --(8.293,5.390)--(8.308,5.385)--(8.322,5.380)--(8.336,5.375)--(8.351,5.371)--(8.365,5.366)%
  --(8.379,5.361)--(8.394,5.356)--(8.408,5.351)--(8.423,5.346)--(8.437,5.342)--(8.451,5.337)%
  --(8.466,5.332)--(8.480,5.327)--(8.494,5.322)--(8.509,5.317)--(8.523,5.312)--(8.537,5.307)%
  --(8.552,5.303)--(8.566,5.298)--(8.580,5.293)--(8.595,5.288)--(8.609,5.283)--(8.623,5.279)%
  --(8.638,5.274)--(8.652,5.269)--(8.666,5.264)--(8.681,5.259)--(8.695,5.254)--(8.710,5.249)%
  --(8.724,5.244)--(8.738,5.239)--(8.753,5.234)--(8.767,5.228)--(8.781,5.223)--(8.796,5.218)%
  --(8.810,5.213)--(8.824,5.208)--(8.839,5.202)--(8.853,5.197)--(8.867,5.192)--(8.882,5.186)%
  --(8.896,5.181)--(8.910,5.175)--(8.925,5.170)--(8.939,5.164)--(8.953,5.158)--(8.968,5.153)%
  --(8.982,5.147)--(8.996,5.142)--(9.011,5.136)--(9.025,5.131)--(9.040,5.125)--(9.054,5.120)%
  --(9.068,5.114)--(9.083,5.108)--(9.097,5.102)--(9.111,5.096)--(9.126,5.091)--(9.140,5.085)%
  --(9.154,5.079)--(9.169,5.073)--(9.183,5.067)--(9.197,5.060)--(9.212,5.054)--(9.226,5.048)%
  --(9.240,5.041)--(9.255,5.035)--(9.269,5.029)--(9.283,5.022)--(9.298,5.016)--(9.312,5.009)%
  --(9.326,5.002)--(9.341,4.996)--(9.355,4.989)--(9.370,4.983)--(9.384,4.976)--(9.398,4.969)%
  --(9.413,4.962)--(9.427,4.955)--(9.441,4.949)--(9.456,4.942)--(9.470,4.935)--(9.484,4.928)%
  --(9.499,4.921)--(9.513,4.914)--(9.527,4.907)--(9.542,4.901)--(9.556,4.894)--(9.570,4.887)%
  --(9.585,4.880)--(9.599,4.873)--(9.613,4.866)--(9.628,4.859)--(9.642,4.852)--(9.657,4.845)%
  --(9.671,4.839)--(9.685,4.832)--(9.700,4.825)--(9.714,4.818)--(9.728,4.811)--(9.743,4.804)%
  --(9.757,4.797)--(9.771,4.790)--(9.786,4.783)--(9.800,4.776)--(9.814,4.770)--(9.829,4.763)%
  --(9.843,4.756)--(9.857,4.749)--(9.872,4.742)--(9.886,4.735)--(9.900,4.729)--(9.915,4.722)%
  --(9.929,4.716)--(9.943,4.710)--(9.958,4.703)--(9.972,4.697)--(9.987,4.691)--(10.001,4.685)%
  --(10.015,4.678)--(10.030,4.672)--(10.044,4.666)--(10.058,4.660)--(10.073,4.653)--(10.087,4.647)%
  --(10.101,4.641)--(10.116,4.635)--(10.130,4.629)--(10.144,4.622)--(10.159,4.616)--(10.173,4.609)%
  --(10.187,4.603)--(10.202,4.596)--(10.216,4.590)--(10.230,4.583)--(10.245,4.577)--(10.259,4.570)%
  --(10.273,4.564)--(10.288,4.558)--(10.302,4.552)--(10.317,4.546)--(10.331,4.540)--(10.345,4.534)%
  --(10.360,4.529)--(10.374,4.523)--(10.388,4.517)--(10.403,4.511)--(10.417,4.506)--(10.431,4.500)%
  --(10.446,4.494)--(10.460,4.489)--(10.474,4.483)--(10.489,4.477)--(10.503,4.472)--(10.517,4.466)%
  --(10.532,4.460)--(10.546,4.454)--(10.560,4.449)--(10.575,4.443)--(10.589,4.437)--(10.604,4.432)%
  --(10.618,4.427)--(10.632,4.422)--(10.647,4.417)--(10.661,4.412)--(10.675,4.407)--(10.690,4.401)%
  --(10.704,4.396)--(10.718,4.391)--(10.733,4.386)--(10.747,4.381)--(10.761,4.376)--(10.776,4.371)%
  --(10.790,4.366)--(10.804,4.361)--(10.819,4.356)--(10.833,4.351)--(10.847,4.346)--(10.862,4.342)%
  --(10.876,4.337)--(10.890,4.332)--(10.905,4.327)--(10.919,4.323)--(10.934,4.318)--(10.948,4.313)%
  --(10.962,4.308)--(10.977,4.303)--(10.991,4.298)--(11.005,4.293)--(11.020,4.288)--(11.034,4.284)%
  --(11.048,4.279)--(11.063,4.274)--(11.077,4.269)--(11.091,4.265)--(11.106,4.260)--(11.120,4.255)%
  --(11.134,4.251)--(11.149,4.246)--(11.163,4.241)--(11.177,4.237)--(11.192,4.232)--(11.206,4.227)%
  --(11.220,4.222)--(11.235,4.218)--(11.249,4.213)--(11.264,4.209)--(11.278,4.204)--(11.292,4.200)%
  --(11.307,4.195)--(11.321,4.191)--(11.335,4.186)--(11.350,4.182)--(11.364,4.177)--(11.378,4.173)%
  --(11.393,4.169)--(11.407,4.165)--(11.421,4.160)--(11.436,4.156)--(11.450,4.152)--(11.464,4.148)%
  --(11.479,4.143)--(11.493,4.139)--(11.507,4.135)--(11.522,4.131)--(11.536,4.126)--(11.551,4.122)%
  --(11.565,4.118)--(11.579,4.114)--(11.594,4.109)--(11.608,4.105)--(11.622,4.101)--(11.637,4.097)%
  --(11.651,4.093)--(11.665,4.089)--(11.680,4.085)--(11.694,4.081)--(11.708,4.077)--(11.723,4.074)%
  --(11.737,4.069)--(11.751,4.066)--(11.766,4.061)--(11.780,4.058)--(11.794,4.053)--(11.809,4.050)%
  --(11.823,4.045)--(11.837,4.043)--(11.852,4.038)--(11.866,4.036)--(11.881,4.031)--(11.895,4.030)%
  --(11.909,4.026)--(11.924,4.026)--(11.938,4.034);
\draw[gp path] (1.504,1.452)--(1.508,1.478)--(1.512,1.527)--(1.517,1.600)--(1.523,1.680)%
  --(1.529,1.756)--(1.536,1.834)--(1.542,1.904)--(1.549,1.975)--(1.556,2.045)--(1.563,2.106)%
  --(1.570,2.162)--(1.577,2.215)--(1.584,2.266)--(1.591,2.313)--(1.598,2.357)--(1.605,2.398)%
  --(1.611,2.436)--(1.618,2.472)--(1.625,2.505)--(1.631,2.538)--(1.637,2.568)--(1.644,2.597)%
  --(1.650,2.624)--(1.656,2.651)--(1.662,2.676)--(1.669,2.700)--(1.675,2.723)--(1.681,2.745)%
  --(1.687,2.767)--(1.693,2.788)--(1.699,2.808)--(1.705,2.828)--(1.710,2.846)--(1.716,2.866)%
  --(1.722,2.884)--(1.728,2.902)--(1.734,2.919)--(1.740,2.937)--(1.746,2.953)--(1.753,2.970)%
  --(1.759,2.987)--(1.765,3.003)--(1.771,3.019)--(1.778,3.035)--(1.784,3.051)--(1.791,3.067)%
  --(1.798,3.084)--(1.805,3.103)--(1.815,3.122)--(1.829,3.152)--(1.843,3.189)--(1.856,3.226)%
  --(1.870,3.261)--(1.884,3.295)--(1.898,3.327)--(1.911,3.359)--(1.925,3.390)--(1.939,3.420)%
  --(1.953,3.449)--(1.966,3.477)--(1.980,3.505)--(1.994,3.531)--(2.008,3.558)--(2.021,3.583)%
  --(2.035,3.608)--(2.049,3.632)--(2.063,3.655)--(2.076,3.678)--(2.090,3.701)--(2.104,3.722)%
  --(2.118,3.743)--(2.131,3.764)--(2.145,3.783)--(2.159,3.802)--(2.172,3.820)--(2.186,3.838)%
  --(2.200,3.855)--(2.214,3.872)--(2.227,3.889)--(2.241,3.905)--(2.255,3.921)--(2.269,3.937)%
  --(2.282,3.952)--(2.296,3.967)--(2.310,3.981)--(2.324,3.995)--(2.337,4.008)--(2.351,4.021)%
  --(2.365,4.033)--(2.379,4.045)--(2.392,4.056)--(2.406,4.067)--(2.420,4.078)--(2.434,4.089)%
  --(2.447,4.100)--(2.461,4.111)--(2.475,4.122)--(2.489,4.134)--(2.502,4.145)--(2.516,4.156)%
  --(2.530,4.167)--(2.544,4.179)--(2.557,4.190)--(2.571,4.200)--(2.585,4.211)--(2.599,4.222)%
  --(2.612,4.233)--(2.626,4.243)--(2.640,4.254)--(2.654,4.265)--(2.667,4.275)--(2.681,4.285)%
  --(2.695,4.296)--(2.709,4.306)--(2.722,4.316)--(2.736,4.326)--(2.750,4.336)--(2.763,4.347)%
  --(2.777,4.357)--(2.791,4.367)--(2.805,4.377)--(2.818,4.388)--(2.832,4.397)--(2.846,4.408)%
  --(2.860,4.419)--(2.873,4.429)--(2.887,4.439)--(2.901,4.450)--(2.915,4.460)--(2.928,4.471)%
  --(2.942,4.481)--(2.956,4.491)--(2.970,4.502)--(2.983,4.512)--(2.997,4.522)--(3.011,4.532)%
  --(3.025,4.542)--(3.038,4.552)--(3.052,4.561)--(3.066,4.569)--(3.080,4.578)--(3.093,4.586)%
  --(3.107,4.592)--(3.121,4.599)--(3.135,4.604)--(3.148,4.609)--(3.162,4.613)--(3.176,4.617)%
  --(3.190,4.620)--(3.203,4.623)--(3.217,4.625)--(3.231,4.628)--(3.245,4.630)--(3.258,4.633)%
  --(3.272,4.636)--(3.286,4.639)--(3.300,4.642)--(3.313,4.645)--(3.327,4.648)--(3.341,4.651)%
  --(3.355,4.654)--(3.368,4.657)--(3.382,4.660)--(3.396,4.663)--(3.409,4.667)--(3.423,4.670)%
  --(3.437,4.673)--(3.451,4.677)--(3.464,4.680)--(3.478,4.683)--(3.492,4.687)--(3.506,4.690)%
  --(3.519,4.693)--(3.533,4.697)--(3.547,4.700)--(3.561,4.704)--(3.574,4.707)--(3.588,4.711)%
  --(3.602,4.714)--(3.616,4.718)--(3.629,4.721)--(3.643,4.724)--(3.657,4.727)--(3.671,4.731)%
  --(3.684,4.734)--(3.698,4.737)--(3.712,4.740)--(3.726,4.743)--(3.739,4.746)--(3.753,4.750)%
  --(3.767,4.753)--(3.781,4.756)--(3.794,4.759)--(3.808,4.762)--(3.822,4.764)--(3.836,4.767)%
  --(3.849,4.770)--(3.863,4.773)--(3.877,4.775)--(3.891,4.778)--(3.904,4.781)--(3.918,4.783)%
  --(3.932,4.786)--(3.946,4.789)--(3.959,4.791)--(3.973,4.793)--(3.987,4.796)--(4.000,4.798)%
  --(4.014,4.800)--(4.028,4.802)--(4.042,4.804)--(4.055,4.807)--(4.069,4.809)--(4.083,4.811)%
  --(4.097,4.813)--(4.110,4.815)--(4.124,4.817)--(4.138,4.819)--(4.152,4.821)--(4.165,4.822)%
  --(4.179,4.824)--(4.193,4.825)--(4.207,4.827)--(4.220,4.828)--(4.234,4.830)--(4.248,4.831)%
  --(4.262,4.833)--(4.275,4.834)--(4.289,4.835)--(4.303,4.836)--(4.317,4.837)--(4.330,4.839)%
  --(4.344,4.839)--(4.358,4.840)--(4.372,4.841)--(4.385,4.842)--(4.399,4.843)--(4.413,4.844)%
  --(4.427,4.844)--(4.440,4.845)--(4.454,4.846)--(4.468,4.846)--(4.482,4.847)--(4.495,4.847)%
  --(4.509,4.847)--(4.523,4.847)--(4.537,4.848)--(4.550,4.848)--(4.564,4.848)--(4.578,4.847)%
  --(4.591,4.847)--(4.606,4.847)--(4.620,4.847)--(4.635,4.846)--(4.649,4.846)--(4.663,4.846)%
  --(4.678,4.845)--(4.692,4.844)--(4.706,4.843)--(4.721,4.842)--(4.735,4.842)--(4.749,4.841)%
  --(4.764,4.839)--(4.778,4.838)--(4.792,4.837)--(4.807,4.836)--(4.821,4.835)--(4.835,4.834)%
  --(4.850,4.832)--(4.864,4.831)--(4.878,4.829)--(4.893,4.827)--(4.907,4.825)--(4.922,4.824)%
  --(4.936,4.822)--(4.950,4.820)--(4.965,4.818)--(4.979,4.816)--(4.993,4.814)--(5.008,4.812)%
  --(5.022,4.809)--(5.036,4.807)--(5.051,4.805)--(5.065,4.802)--(5.079,4.800)--(5.094,4.797)%
  --(5.108,4.794)--(5.122,4.792)--(5.137,4.789)--(5.151,4.786)--(5.165,4.783)--(5.180,4.780)%
  --(5.194,4.777)--(5.208,4.774)--(5.223,4.771)--(5.237,4.768)--(5.252,4.764)--(5.266,4.761)%
  --(5.280,4.758)--(5.295,4.754)--(5.309,4.751)--(5.323,4.747)--(5.338,4.744)--(5.352,4.740)%
  --(5.366,4.736)--(5.381,4.733)--(5.395,4.729)--(5.409,4.725)--(5.424,4.721)--(5.438,4.717)%
  --(5.452,4.713)--(5.467,4.709)--(5.481,4.705)--(5.495,4.701)--(5.510,4.697)--(5.524,4.693)%
  --(5.538,4.688)--(5.553,4.684)--(5.567,4.680)--(5.582,4.675)--(5.596,4.671)--(5.610,4.667)%
  --(5.625,4.662)--(5.639,4.658)--(5.653,4.653)--(5.668,4.649)--(5.682,4.644)--(5.696,4.640)%
  --(5.711,4.635)--(5.725,4.630)--(5.739,4.626)--(5.754,4.621)--(5.768,4.617)--(5.782,4.612)%
  --(5.797,4.608)--(5.811,4.603)--(5.825,4.599)--(5.840,4.594)--(5.854,4.590)--(5.869,4.585)%
  --(5.883,4.581)--(5.897,4.577)--(5.912,4.572)--(5.926,4.568)--(5.940,4.564)--(5.955,4.560)%
  --(5.969,4.556)--(5.983,4.551)--(5.998,4.547)--(6.012,4.543)--(6.026,4.539)--(6.041,4.535)%
  --(6.055,4.532)--(6.069,4.528)--(6.084,4.524)--(6.098,4.520)--(6.112,4.516)--(6.127,4.513)%
  --(6.141,4.509)--(6.155,4.505)--(6.170,4.502)--(6.184,4.498)--(6.199,4.495)--(6.213,4.491)%
  --(6.227,4.488)--(6.242,4.484)--(6.256,4.481)--(6.270,4.478)--(6.285,4.474)--(6.299,4.471)%
  --(6.313,4.468)--(6.328,4.465)--(6.342,4.462)--(6.356,4.459)--(6.371,4.456)--(6.385,4.453)%
  --(6.399,4.450)--(6.414,4.447)--(6.428,4.445)--(6.442,4.442)--(6.457,4.439)--(6.471,4.437)%
  --(6.485,4.434)--(6.500,4.431)--(6.514,4.429)--(6.529,4.426)--(6.543,4.424)--(6.557,4.421)%
  --(6.572,4.419)--(6.586,4.417)--(6.600,4.414)--(6.615,4.412)--(6.629,4.409)--(6.643,4.407)%
  --(6.658,4.405)--(6.672,4.402)--(6.686,4.400)--(6.701,4.398)--(6.715,4.395)--(6.729,4.393)%
  --(6.744,4.391)--(6.758,4.388)--(6.772,4.386)--(6.787,4.384)--(6.801,4.382)--(6.816,4.379)%
  --(6.830,4.377)--(6.844,4.375)--(6.859,4.372)--(6.873,4.370)--(6.887,4.368)--(6.902,4.366)%
  --(6.916,4.363)--(6.930,4.361)--(6.945,4.359)--(6.959,4.356)--(6.973,4.354)--(6.988,4.351)%
  --(7.002,4.349)--(7.016,4.347)--(7.031,4.344)--(7.045,4.342)--(7.059,4.339)--(7.074,4.337)%
  --(7.088,4.334)--(7.102,4.332)--(7.117,4.329)--(7.131,4.326)--(7.146,4.324)--(7.160,4.321)%
  --(7.174,4.318)--(7.189,4.316)--(7.203,4.313)--(7.217,4.310)--(7.232,4.307)--(7.246,4.304)%
  --(7.260,4.301)--(7.275,4.299)--(7.289,4.296)--(7.303,4.293)--(7.318,4.290)--(7.332,4.287)%
  --(7.346,4.283)--(7.361,4.280)--(7.375,4.277)--(7.389,4.274)--(7.404,4.270)--(7.418,4.267)%
  --(7.432,4.264)--(7.447,4.260)--(7.461,4.257)--(7.476,4.253)--(7.490,4.250)--(7.504,4.246)%
  --(7.519,4.242)--(7.533,4.238)--(7.547,4.234)--(7.562,4.231)--(7.576,4.227)--(7.590,4.223)%
  --(7.605,4.219)--(7.619,4.215)--(7.633,4.211)--(7.648,4.206)--(7.662,4.202)--(7.676,4.198)%
  --(7.691,4.193)--(7.705,4.189)--(7.719,4.184)--(7.734,4.180)--(7.748,4.175)--(7.763,4.170)%
  --(7.777,4.166)--(7.791,4.161)--(7.806,4.156)--(7.820,4.151)--(7.834,4.146)--(7.849,4.141)%
  --(7.863,4.136)--(7.877,4.131)--(7.892,4.125)--(7.906,4.120)--(7.920,4.115)--(7.935,4.109)%
  --(7.949,4.104)--(7.963,4.098)--(7.978,4.092)--(7.992,4.086)--(8.006,4.081)--(8.021,4.075)%
  --(8.035,4.069)--(8.049,4.063)--(8.064,4.057)--(8.078,4.051)--(8.093,4.044)--(8.107,4.038)%
  --(8.121,4.032)--(8.136,4.025)--(8.150,4.019)--(8.164,4.012)--(8.179,4.006)--(8.193,3.999)%
  --(8.207,3.992)--(8.222,3.985)--(8.236,3.978)--(8.250,3.971)--(8.265,3.965)--(8.279,3.958)%
  --(8.293,3.950)--(8.308,3.943)--(8.322,3.936)--(8.336,3.928)--(8.351,3.921)--(8.365,3.914)%
  --(8.379,3.906)--(8.394,3.899)--(8.408,3.891)--(8.423,3.884)--(8.437,3.877)--(8.451,3.869)%
  --(8.466,3.861)--(8.480,3.854)--(8.494,3.846)--(8.509,3.839)--(8.523,3.831)--(8.537,3.823)%
  --(8.552,3.816)--(8.566,3.808)--(8.580,3.800)--(8.595,3.793)--(8.609,3.785)--(8.623,3.777)%
  --(8.638,3.770)--(8.652,3.762)--(8.666,3.754)--(8.681,3.747)--(8.695,3.739)--(8.710,3.732)%
  --(8.724,3.724)--(8.738,3.717)--(8.753,3.710)--(8.767,3.702)--(8.781,3.695)--(8.796,3.688)%
  --(8.810,3.680)--(8.824,3.673)--(8.839,3.666)--(8.853,3.659)--(8.867,3.651)--(8.882,3.644)%
  --(8.896,3.637)--(8.910,3.630)--(8.925,3.623)--(8.939,3.616)--(8.953,3.609)--(8.968,3.602)%
  --(8.982,3.595)--(8.996,3.588)--(9.011,3.581)--(9.025,3.574)--(9.040,3.568)--(9.054,3.561)%
  --(9.068,3.554)--(9.083,3.548)--(9.097,3.542)--(9.111,3.535)--(9.126,3.529)--(9.140,3.523)%
  --(9.154,3.517)--(9.169,3.511)--(9.183,3.505)--(9.197,3.499)--(9.212,3.494)--(9.226,3.488)%
  --(9.240,3.483)--(9.255,3.478)--(9.269,3.473)--(9.283,3.467)--(9.298,3.462)--(9.312,3.457)%
  --(9.326,3.452)--(9.341,3.447)--(9.355,3.442)--(9.370,3.437)--(9.384,3.432)--(9.398,3.428)%
  --(9.413,3.423)--(9.427,3.419)--(9.441,3.414)--(9.456,3.410)--(9.470,3.406)--(9.484,3.402)%
  --(9.499,3.397)--(9.513,3.393)--(9.527,3.389)--(9.542,3.386)--(9.556,3.382)--(9.570,3.378)%
  --(9.585,3.374)--(9.599,3.371)--(9.613,3.367)--(9.628,3.364)--(9.642,3.360)--(9.657,3.357)%
  --(9.671,3.354)--(9.685,3.351)--(9.700,3.347)--(9.714,3.344)--(9.728,3.341)--(9.743,3.339)%
  --(9.757,3.336)--(9.771,3.333)--(9.786,3.330)--(9.800,3.327)--(9.814,3.324)--(9.829,3.322)%
  --(9.843,3.319)--(9.857,3.317)--(9.872,3.314)--(9.886,3.312)--(9.900,3.310)--(9.915,3.308)%
  --(9.929,3.306)--(9.943,3.304)--(9.958,3.302)--(9.972,3.300)--(9.987,3.299)--(10.001,3.297)%
  --(10.015,3.295)--(10.030,3.294)--(10.044,3.292)--(10.058,3.291)--(10.073,3.290)--(10.087,3.288)%
  --(10.101,3.287)--(10.116,3.286)--(10.130,3.285)--(10.144,3.285)--(10.159,3.284)--(10.173,3.283)%
  --(10.187,3.283)--(10.202,3.282)--(10.216,3.282)--(10.230,3.281)--(10.245,3.281)--(10.259,3.280)%
  --(10.273,3.280)--(10.288,3.279)--(10.302,3.279)--(10.317,3.279)--(10.331,3.278)--(10.345,3.278)%
  --(10.360,3.278)--(10.374,3.279)--(10.388,3.279)--(10.403,3.279)--(10.417,3.279)--(10.431,3.280)%
  --(10.446,3.280)--(10.460,3.281)--(10.474,3.281)--(10.489,3.282)--(10.503,3.283)--(10.517,3.284)%
  --(10.532,3.285)--(10.546,3.286)--(10.560,3.287)--(10.575,3.288)--(10.589,3.289)--(10.604,3.291)%
  --(10.618,3.292)--(10.632,3.294)--(10.647,3.295)--(10.661,3.297)--(10.675,3.299)--(10.690,3.300)%
  --(10.704,3.302)--(10.718,3.304)--(10.733,3.306)--(10.747,3.308)--(10.761,3.311)--(10.776,3.313)%
  --(10.790,3.315)--(10.804,3.318)--(10.819,3.320)--(10.833,3.323)--(10.847,3.325)--(10.862,3.328)%
  --(10.876,3.331)--(10.890,3.334)--(10.905,3.337)--(10.919,3.340)--(10.934,3.343)--(10.948,3.346)%
  --(10.962,3.349)--(10.977,3.352)--(10.991,3.356)--(11.005,3.359)--(11.020,3.363)--(11.034,3.366)%
  --(11.048,3.370)--(11.063,3.374)--(11.077,3.378)--(11.091,3.382)--(11.106,3.386)--(11.120,3.390)%
  --(11.134,3.394)--(11.149,3.398)--(11.163,3.402)--(11.177,3.406)--(11.192,3.411)--(11.206,3.415)%
  --(11.220,3.420)--(11.235,3.424)--(11.249,3.429)--(11.264,3.434)--(11.278,3.439)--(11.292,3.444)%
  --(11.307,3.449)--(11.321,3.454)--(11.335,3.459)--(11.350,3.465)--(11.364,3.470)--(11.378,3.476)%
  --(11.393,3.482)--(11.407,3.488)--(11.421,3.494)--(11.436,3.500)--(11.450,3.507)--(11.464,3.513)%
  --(11.479,3.520)--(11.493,3.527)--(11.507,3.534)--(11.522,3.541)--(11.536,3.548)--(11.551,3.556)%
  --(11.565,3.564)--(11.579,3.572)--(11.594,3.580)--(11.608,3.589)--(11.622,3.597)--(11.637,3.607)%
  --(11.651,3.616)--(11.665,3.625)--(11.680,3.635)--(11.694,3.646)--(11.708,3.656)--(11.723,3.668)%
  --(11.737,3.679)--(11.751,3.692)--(11.766,3.705)--(11.780,3.719)--(11.794,3.733)--(11.809,3.749)%
  --(11.823,3.765)--(11.837,3.784)--(11.852,3.803)--(11.866,3.826)--(11.881,3.849)--(11.895,3.879)%
  --(11.909,3.917)--(11.924,3.961)--(11.938,4.063);
\gpcolor{color=gp lt color border}
\node[gp node right] at (10.479,7.799) {\footnotesize \cite{Jacquin_2009}};
%\gpcolor{rgb color={0.337,0.706,0.914}}
\gpcolor{rgb color={0.580,0.000,0.827}}
\gpsetpointsize{4.00}
\gp3point{gp mark 3}{}{(1.566,4.743)}
\gp3point{gp mark 3}{}{(1.616,5.617)}
\gp3point{gp mark 3}{}{(1.735,6.603)}
\gp3point{gp mark 3}{}{(1.834,6.928)}
\gp3point{gp mark 3}{}{(1.934,7.083)}
\gp3point{gp mark 3}{}{(2.046,7.200)}
\gp3point{gp mark 3}{}{(2.140,7.136)}
\gp3point{gp mark 3}{}{(2.302,7.152)}
\gp3point{gp mark 3}{}{(2.570,7.269)}
\gp3point{gp mark 3}{}{(2.826,7.296)}
\gp3point{gp mark 3}{}{(3.175,7.269)}
\gp3point{gp mark 3}{}{(3.337,7.280)}
\gp3point{gp mark 3}{}{(3.586,7.296)}
\gp3point{gp mark 3}{}{(3.854,7.259)}
\gp3point{gp mark 3}{}{(4.116,7.253)}
\gp3point{gp mark 3}{}{(4.378,7.227)}
\gp3point{gp mark 3}{}{(4.634,7.243)}
\gp3point{gp mark 3}{}{(4.896,7.264)}
\gp3point{gp mark 3}{}{(5.157,7.195)}
\gp3point{gp mark 3}{}{(5.419,7.003)}
\gp3point{gp mark 3}{}{(5.681,6.747)}
\gp3point{gp mark 3}{}{(5.943,6.518)}
\gp3point{gp mark 3}{}{(6.205,6.326)}
\gp3point{gp mark 3}{}{(6.461,6.145)}
\gp3point{gp mark 3}{}{(6.722,5.878)}
\gp3point{gp mark 3}{}{(6.990,5.708)}
\gp3point{gp mark 3}{}{(7.252,5.617)}
\gp3point{gp mark 3}{}{(7.514,5.532)}
\gp3point{gp mark 3}{}{(7.770,5.446)}
\gp3point{gp mark 3}{}{(8.032,5.372)}
\gp3point{gp mark 3}{}{(8.393,5.254)}
\gp3point{gp mark 3}{}{(8.549,5.196)}
\gp3point{gp mark 3}{}{(8.817,5.121)}
\gp3point{gp mark 3}{}{(9.079,5.031)}
\gp3point{gp mark 3}{}{(9.341,4.951)}
\gp3point{gp mark 3}{}{(9.852,4.812)}
\gp3point{gp mark 3}{}{(10.376,4.663)}
\gp3point{gp mark 3}{}{(10.906,4.530)}
\gp3point{gp mark 3}{}{(11.430,4.428)}
\gp3point{gp mark 3}{}{(11.941,4.311)}
\gp3point{gp mark 3}{}{(11.436,3.682)}
\gp3point{gp mark 3}{}{(11.155,3.565)}
\gp3point{gp mark 3}{}{(10.900,3.485)}
\gp3point{gp mark 3}{}{(10.382,3.432)}
\gp3point{gp mark 3}{}{(9.858,3.469)}
\gp3point{gp mark 3}{}{(9.341,3.559)}
\gp3point{gp mark 3}{}{(8.811,3.693)}
\gp3point{gp mark 3}{}{(8.399,3.821)}
\gp3point{gp mark 3}{}{(7.789,4.002)}
\gp3point{gp mark 3}{}{(7.234,4.172)}
\gp3point{gp mark 3}{}{(6.729,4.322)}
\gp3point{gp mark 3}{}{(6.192,4.498)}
\gp3point{gp mark 3}{}{(5.694,4.615)}
\gp3point{gp mark 3}{}{(5.151,4.684)}
\gp3point{gp mark 3}{}{(4.634,4.689)}
\gp3point{gp mark 3}{}{(4.123,4.610)}
\gp3point{gp mark 3}{}{(3.586,4.466)}
\gp3point{gp mark 3}{}{(3.337,4.370)}
\gp3point{gp mark 3}{}{(3.169,4.295)}
\gp3point{gp mark 3}{}{(2.813,4.092)}
\gp3point{gp mark 3}{}{(2.551,3.869)}
\gp3point{gp mark 3}{}{(2.290,3.783)}
\gp3point{gp mark 3}{}{(2.034,3.346)}
\gp3point{gp mark 3}{}{(1.822,2.989)}
\gp3point{gp mark 3}{}{(1.604,2.360)}
\gp3point{gp mark 3}{}{(1.516,2.014)}
\gp3point{gp mark 3}{}{(11.121,7.799)}
\gpcolor{color=gp lt color border}
\draw[gp path] (1.504,8.441)--(1.504,0.985)--(11.947,0.985)--(11.947,8.441)--cycle;
%% coordinates of the plot area
\gpdefrectangularnode{gp plot 1}{\pgfpoint{1.504cm}{0.985cm}}{\pgfpoint{11.947cm}{8.441cm}}
\end{tikzpicture}

%% file: c_f-table.tex
\nprounddigits{5}\npfourdigitnosep
\setlength{\tabcolsep}{4pt}
\footnotesize\begin{tabular}{lrrrrrr}
  \toprule
     & \bfseries{L 11}                 & \bfseries{L 12}                 & \bfseries{L 13}                 & \bfseries{L 14}                 & \bfseries{extrapolated}                      & \bfseries{OSMP7} \\[0.5em]
     \cmidrule(lr){6-6} \cmidrule(lr){7-7}
  1  & \numprint{ 1.20409376273806e-1} & \numprint{ 1.20376707836993e-1} & \numprint{ 1.20423419916032e-1} & \numprint{ 1.20433948810790e-1} & \numprint{ 1.20424e-1}\nprounddigits{2}\,\numprint{\pm\,2.3620e-5}\nprounddigits{5} & \numprint{ 1.2264257810e-1} ( 0.06\%) \\
  2  & \numprint{-7.82459964699053e-1} & \numprint{-8.19449239371003e-1} & \numprint{-8.40837716964292e-1} & \numprint{-8.49506823460467e-1} & \numprint{-8.59232e-1}\nprounddigits{2}\,\numprint{\pm\,9.4290e-4}\nprounddigits{5} & \numprint{-8.1256842040e-1} ( 1.21\%) \\
  3  & \numprint{-2.49532929503623e-1} & \numprint{-2.49041832201642e-1} & \numprint{-2.47106424416843e-1} & \numprint{-2.45695194645541e-1} & \numprint{-2.45930e-1}\nprounddigits{2}\,\numprint{\pm\,8.5010e-4}\nprounddigits{5} & \numprint{-2.5975677490e-1} (-0.36\%) \\
  4  & \numprint{-4.82151799615986e-1} & \numprint{-4.95618270160579e-1} & \numprint{-4.99140059996885e-1} & \numprint{-5.00706620725300e-1} & \numprint{-5.04498e-1}\nprounddigits{2}\,\numprint{\pm\,1.4170e-3}\nprounddigits{5} & \numprint{-5.0692608640e-1} (-0.06\%) \\
  5  & \numprint{ 4.15340699489378e-2} & \numprint{ 5.22412926564442e-2} & \numprint{ 5.80261658369271e-2} & \numprint{ 6.23158571486318e-2} & \numprint{ 6.43990e-2}\nprounddigits{2}\,\numprint{\pm\,7.2400e-4}\nprounddigits{5} & \numprint{ 3.5785827640e-2} (-0.74\%) \\
  6  & \numprint{-1.64434411751677e-1} & \numprint{-1.65337490055917e-1} & \numprint{-1.62150911366611e-1} & \numprint{-1.61934774588354e-1} & \numprint{-1.62017e-1}\nprounddigits{2}\,\numprint{\pm\,1.2510e-3}\nprounddigits{5} & \numprint{-1.6944293210e-1} (-0.19\%) \\
  7  & \numprint{ 2.07532069130256e-1} & \numprint{ 2.11528865972370e-1} & \numprint{ 2.09586586298533e-1} & \numprint{ 2.08250541806358e-1} & \numprint{ 2.09821e-1}\nprounddigits{2}\,\numprint{\pm\,1.7750e-3}\nprounddigits{5} & \numprint{ 2.3291232300e-1} ( 0.60\%) \\
  8  & \numprint{ 5.19467160296020e-2} & \numprint{ 4.77461920212842e-2} & \numprint{ 6.44445067053358e-2} & \numprint{ 6.91407271043740e-2} & \numprint{ 6.74767e-2}\nprounddigits{2}\,\numprint{\pm\,7.0790e-3}\nprounddigits{5} & \numprint{ 3.1794084550e-2} (-0.93\%) \\
  9  & \numprint{ 1.96533003533094e-1} & \numprint{ 1.94637365454839e-1} & \numprint{ 2.01283287553661e-1} & \numprint{ 2.08307161662118e-1} & \numprint{ 2.05560e-1}\nprounddigits{2}\,\numprint{\pm\,4.3980e-3}\nprounddigits{5} & \numprint{ 1.2753921510e-1} (-2.03\%) \\
  0  & \numprint{-1.87467252057716e-1} & \numprint{-1.87351874450880e-1} & \numprint{-1.94432427608844e-1} & \numprint{-2.01263785163033e-1} & \numprint{-1.99171e-1}\nprounddigits{2}\,\numprint{\pm\,4.0690e-3}\nprounddigits{5} & \numprint{-1.5206761170e-1} ( 1.22\%) \\
  11 & \numprint{-9.16785240629972e-2} & \numprint{-8.55851919997993e-2} & \numprint{-9.30821059507293e-2} & \numprint{-9.82795909259543e-2} & \numprint{-9.50432e-2}\nprounddigits{2}\,\numprint{\pm\,4.8910e-3}\nprounddigits{5} & \numprint{-6.9926071170e-2} ( 0.65\%) \\
  12 & \numprint{-2.47024162911368e+0} & \numprint{-2.63696489133308e+0} & \numprint{-2.69976595051529e+0} & \numprint{-2.72151411009086e+0} & \numprint{-2.73529e+0}\nprounddigits{2}\,\numprint{\pm\,1.3110e-3}\nprounddigits{5} & \numprint{-2.7300000000e+0} ( 0.14\%) \\
  13 & \numprint{ 9.79492658414748e-1} & \numprint{ 1.05681319137703e+0} & \numprint{ 1.08511905324862e+0} & \numprint{ 1.10435458061785e+0} & \numprint{ 1.11416e+0}\nprounddigits{2}\,\numprint{\pm\,7.6330e-3}\nprounddigits{5} & \numprint{ 1.0075306400e+0} (-2.77\%) \\
  14 & \numprint{-3.90479051766157e-1} & \numprint{-4.27087726555475e-1} & \numprint{-4.54108954991645e-1} & \numprint{-4.63848845499747e-1} & \numprint{-4.73529e-1}\nprounddigits{2}\,\numprint{\pm\,2.9370e-3}\nprounddigits{5} & \numprint{-4.2079824830e-1} ( 1.37\%) \\
  15 & \numprint{-2.52896375159032e-1} & \numprint{-2.56742198600020e-1} & \numprint{-2.65622097487304e-1} & \numprint{-2.74443854783489e-1} & \numprint{-2.72934e-1}\nprounddigits{2}\,\numprint{\pm\,4.3840e-3}\nprounddigits{5} & \numprint{-1.6270483400e-1} ( 2.86\%) \\
  16 & \numprint{-6.91755652184979e-1} & \numprint{-7.17226550306309e-1} & \numprint{-7.35331871438397e-1} & \numprint{-7.44840743463044e-1} & \numprint{-7.50446e-1}\nprounddigits{2}\,\numprint{\pm\,2.4070e-3}\nprounddigits{5} & \numprint{-7.4094769290e-1} ( 0.25\%) \\
  17 & \numprint{ 6.05556782255291e-1} & \numprint{ 6.15460543676496e-1} & \numprint{ 6.06902790788074e-1} & \numprint{ 6.01852445314633e-1} & \numprint{ 6.06099e-1}\nprounddigits{2}\,\numprint{\pm\,5.9470e-3}\nprounddigits{5} & \numprint{ 7.5986322020e-1} ( 3.99\%) \\
  18 & \numprint{-2.10579423940998e-1} & \numprint{-2.27278772911497e-1} & \numprint{-2.42178209606447e-1} & \numprint{-2.46744876459952e-1} & \numprint{-2.51336e-1}\nprounddigits{2}\,\numprint{\pm\,2.2020e-3}\nprounddigits{5} & \numprint{-2.4175431820e-1} ( 0.25\%) \\
  19 & \numprint{ 1.17573236667707e-1} & \numprint{ 1.12928009846982e-1} & \numprint{ 1.06724235483406e-1} & \numprint{ 1.05138094260686e-1} & \numprint{ 1.03806e-1}\nprounddigits{2}\,\numprint{\pm\,1.3280e-3}\nprounddigits{5} & \numprint{ 1.1596183780e-1} ( 0.32\%) \\
  20 & \numprint{ 1.01598744303537e-1} & \numprint{ 9.45446423652579e-2} & \numprint{ 8.63467186532705e-2} & \numprint{ 8.39869486035175e-2} & \numprint{ 8.20506e-2}\nprounddigits{2}\,\numprint{\pm\,1.6100e-3}\nprounddigits{5} & \numprint{ 1.0164488220e-1} ( 0.51\%) \\
  21 & \numprint{ 1.71732930733951e-1} & \numprint{ 1.78969070618106e-1} & \numprint{ 1.79481668940860e-1} & \numprint{ 1.79635243133770e-1} & \numprint{ 1.81822e-1}\nprounddigits{2}\,\numprint{\pm\,1.3460e-3}\nprounddigits{5} & \numprint{ 1.8956013490e-1} ( 0.20\%) \\
  22 & \numprint{-1.01706817126602e+0} & \numprint{-1.05988119974777e+0} & \numprint{-1.07840001568767e+0} & \numprint{-1.08323638772591e+0} & \numprint{-1.09598e+0}\nprounddigits{2}\,\numprint{\pm\,2.8240e-3}\nprounddigits{5} & \numprint{-1.1288459470e+0} (-0.85\%) \\
  23 & \numprint{ 5.40265975332290e-1} & \numprint{ 5.37249861091246e-1} & \numprint{ 5.47784814838180e-1} & \numprint{ 5.56150529754043e-1} & \numprint{ 5.52875e-1}\nprounddigits{2}\,\numprint{\pm\,6.0380e-3}\nprounddigits{5} & \numprint{ 5.8692303470e-1} ( 0.88\%) \\
  24 & \numprint{ 3.97091275806076e-1} & \numprint{ 3.94569268322583e-1} & \numprint{ 3.96378627403379e-1} & \numprint{ 3.95293951986831e-1} & \numprint{ 3.95108e-1}\nprounddigits{2}\,\numprint{\pm\,9.9770e-4}\nprounddigits{5} & \numprint{ 3.9236697390e-1} (-0.07\%) \\
  25 & \numprint{ 4.92919306146496e-1} & \numprint{ 5.04619108014532e-1} & \numprint{ 5.05090241792478e-1} & \numprint{ 5.09868006872447e-1} & \numprint{ 5.11600e-1}\nprounddigits{2}\,\numprint{\pm\,1.8070e-3}\nprounddigits{5} & \numprint{ 5.0768316650e-1} (-0.10\%) \\
  26 & \numprint{-5.70346467990356e-2} & \numprint{-6.10185283985007e-2} & \numprint{-7.03092559361799e-2} & \numprint{-7.31802766102367e-2} & \numprint{-7.40772e-2}\nprounddigits{2}\,\numprint{\pm\,2.6750e-3}\nprounddigits{5} & \numprint{-5.8938621520e-2} ( 0.39\%) \\
  27 & \numprint{-3.29090336110006e-2} & \numprint{-3.10078816676207e-2} & \numprint{-3.30164868046366e-2} & \numprint{-3.36559556112350e-2} & \numprint{-3.30018e-2}\nprounddigits{2}\,\numprint{\pm\,1.1620e-3}\nprounddigits{5} & \numprint{-3.2502002720e-2} ( 0.01\%) \\
  28 & \numprint{-6.50035829936763e-2} & \numprint{-6.82720910073972e-2} & \numprint{-6.93864559323882e-2} & \numprint{-6.95723650248219e-2} & \numprint{-7.05913e-2}\nprounddigits{2}\,\numprint{\pm\,3.4370e-4}\nprounddigits{5} & \numprint{-7.8727676390e-2} (-0.21\%) \\
  29 & \numprint{ 5.78878104357922e-2} & \numprint{ 5.68206397547076e-2} & \numprint{ 5.51981499845772e-2} & \numprint{ 5.57125083788555e-2} & \numprint{ 5.50454e-2}\nprounddigits{2}\,\numprint{\pm\,4.4230e-4}\nprounddigits{5} & \numprint{ 4.6848423000e-2} (-0.21\%) \\
  30 & \numprint{-1.74615003439992e-1} & \numprint{-1.91536763614145e-1} & \numprint{-1.95747716640579e-1} & \numprint{-1.97235861049769e-1} & \numprint{-2.02170e-1}\nprounddigits{2}\,\numprint{\pm\,1.9810e-3}\nprounddigits{5} & \numprint{-2.0642430110e-1} (-0.11\%) \\
  31 & \numprint{ 1.56319466196051e-1} & \numprint{ 1.78150280655956e-1} & \numprint{ 1.85909546643379e-1} & \numprint{ 1.89010574694147e-1} & \numprint{ 1.95116e-1}\nprounddigits{2}\,\numprint{\pm\,1.5720e-3}\nprounddigits{5} & \numprint{ 2.0723223880e-1} ( 0.31\%) \\
  \bottomrule
\end{tabular}